\documentclass{mcom-l}

\usepackage{amsfonts, amssymb, epsfig,subfigure}
\usepackage{mathrsfs}
\usepackage{stmaryrd}
\usepackage{graphicx,amsmath}
\usepackage{latexsym}
\usepackage{verbatim}
\usepackage{epsfig}
\usepackage{amsfonts}
\usepackage{verbatim}
\usepackage{latexsym}
\usepackage{cite}
\usepackage{amsmath,amssymb,color}

\newtheorem{theorem}{Theorem}[section]
\newtheorem{lemma}[theorem]{Lemma}
\newtheorem{corollary}[theorem]{Corollary}

\theoremstyle{definition}

\newtheorem{example}[theorem]{Example}

\theoremstyle{remark}
\newtheorem{remark}[theorem]{Remark}

\numberwithin{equation}{section}

\begin{document}
\allowdisplaybreaks
\newcommand{\E}{\mathbb{E}}
\newcommand{\PP}{\mathbb{P}}
\newcommand{\RR}{\mathbb{R}}
\newcommand{\I}{\mathbb{I}}
\newcommand{\dis}{\displaystyle}
\newtheorem{assumption}{Assumption}
\def\t{\triangle}  \def\ts{\triangle^*}
\def\ra{\rightarrow}
\def\a{\alpha} \def\b{\beta} \def\d{\delta}\def\g{\gamma}
\def\e{\varepsilon} \def\z{\zeta} \def\y{\eta} \def\o{\theta}
\def\vo{\vartheta} \def\k{\kappa} \def\l{\lambda} \def\m{\mu} \def\n{\nu}
\def\x{\xi}  \def\r{\rho} \def\s{\sigma}
\def\p{\phi} \def\f{\varphi}   \def\w{\omega}
\def\q{\surd} \def\i{\bot} \def\h{\forall} \def\j{\emptyset}
 \def\si{s_{ij}}
\def\be{\beta} \def\de{\delta} \def\up{\upsilon} \def\eq{\equiv}
\def\ve{\vee} \def\we{\wedge}
\def\SS{{\cal S}}
\def\F{{\mathcal F}} \def\Le{{\cal L}}
\def\X{\Xi} \def\S{\Sigma}
\def\lf{\left} \def\rt{\right}
\def\trace{\hbox{\rm trace}}
\def\diag{\hbox{\rm diag}}
\def\nn{\nonumber}
\def\be{\begin{equation}}
\def\ee{\end{equation}}
\def\la{\label}
%
%
\title[Convergence and stability of explicit numerical schemes for SDEs]{Strong convergence and asymptotic stability of explicit numerical schemes for nonlinear stochastic differential equations}


\author{Xiaoyue Li}
\address{School of Mathematics and Statistics,
Northeast Normal University, Changchun, Jilin, 130024, China.}
\curraddr{}
\email{lixy209@nenu.edu.cn}
\thanks{The research of the first author was  supported in part by the National Natural Science Foundation of Chain (11171056 and 11471071), the  Natural Science Foundation of Jilin Province (20170101044JC), and the  Education Department of Jilin Province (JJKH20170904KJ)}

\author{Xuerong Mao}
\address{Department of Mathematics and Statistics,
University of Strathclyde, Glasgow G1 1XH, U.K.}
\curraddr{}
\email{x.mao@strath.ac.uk}
\thanks{The research of the second author was supported in part by   the Royal Society (WM160014, Royal Society Wolfson Research Merit Award), the  Royal Society and the Newton Fund (NA160317, Royal Society-Newton Advanced Fellowship), the  EPSRC (EP/K503174/1)}

\author{Hongfu Yang}
\address{School of Mathematics and Statistics,
Northeast Normal University, Changchun, Jilin, 130024, China.}
\curraddr{}
\email{yanghf783@nenu.edu.cn}

\date{}

\dedicatory{} 
\begin{abstract}
In this article we introduce several kinds of easily implementable explicit schemes, which are amenable to Khasminski's techniques and are particularly suitable for highly nonlinear stochastic differential equations (SDEs).
 We show that without additional restriction conditions except  those which guarantee the exact solutions possess their boundedness in expectation  with respect to certain Lyapunov functions,
 the numerical solutions converge  strongly to the exact solutions in finite-time.
 Moreover, based on  the nonnegative  semimartingale convergence theorem,
 positive results about the ability of explicit numerical approximation to reproduce the   well-known LaSalle-type theorem of SDEs are proved here,   from which we deduce the asymptotic stability of numerical solutions.
 Some examples and simulations are provided to support the theoretical results and to demonstrate the validity of the approach.
\end{abstract}

\maketitle

%
\section{Introduction}
\label{intro}
In 1949, It\^{o} established the well known stochastic calculus. Since then, the theory of stochastic differential equations (SDEs) has been developed  quickly. Particularly, the Lyapunov method has been used to deal with the dynamical behaviors of SDEs in both finite and infinite intervals by  many authors, and here we only mention Arnold \cite{1}, Friedman \cite{2}, Khasminskii \cite{Khasminski80}, Kushner \cite{3},  Mao \cite{Mao08,Mao2002}, and Yin and Zhu \cite{yz09}.
These studies also showed  that  the integrability and the stability of solutions to SDEs can be obtained via stochastic Lyapunov analysis. However, these properties are not necessarily inherited by standard numerical approximations.
 The main goal of this article is to develop the approximations techniques of nonlinear SDEs that are flexible enough for the stochastic Lyapunov method. Despite of the lack of a discrete version of  It\^{o}'s formula, asymptotic and qualitative properties in discrete version are obtained by our explicit schemes.
 In particular,  this article is to construct
 easily implementable numerical solutions and prove that they converge to the true solution of the underlying SDEs. In addition to obtaining the $V$-integrability (see, \cite[Definition 2.3]{Lukasz17}) and  convergence rate, we consider the  explicit numerical approximations to reproduce the well-known LaSalle-type theorem of SDEs, from which we deduce the asymptotic stability of numerical solutions.

Explicit Euler-Maruyama (EM) schemes are most popular for approximating the solutions of SDEs
under global Lipschitz continuously (see, e.g., \cite{Kloeden, Mao08, Hi02}).
  However, the coefficients of many important SDE models are only locally Lipschitz and superlinear (see, e.g., \cite{16, Hutzenthaler15} and the references therein).
 If the global Lipschitz condition does not hold for either of the coefficients,
Hutzenthaler, Jentzen and Kloeden \cite{16} showed that the explicit Euler scheme may
have unbounded moments, and consequently the classical Euler scheme may fail to
converge strongly.
 Implicit methods were developed to approximate the solutions of these SDEs. Higham et al. \cite{Hi02}
showed that the backward EM schemes converge if the diffusion coefficients are globally Lipschitz
while the drift coefficients satisfy a one-sided Lipschitz condition. Therefore,  the methods with implicit structure are often employed as the alternatives, more details on the implicit methods can be found in \cite{Kloeden, Saito1993, Lukasz11, Mao08}.
 Nevertheless, additional computational efforts are required for its implementation since the solution of an algebraic equation has to be found before each iteration.

Due to the advantages of explicit schemes (e.g., simple structure and cheap computational cost), a few modified EM methods have been developed for nonlinear SDEs including the tamed EM method \cite{Hutzenthaler12, Hutzenthaler15,Sabanis13,Sabanis16}, the tamed Milstein method \cite{Wang13}, the stopped EM method \cite{LiuW13} and the truncated EM method \cite{Mao20152,Guo2017,Lan,Li18}. These modified EM methods have shown their advantages to approximate the solutions of nonlinear SDEs  in any finite time interval.
  In all of the above mentioned results from the literature, the moment bounds relies on simple Lyapunov functions such as $|\cdot|^p$ with $p\geq 2$.
 $V$-integrability results on more general Lyapunov-type functions
 of some specific modified (tamed)  EM schemes were established  (see, e.g.   \cite{Lukasz17,Hutzenthaler15,Hutzenthaler17}).
  For examples,   Hutzenthaler and Jentzen \cite{Hutzenthaler15}  showed some criteria for moment bounds
under
priori estimates (see, \cite[Proposition 2.7]{Hutzenthaler15}),
these moment bounds are then used to prove strong convergence of the proposed methods.
  Szpruch and Zhang \cite{Lukasz17}
investigated the $V$-integrability  of explicit numerical approximations of SDEs with additional restriction conditions on  coefficients, and provided the (1/2)-order rate of convergence in the strong mean square sense for the projected schemes. It is observed that  to construct the appropriate scheme to inherit the integrability is  challenging.
Motivated by this, we construct easily implementable
explicit EM schemes for nonlinear SDEs and establish their $V$-integrability by requiring only that the drift and diffusion coefficients are locally Lipschitz and satisfy a structure condition \eqref{1.4} for the $V$-integrability of the exact solution.
  We also demonstrate the convergence of the algorithm under weak  conditions. 
   Then under slightly stronger conditions,  we prove the strong convergence rate for the explicit schemes, with respect to a larger class of Lyapunov functions.

  On the other hand,  long-time behaviors of SDEs are also the hot concerns in stochastic processes, systems theory, control  and optimization (see, e.g., the monographs \cite{Mao08,Khasminski80, yz09} and the references therein).
 So far, the dynamical properties of SDEs are investigated deeply including stochastic stability (see, e.g., \cite{Khasminski80, Mao08, Lukasz17}),  ergodicity (see, e.g., \cite{Khasminski80,Bao1,Mattingly,Martin}) and so on. Here we focus on the asymptotic stability.  Although the finite-time convergence is one of the fundamental concerns,  how to preserve the asymptotic stability of the underlying SDEs  is   significant and challenging.
 Recently, considerable effort has been made in this direction (mainly for implicit schemes) in \cite{Higham3, Higham4, Higham5, S5, SS1},
 where  the diffusion coefficients of SDEs are always required to be global Lipschitz continuous.
 Accommodating the many applications although these systems are more realistically addressing the demands, the nontraditional setup make the discrete approximations of the asymptotic stability  for    nonlinear SDEs  more difficult.
 Thus,  in order to close the  gap, a few modified EM methods have been developed
to approximate the asymptotic  stability for nonlinear SDEs. For instance,
 Guo et  al. \cite{Guo2017}
 showed that the partially truncated EM method can preserve the mean square exponential stability of the underlying SDEs.
 Liu and Mao  \cite{S7} made use of the EM method with random variable stepsize
to reproduce the almost sure stability of  the underlying SDEs. Szpruch and Zhang \cite{Lukasz17}
 established  the asymptotic stability properties for
tamed  EM schemes and projected schemes, which  admitting
certain Lyapunov functions.   
    For the further development of numerical schemes for SDEs, we refer  readers to \cite{Lukasz17,Hutzenthaler17,Li18}, for example, and the references therein.   To the best of our knowledge,
much research on  numerical stability 
relies on simple Lyapunov functions such as $|\cdot|^2$,  
with the exception of  \cite{Lukasz17,Li18}.
 Here our aim is to  handle more general cases by new schems, particularly Lyapunov functions of the form  $|\cdot|^p$ for any $p>0$ or  polynomials of the more general form.


In this paper,  borrowing  the truncation idea from \cite{Mao20152,Li18} and using novel approximation technique, we  construct  a new explicit scheme  that preserve the $V$-integrability  of SDEs with respect to a larger class of Lyapunov functions, and derive strong convergence  result in a finite  time interval.
 Then  we go further to improve the scheme according to the structure  condition of the LaSalle-type theorem
  such that it is easily  implementable for approximating the underlying stability  of the SDEs, admitting a large class of Lyapunov functions.
 The schemes proposed in this paper  are obviously different from those  of
   \cite{Guo2017,S7,Mao20152,Lukasz17}.
  More precisely, the numerical solutions at the grid points are modified before each iteration according to the growth rate of the drift and diffusion coefficients such that the numerical solutions keep  the underlying excellent properties of the exact solutions of SDEs.
 Our main contributions are as follows:
\begin{itemize}

\item[$\bullet$] We construct an easily implementable scheme for the SDEs with only local Lipschitz drift and diffusion coefficients, which
   numerical solutions preserve the $V$-integrability of the exact solution almost perfectly with respect to a larger class of Lyapunov functions, and
establish finite-time strong convergence results.

\item[$\bullet$] We reconstruct a more precise explicit scheme  to reproduce the   LaSalle-type results in stochastic version 
for a large class of auxiliary Lyapunov functions.  Especially, the explicit scheme inherit the exponential stability of the exact solution well.

\item[$\bullet$] Without extra restrictions, the numerical solutions of the explicit schemes stay in step of asymptotic stability of the exact solution  in respect to certain Lyapunov functions.
 Compared with the existing results on the asymptotic stability \cite{Guo2017,S7,Li18,Lukasz17},  the range of the auxiliary  Lyapunov functions is extended by using our explicit scheme.
\end{itemize}

The rest of the paper is organized as follows. Section \ref{n-p} begins with notations and preliminaries on the properties of the  exact solution.
 Section \ref{3s-c}   constructs an explicit scheme, and yields the strong convergence and integrability   in a finite time interval, with respect to a larger class of Lyapunov functions. Section \ref{c-r} provides the rate of convergence.
 Section \ref{5s-b} reconstructs a more precise explicit scheme to  reproduce the  LaSalle-type results in stochastic version. 
  Section \ref{exa} presents a couple of examples and simulations  to illustrate
our results. Section \ref{concluding}  gives some futher remarks to conclude the paper.

\section{Notations and preliminaries}\label{n-p}
Throughout this paper,  we  use the following notations. Let $d$,  $m$ and $n$ denote finite positive integers,  $|\cdot|$ denote the Euclidean norm in $\RR^d:=\RR^{d\times 1}$ and the trace norm in $\RR^{d\times m}$,  $\langle\cdot,\cdot\rangle$ stand for the dot product (usual Euclidean scalar product) on $\mathbb{R}^d$.
 For any $a, b\in \mathbb{R}$,
 $a\vee b:=\max\{a,b\}$, and $a\wedge b:=\min\{a,b\}$.  If $\mathbb{D}$ is a set, its indicator function is denoted by $I_{\mathbb{D}}$, namely $I_{\mathbb{D}}(x)=1$ if $x\in \mathbb{D}$ and $0$ otherwise.
 Let $( \Omega, \mathcal{F}, \PP )$ be a complete  probability space, and $\mathbb{E}$ denotes the expectation corresponding to $\PP$. Let~$B(t)=\big(B^{(1)}(t), \ldots, B^{(m)}(t)\big)^T$ be an $m$-dimensional Brownian motion defined on this probability space.   Suppose $\{\mathcal{F}_{t}\} _{t \geq 0}$ is a filtration defined on this probability space satisfying the usual conditions (i.e., it is right continuous and $\mathcal{F}_0$ contains all $\mathbb{P}$-null sets) such that $B(t)$ is $\mathcal{F}_{t}  $ adapted. Let  $\mathbb{R}_+:=(0,\infty) $,  $\bar{\mathbb{R}}_+:=[0,\infty) $ and $\mathbf{0}$ denote a null matrix whose dimension may change in different appearances.
Also let $C_i$ and $C$ denote two generic positive real constants respectively, whose value may change in different appearances, where $C_i$ is  dependent on $i$ and $C$ is independent of  $i$.
 In this paper, we consider the  $d$-dimensional stochastic differential equation (SDE)
\be\label{e1}
\mathrm{d}X(t)=f(  X(t))\mathrm{d}t +g(  X(t))\mathrm{d}B(t)
\ee
with an initial value $X(0)=x_0\in \RR^d$, 
where the drift and diffusion terms
$$
  f:  \RR^d    \rightarrow \RR^d, \qquad \quad g: \RR^d  \rightarrow \RR^{d\times m},
$$
 are   local Lipschitz continuous,  this is,
for any $N>0$  there exists a positive   constant $C_N$ such that, for any  $x, y\in \RR^d  $ with  $|x |\vee |y|\leq N$,
  $$|f(x)-f(y)|\vee |g(x)-g(y)|\leq C_N |x-y|.$$

Let
  $\mathcal{C}^{p} (\RR^d; \bar{\mathbb{R}}_+)$ denote the family of all nonnegative functions
 $V(x)$ on $\RR^d$ which are continuously $p$th differentiable in $x$. Let
  $\mathcal{C}_\infty^{p}(\RR^d; \bar{\mathbb{R}}_+)$ denote the family of all  functions
 $V(\cdot)\in \mathcal{C}^{p} (\RR^d; \bar{\mathbb{R}}_+)$ with the property $\lim_{|x|\rightarrow \infty} V(x)=\infty$.
 For convenience, we cite the following notations introduced by \cite[p.617]{Evans}.
 A vector of the form  $\alpha=(\alpha_1,\ldots, \alpha_d)$, where each component $\alpha_i$ is nonnegative integer, is called a {\it multiindex} of order
$
|\alpha|:=\alpha_1+\cdots+\alpha_d.
$
For any
 multiindex   $\alpha=(\alpha_1,\ldots, \alpha_d)$ and any vector $x \in \mathbb{R}^d$,  we set as usual
$   \alpha!:=\alpha_1!\cdots\alpha_d!$,  $x^\alpha:=x_1^{\alpha_1}\cdots x_d^{\alpha_d}.$  If $\beta=(\beta_1,\ldots,\beta_d)$ is also a multiindex, then
$\alpha\geq \beta $
means each component $\alpha_{i}\geq \beta_{i}$ for any $1\leq i\leq d$, multiindex
$\alpha-\beta:=(\alpha_1-\beta_1,\ldots, \alpha_d-\beta_d)$, and multiindex   $\alpha+\beta:=(\alpha_1+\beta_1,\ldots, \alpha_d+\beta_d)$.
  For each $V(x)\in \mathcal{C}^{p} (\RR^d; \bar{\mathbb{R}}_+)$ and    a multiindex $\alpha$ with $|\alpha|\leq p$, define
$$
D^{\alpha}V(x):=\frac{\partial^{|\alpha|}V(x)}{\partial x_1^{\alpha_1}\partial x_2^{\alpha_2}\cdots\partial x_d^{\alpha_d}}.
$$
For any nonnegative integer
$n\leq  p$,  define
$
D^{(n)}V(x):=\big\{D^{\alpha}V(x)\big||\alpha|=n\big\},
$
the set of all partial derivatives with $n$th order.
Assigning some ordering to the various partial derivatives, we can also regard
$D^{(n)}V(x)$ as a point in $\mathbb{R}^{d^n}$ and define
 $
|D^{(n)}V(x)|=\Big(\sum_{|\alpha|=n}|D^{\alpha}V(x)|^2\Big)^{\frac{1}{2}}.
$
Especially,  if $n=1$, we regard the elements of $D^{(1)}V(x)$ as being arranged in a vector
$$
D^{(1)}V(x)=\left(\frac{\partial V(x)}{\partial x_1},
 \ldots,\frac{\partial V(x)}{\partial x_d}\right), \quad  \mathrm{and}\quad  |D^{(1)}V(x)|=\bigg[\sum_{i=1}^d\Big(\frac{\partial V(x)}{\partial x_i}\Big)^2\bigg]^{\frac{1}{2}};
$$
If $n=2$, we regard the elements of $D^{(2)}V(x)$ as being arranged in a matrix
$$D^{(2)}V(x) =\left(\frac{\partial^2
V(x)}{\partial x_i  \partial x_j}\right)_{d\times d},\quad  \mathrm{and}\quad  |D^{(2)}V(x)|=\bigg[\sum_{i,j=1}^d\Big(\frac{\partial^2 V(x)}{\partial x_i\partial x_j}\Big)^2\bigg]^{\frac{1}{2}}.$$

  For each $V(x)\in \mathcal{C}^{2} (\RR^d; \bar{\mathbb{R}}_+)$, define an
operator ${\mathcal{L}}V$ from
 $  \RR^d$ to $  \RR$ by
\begin{align}\la{1e1.2}
{\mathcal{L}} V(x) =& \langle  D^{(1)}V(x), f(x)\rangle
 +\frac{1}{2} \hbox{tr}\Big[g^T(x)  D^{(2)}V(x) g(x)\Big].
\end{align}
As \cite{Hutzenthaler15, Lukasz17}, for a pair of intergers  $p\in [2, +\infty)$ and  $1/\delta_p\in [p, +\infty)$, define
\begin{align}\la{1*}
  \mathcal{V}^{p}_{\delta_p}:=\Big\{V\in \mathcal{C}^{p}_\infty(\mathbb{R}^d;\bar{\mathbb{R}}_{+})\Big|&  \exists~c>0~~ \mathrm{s.t.}~~\nn\\
   &~~ |D^{(n)}  V(\cdot)|\leq c \big(1+V(\cdot)\big)^{1- n \delta_{p}},~ n=1,2,\ldots, p\Big\} .
\end{align}
Note that  many frequently-used functions belong to the set $\mathcal{V}^{p}_{\delta_p}$.
Hence  one has a lot of opportunities  to the underlying given SDE (see \cite{Hutzenthaler15} for more details).

Now we prepare the regularity and $V$-integrability of the exact solution.
%
\begin{theorem}\la{th1}
Assume there is a function $V\in \mathcal{C}_\infty^2(\mathbb{R}^d; \bar{\mathbb{R}}_+)$ and a pair of  positive constants $\rho$ and
$\lambda$ such that
\begin{align}\la{1.4}
  \mathcal{L} (1+ V(x))^{\rho}  &= \frac{\rho}{2}\big(1+V(x)\big)^{\rho-2}\Big[2\big(1+V(x)\big)
 \mathcal{L}V(x) +(\rho-1)|D^{(1)}V(x) g( x)|^2\Big] \nn\\
 &\leq   \lambda \big[1+V^{\rho}(x)\big],\qquad\qquad \forall x\in \mathbb{R}^d.
\end{align}
Then the SDE \eqref{e1} with any initial value $x_0\in \mathbb{R}^d$ has a unique regular solution $X(t)$ satisfying
\begin{align}\la{1.5}
\sup_{0\leq t\leq T}\E V^{\rho}(X(t))
 \leq  C,\qquad \forall T>0.
\end{align}
\end{theorem}

\begin{remark}
In fact  this is the Khasminskii test as if we set $\bar V(x) = (1+ V(x))^{\rho}$.
In other words, the conditions of Theorem \ref{th1}  is an alternative to Khasminskii's condition that there exists a positive constant $\bar{\lambda}$ such that
${\mathcal{L}}\bar V(x) \leq \bar{\lambda}(1+\bar V(x))$, see \cite[Theorem 3.5, p.75]{Khasminski80}.
\end{remark}

\section{$V$-integrability  and strong convergence}\label{3s-c}
In this section, we aim  to construct an easily implementable explicit scheme and show that its numerical solution
 converges strongly to the exact solution of SDE \eqref{e1}. 
  If it can efficiently prevent the diffusion term from producing extra-ordinary large value, the numerical method will keep the properties of  the exact solution  by using  the Taylor expansion.
Thus we define the explicit scheme by the appropriate truncation map.

Let $V\in   \mathcal{V}^{4}_{\delta_4}$ for some  integer
 $1/\delta_4\in [4, +\infty)$.
 To define  appropriate numerical solutions, we choose a strictly increasing continuous  function  $\f: \bar{\RR}_+\rightarrow \bar{\RR}_+$ such that $\f(u)\rightarrow \infty$
as $u\rightarrow \infty$ and
\be\la{e21}
\sup_{|x|\leq u}  \left(\frac{ |f (  x)|}{\big(1+V(x)\big)^{\delta_4}}\vee\frac{| g(  x)|^2 }{\big(1+V(x)\big)^{2\delta_4}}\right)\leq \f(u),\qquad\forall~ u\geq 1.
\ee
 We note that $\f$ is well defined since $f(x)$ and $g(x)$ are locally bounded
 in $x$.
 Denote by $\f^{-1} $ the inverse function of $\f$, obviously $\f^{-1}: [\f(1),\infty)\rightarrow \bar{\RR}_+ $ is  a strictly increasing continuous function.
We also choose  a pair of positive
constants $\t^*\in (0,1]$ and  $K$
 such that
 $
K(\t^{*})^{-\theta}\geq \f(|x_0|\vee 1)
$
  holds for some $0<\theta\leq 1/2$, where  $K$  is independent of the iteration order  $k$ and the time stepsize $\t$. For  the given $\t\in (0, \t^*]$,   define a truncation mapping $\pi_{\t}:\RR^d\ra \RR^d　$ by
 \begin{align*}
\pi_\t(x)= \Big(|x|\wedge \f^{-1}\big(K \t^{-\theta}\big)\Big) \frac{x}{|x|},
\end{align*}
where we let  $\frac{x}{|x|}=\mathbf{0}$  as $x=\mathbf{0}\in \mathbb{R}^d$.
Obviously,
 \begin{align*}
& |f( \pi_{\t}(x))|
\leq  \f(\f^{-1}(K \t^{-\theta})) \big(1+V(\pi_{\t}(x))\big)^{\delta_4} =     K \t^{-\theta}\big(1+V(\pi_{\t}(x))\big)^{\delta_4},\nn\\
&|g( \pi_{\t}(x))|^2
\leq  \f(\f^{-1}(K \t^{-\theta}))\big(1+V(\pi_{\t}(x))\big)^{ 2 \delta_4}  =     K \t^{-\theta}\big(1+V(\pi_{\t}(x))\big)^{ 2 \delta_4}
 \end{align*}
for any   $x\in \RR^d.$

  Next we propose our numerical method  to approximate the exact solution of the SDE (\ref{e1}).    For any given stepsize $\t\in (0, \t^*]$, define
 \begin{align}\la{Y_0}
\left\{
\begin{array}{ll}
Y_0=x_0,&\\
\tilde{Y}_{k+1}=Y_k+f(Y_k)\t +g(Y_k)\t B_k,& \\
Y_{k+1}=\pi_\t(\tilde{Y}_{k+1}),&
\end{array}
\right.
\end{align}
for any integer $k\geq 0$, where $t_k=k \t$ and  $\t B_k=(\t B^{(1)}_k,\ldots, \t B^{(m)}_k)^{T}=B(t_{k+1})-B(t_{k})$. The explicit method   \eqref{Y_0}  is called  the {\it $V$-truncated EM scheme} which  modifies the values of nodes before each iteration avoiding the extra-ordinary large deviations. One further observes that
\begin{align}\la{Y_01}
|f (Y_k) |  \leq K \t^{-\theta}\big(1+V(Y_k)\big)^{\delta_4},\qquad |g( Y_k)|^2 \leq  K \t^{-\theta}  \big(1+V(Y_k)\big)^{2\delta_4}.
\end{align}
To obtain the continuous-time approximations,   define $\tilde{Y}(t)$ and $Y(t)$   by
 \begin{align*}
\tilde{Y}(t):=\tilde{Y}_{k},\qquad  Y(t):=Y_k,\quad \forall t\in[t_k,  t_{k+1}).
\end{align*}
  We write $\mathbb{E}_k[\cdot]:=\mathbb{E} [\cdot|\mathcal{F}_{t_k} ]$ for simplicity. The following  lemmas will play their important roles  in the  proof  of the $V$-integrability of the numerical solutions.

\begin{lemma}\la{le1.3}
If $V\in   \mathcal{V}^{4}_{\delta_4}$ for some integer $1/\delta_4\in [4, +\infty)$,
  then the $V$-truncated EM scheme (\ref{Y_0}) has the property that
\begin{align}\la{neq:9}
\sum_{|\alpha|=1}^{3} \frac{ D^{\alpha} V (Y_k) }{\alpha!} \big(\tilde{Y}_{k+1}-Y_k\big)^{\alpha}
\leq\mathcal{L}V(Y_{k})\t+\mathcal{R}^{\t}V(Y_k)+\sum_{i=1}^{3}\mathcal{S}_{i}^{\t}V(Y_k),
\end{align}
where
 $\mathcal{S}_{1}^{\t}V(\cdot), \mathcal{S}_{2}^{\t}V(\cdot)$ and $\mathcal{S}_{3}^{\t}V(\cdot)$  are defined by
\eqref{fA.1}, \eqref{fA.3}, \eqref{fA.5}, respectively,
\begin{align*}
\mathcal{R}^{\t}V(Y_k):=C\sum_{i=2}^3\sum_{j=0}^{i-2}|f( Y_k)|^{i-2j}|g(Y_k)|^{2j}|D^{(i)} V (Y_k)|\t^{i-j}.
\end{align*}
We also have $
\mathbb{E}_k\big[\mathcal{S}_{i}^{\t}V(Y_k)\big]=0$ for $i=1,2,3$.
\end{lemma}

\begin{lemma}\la{Z:1}
If $V\in   \mathcal{V}^{4}_{\delta_4}$ for some integer $1/\delta_4\in [4, +\infty)$,
we then have
 \begin{align}\la{lyhf1}
\big| \mathcal{L}V(Y_{k})\big| \leq 2cK\big(1+V(Y_{k})\big) \t^{-\theta},\qquad
  \mathcal{R}^{\t}V(Y_k)
\leq C\big(1+V(Y_{k})\big)\t^{2(1-\theta)},
\end{align}
where  $\mathcal{L}V(\cdot)$ is defined by \eqref{1e1.2}.  Moreover, we  have
\begin{align*}
\mathbb{E}_k\big[|\mathcal{S}_{1}^{\t}V(Y_k)|^2\big] = | D^{(1)}V(Y_{k}) g( Y_{k})|^2\t,\ \
\mathbb{E}_k\big[|\mathcal{S}_{i}^{\t}V(Y_k)|^2\big] \leq C\big(1+V(Y_{k})\big)^2\t^{1-\theta}
\end{align*}
for $i=1, 2, 3$ and
\begin{align*}
  \mathbb{E}_k\big[\mathcal{S}_{1}^{\t}V(Y_k)\mathcal{S}_{j}^{\t}V(Y_k)\big] \geq- C \big(1+V(Y_{k})\big)^2\t^{2(1-\theta)}\quad \mathrm{for}\ \ j= 2, 3.
\end{align*}

  \end{lemma}
The proofs of both  lemmas above can be found in  Appendix A.
 Let us begin to establish the criterion on
  the $V$-integrability  of  the scheme \eqref{Y_0}. It makes use of Lemma \ref{le1.3} and Lemma \ref{Z:1} above.

  \begin{theorem}\la{co3.1}   Let the conditions of Theorem \ref{th1} hold.  Assume moreover that the function
  $V\in  \mathcal{V}^{4}_{\delta_4}$ for some integer $1/\delta_4\in [4, +\infty)$ and has the property
$$
 V(\epsilon x)\leq V(x) \qquad   \forall\, x \in \mathbb{R}^d,\ \ 0<\epsilon\leq 1.$$
Then  the truncation scheme defined by \eqref{Y_0} has the property
\begin{align}\la{3.22y}
\sup_{\t\in(0,\t^*]}\sup_{0\leq k \t\leq T} \E V^{\rho}(Y_{k})
 \leq  C, \qquad\forall T>0.
\end{align}
\end{theorem}
\begin{proof}
Due to  $V \in \mathcal{V}^{4}_{\delta_4}$, using the Taylor formula with integral remainder term, 
\begin{align}\la{2*}
V(\tilde{Y}_{k+1})
=&V(Y_{k})+\sum_{|\alpha|=1}^{3}\frac{ D^{\alpha} V (Y_k) }{\alpha!} \big(\tilde{Y}_{k+1}-Y_k\big)^{\alpha}+J(\tilde{Y}_{k+1}, Y_k),
\end{align}
where
\begin{align*}
J(\tilde{Y}_{k+1}, Y_k)
:= 4\sum_{|\alpha|=4}\frac{\big(\tilde{Y}_{k+1}- Y_k\big)^{\alpha}}{\alpha!}\int_{0}^{1}(1-t)^{3} D^{\alpha}
 V\big(Y_k+t\big(\tilde{Y}_{k+1}- Y_k\big)\big)\mathrm{d}t.
\end{align*}
One observes that
\begin{align*}
 \big| J(\tilde{Y}_{k+1}, Y_k)\big|
\leq& 4\sum_{|\alpha|=4}  \frac{\big|\big(\tilde{Y}_{k+1}- Y_k\big)^{\alpha}\big|}{\alpha!} \int_{0}^{1}(1-t)^{3} \big|
 D^{(4)} V\big(Y_k+t\big(\tilde{Y}_{k+1}- Y_k\big)\big)\big|\mathrm{d}t\nn\\
 \leq& \frac{c}{3!}\bigg(\sum_{i=1}^{d} \Big|f_{i}( Y_k)\t  +\sum_{j=1}^{m}g_{ij}(Y_k) \t B_{k}^{(j)}\Big|\bigg)^4 \nn\\
 &\qquad \quad \quad \times\int_{0}^{1}(1 -t)^{3}\Big[1+
 V\big(Y_k+t\big(\tilde{Y}_{k+1} - Y_k\big)\big)\Big]^{1- 4\delta_4}  \mathrm{d}t.
\end{align*}
Note that for any  $U\in \mathcal{V}^{4}_{\delta_4}$ we know
$|D^{(1)} (1+U(x) ) |\leq c (1+U(x) )^{1-\delta_4}$. By the result of \cite[Lemma 2.12, p.22]{Hutzenthaler15} we have
\begin{align*}
1+U(x+y)\leq c^{\frac{1}{\delta_4}}2^{\frac{1}{\delta_4}-1}\Big(\big|1+U(x)\big| +|y|^{\frac{1}{\delta_4}}\Big),\qquad\forall x,y\in \mathbb{R}^d,
\end{align*}
which leads to
\begin{align*}
  \Big[1+
 V\big(Y_k+t\big(\tilde{Y}_{k+1}- Y_k\big)\big)\Big]^{1-  4 \delta_4}
  \!\!\leq & \Big[c^{\frac{1}{\delta_4}}2^{\frac{1}{\delta_4}-1}\Big(1\!+ V(Y_k)+\!  t^{\frac{1}{\delta_4}} |\tilde{Y}_{k+1}- Y_k|^{\frac{1}{\delta_4}}\Big) \Big]^{1-  4 \delta_4}\nn\\
 \leq & C\Big[ \big(1+V(Y_k)\big)^{1- 4 \delta_4} + |\tilde{Y}_{k+1}- Y_k|^{\frac{1}{\delta_4} -4}\Big]
\end{align*}
for any $V\in \mathcal{V}^{4}_{\delta_4}$ with  $1/\delta_4\in [4, +\infty)$.
Therefore, we derive from  \eqref{Y_01} that 
\begin{align}\la{EQ1}
\big|J(\tilde{Y}_{k+1}, Y_k)\big|
\leq&C \Big[\Big(|f( Y_k)|^{4}\t^{4}+|g( Y_k)|^{4}|\t B_k|^{4}\Big) \big(1+V(Y_k)\big)^{1- 4 \delta_4 }\nn\\
&\qquad\qquad+ \Big(|f( Y_k)|^{\frac{1}{\delta_4}}\t^{\frac{1}{\delta_4}}+|g( Y_k)|^{\frac{1}{\delta_4}}|\t B_k|^{\frac{1}{\delta_4}}\Big)  \Big]\nn\\
\leq&C \bigg\{\Big[\big(1+V(Y_k)\big)^{ 4 \delta_4} \t^{4(1-\theta)}\nn\\
&\qquad\qquad+\big(1+V(Y_k)
\big)^{ 4 \delta_4} \t^{-2\theta} |\t B_k|^{4}\Big] \big(1+V(Y_k)\big)^{1- 4 \delta_4} \nn\\
&\qquad\qquad+  \big(1+V(Y_k)\big) \t^{\frac{1-\theta}{\delta_4}}
+\big(1+V(Y_k)\big) \t^{\frac{-\theta}{2\delta_4}} |\t B_k|^{\frac{1}{\delta_4}}   \bigg\}\nn\\
\leq&
C\big(1+V(Y_k)\big)\t^{4(1-\theta)}+\mathcal{J}^{\t}V(Y_k),
\end{align}
where
\begin{align}\la{EQ*1}
\mathcal{J}^{\t}V(Y_k)
=C\big(1+V(Y_k)\big)\big(\t^{-2\theta}|\t B_k|^{4}
 + \t^{\frac{-\theta}{2\delta_4}}|\t B_k|^{\frac{1}{\delta_4}}  \big).
\end{align}
  For any $\rho>0$,  substituting \eqref{neq:9} and \eqref{EQ1} into \eqref{2*}, then using   the second inequality of \eqref{lyhf1},  we yield that
\begin{align}\la{yv3.20}
\big(1+V(\tilde{Y}_{k+1})\big)^{\rho}
\leq& \big(1+V(Y_{k})\big)^{\rho}\big(1+\xi_k\big)^{\rho},
\end{align}
where
\begin{align*}
\xi_k=\frac{\mathcal{L}V(Y_{k}) \t+C \big(1+V(Y_k)\big) \t^{2(1-\theta)} +\sum_{i=1}^{3}\mathcal{S}_{i}^{\t}V(Y_k)
+\mathcal{J}^{\t}V(Y_k)}{1+V(Y_{k})},
\end{align*}
and  we can see that $\xi_k>-1$. By the virtue of \cite[Inequality (3.12)]{Li18},
without loss of  generality we prove \eqref{3.22y} only for $0<\rho\leq 1$.   It follows from
\eqref{yv3.20} that
\begin{align}\la{vys1}
\mathbb{E}_k\Big[\big(1+V(\tilde{Y}_{k+1})\big)^{\rho}\Big]
\leq\big(1+V(Y_{k})\big)^{\rho}\bigg(&1+ \rho\mathbb{E}_k\big[\xi_k\big]+\frac{\rho(\rho-1)}{2}
\mathbb{E}_k\big[\xi_k^2\big]\nn\\
&
+\frac{\rho(\rho-1)(\rho-2)}{6}\mathbb{E}_k\big[\xi_k^3\big]\bigg).
\end{align}
In order to  independently estimates each of the expectations on the right-hand side of   inequality \eqref{vys1}, we divide it into three steps.
%
%
\\
{\bf Step 1.}   We estimate $\mathbb{E}_k\big[\xi_k\big]$. Due to \eqref{EQ*1}  and \eqref{E_k}, we deduce that
\begin{align}\la{y1.27}
\mathbb{E}_k\big[|\mathcal{J}^{\t}V(Y_k)|\big]
\leq&C \big(1+V(Y_k)\big) \big(\t^{2(1-\theta)} +\t^{\frac{ 1-\theta }{2\delta_4}} \big) \nn\\
\leq&C  \big(1+V(Y_k)\big)  \t^{2(1-\theta)}.
\end{align}
This together with  Lemma \ref{le1.3}  implies
\begin{align}\la{vys2}
 \mathbb{E}_k\big[\xi_k\big]
\leq& \big(1+V(Y_k)\big)^{-1}
 \Big[\mathcal{L}V(Y_{k})\t+C \big(1+V(Y_k)\big) \t^{2(1-\theta)}
  \Big]\nn\\
=&\big(1+V(Y_k)\big)^{-1}  \mathcal{L}V(Y_{k})\t+C\t^{2(1-\theta)}.
\end{align}
{\bf Step 2.}   We estimate $\mathbb{E}_k\big[\xi_k^2\big]$.
   Similar to \eqref{y1.27}, combining \eqref{fA.1} and \eqref{EQ*1} imply
\begin{align*}
&\mathbb{E}_k\Big[\mathcal{S}_{1}^{\t}V(Y_k) \mathcal{J}^{\t}V(Y_k) \Big]\nn\\
=&C\big(1+V(Y_k)\big)\mathbb{E}_k\Big[\langle D^{(1)}V(Y_k), g( Y_k)\t B_k\rangle\Big( \t^{-2\theta}|\t B_k|^{4}
 + \t^{\frac{-\theta}{2\delta_4}} |\t B_k|^{\frac{1}{\delta_4}} \Big)\Big]\nn\\
\geq&-C\t^{\frac{-\theta}{2\delta_4}}\big(1+V(Y_k)\big)^{2-\delta_4}
 |g( Y_k)| \mathbb{E}_k\Big[|\t B_k|^{1+\frac{1}{\delta_4}} \Big]\!
 \geq -C\t^{2(1-\theta )}\big(1+V(Y_k)\big)^2
\end{align*}
for  $1/\delta_4\in [4, +\infty)$.   This together with
  Lemma \ref{Z:1} implies
\begin{align}
\mathbb{E}_k\big[\xi_k^2\big]=& \big(1+V(Y_k)\big)^{-2}
\mathbb{E}_k\bigg[\Big( \mathcal{L}V(Y_{k})\t+C\big(1+V(Y_k)\big) \t^{2(1-\theta)}\nn\\
&\qquad\qquad\quad+\sum_{i=1}^{3}\mathcal{S}_{i}^{\t}V(Y_k)+ \mathcal{J}^{\t}V(Y_k)\Big)^2\bigg]\nn\\
\geq& \big(1+V(Y_k)\big)^{-2}
\mathbb{E}_k\bigg\{ |\mathcal{S}_{1}^{\t}V(Y_k)|^2+2\mathcal{S}_{1}^{\t}V(Y_k)\Big[\mathcal{L}V(Y_{k})\t \nn\\
&\qquad\qquad\quad+C\big(1+V(Y_k)\big) \t^{2(1-\theta)}+\sum_{i=2}^{3}\mathcal{S}_{i}^{\t}V(Y_k)
+ \mathcal{J}^{\t}V(Y_k)\Big] \bigg\} \nn\\
\geq&  \big(1+V(Y_k)\big)^{-2} \bigg[ | D^{(1)}V(Y_{k})  g( Y_{k})|^2\t
+ 2\mathbb{E}_k\Big(\mathcal{S}_{1}^{\t}V(Y_k)\mathcal{S}_{2}^{\t}V(Y_k)\Big)\nn\\
&\qquad\qquad\quad
+ 2\mathbb{E}_k\Big(\mathcal{S}_{1}^{\t}V(Y_k)\mathcal{S}_{3}^{\t}V(Y_k)\Big)
+ 2\mathbb{E}_k\Big(\mathcal{S}_{1}^{\t}V(Y_k)\mathcal{J}^{\t}V(Y_k)\Big)\bigg] \nn\\
\geq& \big(1+V(Y_k)\big)^{-2}| D^{(1)}V(Y_{k}) g( Y_{k})|^2\t
 - C  \t^{2(1-\theta)}.\la{vys3}
\end{align}
{\bf Step 3.}  We estimate $\mathbb{E}_k\big[\xi_k^3\big]$. Similar to \eqref{y1.27}, we can also prove that
\begin{align*}
  \mathbb{E}_k\big[|\mathcal{J}^{\t}V(Y_k)|^2\big]
\leq&C \big(1+V(Y_k)\big)^2 \mathbb{E}_k\Big[\t^{-4\theta}|\t B_k|^{8}
 + \t^{\frac{-\theta}{\delta_4}}|\t B_k|^{\frac{2}{\delta_4}} \Big]\nn\\
\leq&C  \big(1+V(Y_k)\big)^2  \t^{4(1-\theta)},
  \end{align*}
and
\begin{align*}
\mathbb{E}_k\big[|\mathcal{J}^{\t}V(Y_k)|^3\big]
\leq&C\big(1+V(Y_k)\big)^3 \mathbb{E}_k\Big[ \t^{-6\theta} |\t B_k|^{12}
 + \t^{\frac{-3\theta}{2\delta_4}}|\t B_k|^{\frac{3}{\delta_4}}   \Big]\nn\\
\leq&C \big(1+V(Y_k)\big)^3 \t^{6(1-\theta)}.
  \end{align*}
 This together with   \eqref{y1.27}  as well as Lemma \ref{Z:1}  implies
\begin{align}\la{xi_3}
 \mathbb{E}_k\big[\xi_k^3\big]
\leq& \big(1+V(Y_k)\big)^{-3}
\mathbb{E}_k \Big[\Big(C\big(1+V(Y_k)\big) \t^{1-\theta} +\sum_{i=1}^{3}\mathcal{S}_{i}^{\t}V(Y_k)
+\mathcal{J}^{\t}V(Y_k)\Big)^3\Big] \nn\\
\leq& C\big(1+V(Y_k)\big)^{-3}
\mathbb{E}_k\bigg\{ \big(1\!+V(Y_k)\big)^3 \t^{3(1- \theta)} \!+\Big(\sum_{i=1}^{3}\mathcal{S}_{i}^{\t}V(Y_k)\!+\mathcal{J}^{\t}V(Y_k)\Big)^3\nn\\
&\qquad\qquad\qquad\qquad+   \big(1+V(Y_k)\big)^2  \t^{2(1-\theta)} \bigg( \sum_{i=1}^{3} \mathcal{S}_{i}^{\t}V(Y_k)
+\mathcal{J}^{\t}V(Y_k) \bigg) \nn\\
&\qquad\qquad\qquad\qquad+  \big(1+V(Y_k)\big)  \t^{1-\theta} \bigg( \sum_{i=1}^{3} \mathcal{S}_{i}^{\t}V(Y_k)
+\mathcal{J}^{\t}V(Y_k) \bigg)^2 \bigg\}
\nn\\
\leq& C\big(1+V(Y_k)\big)^{-3}
\bigg\{ \big(1+V(Y_k)\big)^{3} \t^{3(1-\theta)}+\mathbb{E}_k\Big[\big|\mathcal{J}^{\t}V(Y_k)\big|^3\Big]
\nn\\
&\!\!\! +\sum_{i=1}^{3}\mathbb{E}_k\Big[\big|\mathcal{S}_{i}^{\t}V(Y_k)\big|^2\big|\mathcal{J}^{\t}V(Y_k)\big| \Big]
+ \big(1+V(Y_k)\big)^{2} \t^{2(1- \theta)}   \mathbb{E}_k\Big[|\mathcal{J}^{\t}V(Y_k)|\Big] \nn\\
&\qquad  + \big(1+V(Y_k)\big)  \t^{1-\theta}  \bigg( \sum_{i=1}^{3} \mathbb{E}_k\Big[\big|\mathcal{S}_{i}^{\t}V(Y_k)\big|^2\Big]
+\mathbb{E}_k\Big[\big|\mathcal{J}^{\t}V(Y_k)\big|^2\Big] \bigg)
\bigg\}
\nn\\
\leq& C\big(1+V(Y_k)\big)^{-3}
\bigg\{ \big(1+V(Y_k)\big)^{3} \t^{3(1-\theta)}
 + \big(1+V(Y_k)\big)^3  \t^{2(1-\theta)}
 \nn\\
& \qquad \qquad\qquad\qquad +\sum_{i=1}^{3}\mathbb{E}_k\Big[\big|\mathcal{S}_{i}^{\t}V(Y_k)\big|^2\big|\mathcal{J}^{\t}V(Y_k)\big| \Big]\bigg\}.
\end{align}
On the other hand,  by \eqref{E_k}, \eqref{H_11} and \eqref{EQ*1}, one observes
\begin{align*}
&\mathbb{E}_k\Big[\big|\mathcal{S}_{1}^{\t}V(Y_k)\big|^2|\mathcal{J}^{\t}V(Y_k)|\Big]
\nn\\
\leq&C\big(1+V(Y_k)\big)\mathbb{E}_k\Big[\Big(\big(1+V(Y_{k})\big)^2\t^{1-\theta}\nn\\
  &\qquad\qquad\qquad\qquad\qquad+\mathcal{H}_{1,1}^{\t}V(Y_k)\Big)
\Big(\t^{-2\theta}|\t B_k|^{4}
 + \t^{\frac{-\theta}{2\delta_4}}|\t B_k|^{\frac{1}{\delta_4}}  \Big) \Big]\nn\\
\leq&C\big(1+V(Y_k)\big)^3\mathbb{E}_k\Big[\big(\t^{1-\theta}+ \t^{-\theta}|\t B_k|^{2}\big)
\big(\t^{-2\theta}|\t B_k|^{4}
 + \t^{\frac{-\theta}{2\delta_4}}|\t B_k|^{\frac{1}{\delta_4}}  \big) \Big]\nn\\
\leq&C\big(1+V(Y_k)\big)^3 \t^{3(1-\theta)}.
\end{align*}
 Similarly, by \eqref{E_k}, \eqref{H_22}, \eqref{H_33} and \eqref{EQ*1},   we can also prove that
\begin{align*}
 &\mathbb{E}_k\Big[\big|\mathcal{S}_{2}^{\t}V(Y_k)\big|^2 |\mathcal{J}^{\t}V(Y_k)| \Big]\nn\\
 \leq& C\big(1+V(Y_{k})\big)^3 \mathbb{E}_k\Big[\Big( \t^{-2\theta}|\t B_k|^4+ \t^{2(1-\theta)}\nn\\
  &\qquad\qquad\qquad\qquad\qquad+ |\t B_k|^2\t^{2-3\theta}\Big)\Big(  \t^{-2\theta}|\t B_k|^{4}
 + \t^{\frac{-\theta}{2\delta_4}}|\t B_k|^{\frac{1}{\delta_4}}    \Big)\Big]\nn\\
  \leq& C\big(1+V(Y_{k})\big)^3\t^{4(1-\theta)},
\end{align*}
and
\begin{align*}
 &\mathbb{E}_k\Big[\big|\mathcal{S}_{3}^{\t}V(Y_k)\big|^2 |\mathcal{J}^{\t}V(Y_k)| \Big]\nn\\
 \leq& C\big(1+V(Y_{k})\big)^3 \mathbb{E}_k\Big[\Big( |\t B_k|^2\t^{4-5\theta)}+\t^{2-4\theta} |\t B_k|^4+ \t^{4(1-\theta)} \nn\\
  &\qquad\qquad\qquad\qquad\qquad+  \t^{-3\theta}|\t B_k|^6\Big)\Big( \t^{-2\theta}|\t B_k|^{4}
 + \t^{\frac{-\theta}{2\delta_4}}|\t B_k|^{\frac{1}{\delta_4}}   \Big)\Big]\nn\\
   \leq& C\big(1+V(Y_{k})\big)^3\t^{5(1-\theta)}.
\end{align*}
Thus the above inequalities and  \eqref{xi_3}  imply
\begin{align}\la{vys4}
\mathbb{E}_k\big[\xi_k^3\big]
\leq C\t^{2(1-\theta)}
  \Big(1 +\t^{1-\theta} +  \t^{2(1-\theta)} +  \t^{3(1-\theta)}\Big)
\leq C \t^{2(1-\theta)}.
\end{align}
 Using \eqref{vys2}, \eqref{vys3} and \eqref{vys4} as well as \eqref{1.4} we now establish inequality \eqref{3.22y}. To this end,
 combining \eqref{vys1}-\eqref{vys4}, we know that for
  any integer $k\geq 0$,
\begin{align*}
 \mathbb{E}_k\Big[ \big(1+V(\tilde{Y}_{k+1})\big)^{\rho}\Big]
\leq& \big(1+V(Y_{k})\big)^{\rho}\bigg[1+C \t^{2(1-\theta)}\nn\\
&+\frac{\rho\t}{2}\frac{2 \big(1+V(Y_{k})\big)
\mathcal{L}V(Y_{k})-(1-\rho)| D^{(1)}V(Y_{k})  g(  Y_{k})|^2}{ \big(1+V(Y_{k}) \big)^{2}}\bigg]\nn\\
\leq& \big(1+V(Y_{k})\big)^{\rho}\Big(1+C \t^{2(1-\theta)}\Big) +\mathcal{L}\big(1+V(Y_{k})\big)^{\rho}.
\end{align*}
Making use of  the above inequality as well as  \eqref{1.4}  yields
 \begin{align*}
 \mathbb{E}_k\Big[\big(1+ V(\tilde{Y}_{k+1})\big)^{\rho} \Big]
\leq& \big(1+  V(Y_{k})\big)^{\rho}\big(1+ C  \t\big)
+ \lambda \big(1+V^{\rho}(Y_{k})\big) \t\nn\\
\leq&\big(1+  V(Y_{k})\big)^{\rho}\big(1+ C  \t\big)+ \lambda\t
\end{align*}
 for any integer $k\geq 0$.   Furthermore,
$$V(\pi_{\t}(x))\leq  V(x) \qquad \forall~x\in \mathbb{R}^d,$$
which implies that
\begin{align*}
\mathbb{E}_k\Big[\big(1+ V(Y_{k+1})\big)^{\rho} \Big]
\leq \mathbb{E}_k\Big[\big(1+ V(\tilde{Y}_{k+1})\big)^{\rho} \Big]
\leq \big(1+  V(Y_{k})\big)^{\rho} \big(1+ C  \t\big)+ \lambda\t .
\end{align*}
 Repeating this procedure we obtain
\begin{align*}
\mathbb{E}\big[ \big(1+V(Y_{k})\big)^{\rho}\big|\mathcal{F}_0\big]
 \leq&  (1+C\t)^{k} \big(1+V(x_0)\big)^{\rho}+ \lambda \t \sum_{i=0}^{k-1}(1+C\t)^{i}.
\end{align*}
Taking expectations on both sides yields
\begin{align*}
\mathbb{E}\big[ V^{\rho}(Y_{k})\big]
 \leq&  C(1+C\t)^{k}\leq  C\exp\left(C k \t\right)
 \leq  C\exp\left(CT\right)
\end{align*}
for any integer $k$ satisfying $0\leq k\t\leq T$. Therefore the desired result follows.
\end{proof}

\begin{remark}
 If $V\in  \mathcal{V}^{p}_{\delta_p}\cap \{D^{(p+1)}V(\cdot)\equiv0\}$ for $p=2$ or $3$ and some integer $1/\delta_p\in [p, +\infty)$, then $V\in  \mathcal{V}^{4}_{\delta_4}$ with $1/\delta_4=4\vee 1/\delta_p$.
\end{remark}

\begin{remark}
  If $\tilde{V}\in  \mathcal{V}^{4}_{\delta_4}$ in Theorme \ref{co3.1} is replaced by $\tilde{V}\in  \mathcal{V}^{3}_{\delta_3}$ for some integer  $1/\delta_3\in [3, +\infty)$,
we can then choose a strictly increasing continuous  function  $\f: \bar{\RR}_+\rightarrow \bar{\RR}_+$ such that $\f(u)\rightarrow \infty$
as $u\rightarrow \infty$ and
$$
\sup_{|x|\leq u}  \left(\frac{ |f (  x)|}{\big(1+\tilde{V}(x)\big)^{\delta_3}}\vee\frac{| g(  x)|^2 }{\big(1+\tilde{V}(x)\big)^{2 \delta_3}}\right)\leq \f(u),\qquad\forall~ u\geq 1.
$$
Then for any given $0<\theta\leq 1/3$ the corresponding  $V$-truncated EM scheme \eqref{Y_0} still has $V$-integrability \eqref{3.22y}.
  It turns out that the smoothness of $V(x)$ affects the construction of the scheme (\ref{Y_0}).
\end{remark}

\begin{remark}
 $V$-integrability of numerical schemes has already been well studied in \cite{Hutzenthaler15, Lukasz17}.   However the results are based on some priori inequality estimates (see, \cite[Proposition 2.7]{Hutzenthaler15} and \cite[Theorem 2.5]{Lukasz17}). Here we show that without additional restricted conditions except those which guarantee the  $V$-integrability of exact solutions, the $V$-truncated EM scheme \eqref{Y_0} has $V$-integrability just like the form of \eqref{3.22y}.
\end{remark}
\begin{lemma}\la{L:C_1}
Under the conditions of Theorem \ref{co3.1}, for any $\t\in(0,\t^*]$  define
\begin{align}\la{3.18}
\eta_{\t}=:\inf\big\{t\geq 0:|\tilde{Y}(t)|\geq  \varphi^{-1}\big(K\t^{-\theta}\big)\big\}.
\end{align}
Then for any   $T>0$,
\begin{align*}
 \E V^{\rho}\big(\tilde{Y}(T\wedge \eta_{\t})\big)\leq C.
\end{align*}
  \end{lemma}
\begin{proof}
 Define~$\beta_{\t}=:\inf\left\{k\geq 0:|\tilde{Y}_k|\geq \varphi^{-1}\big(K\t^{-\theta}\big)\right\},$ we have
$
\eta_{\t}=\t \beta_{\t}.
$
  We write $\beta=\beta_{\t}$ shortly. Obviously, both  $\eta_{\t}$ and $\beta$ are   $\F_{t_k}$ stopping times.
For $\omega\in \{\beta\geq k+1\}$, we have  $Y_k=\tilde{Y}_{k\wedge \beta}$ and
$$
\tilde{Y}_{(k+1)\wedge \beta}=\tilde{Y}_{k+1}.
$$
On the other hand, for $\omega\in \{\beta< k+1\}$, we have $\beta\leq k$ and hence
$$
\tilde{Y}_{(k+1)\wedge \beta}=\tilde{Y}_{\beta}=\tilde{Y}_{k\wedge\beta}.
$$
Therefore, we derive from (\ref{Y_0}) that for any integer $k\geq 0$,
\begin{align*}
 \tilde{Y}_{(k+1)\wedge \beta}
  =\tilde{Y}_{k\wedge \beta}+\left[f(\tilde{Y}_{k\wedge \beta})\t+g(\tilde{Y}_{k\wedge \beta})\t B_{k\wedge \beta}\right]I_{[[0, \beta]]}(k+1).
\end{align*}
Since $V(x)\in \mathcal{V}^{4}_{\delta_4}$, using the Taylor formula with integral remainder term we get
\begin{align}\la{Leq1}
V(\tilde{Y}_{(k+1)\wedge \beta})&
=\! V(\tilde{Y}_{k\wedge \beta})\!+\!\!\sum_{|\alpha|=1}^{3}\frac{ D^{\alpha} V (\tilde{Y}_{k\wedge \beta}) }{\alpha!} \big(\tilde{Y}_{(k+1)\wedge \beta}\!-\tilde{Y}_{k\wedge \beta}\big)^{\alpha}+J(\tilde{Y}_{(k+1)\wedge \beta}, \tilde{Y}_{k\wedge \beta}).
\end{align}
Note that
$$
\t B_{k\wedge \beta}I_{[[0, \beta]]}(k+1)=\big(B(t_{(k+1)\wedge \beta})-B(t_{k\wedge \beta})\big)I_{[[0, \beta]]}(k+1).
$$
Since $B(t)$ is a continuous local martingale, by the virtue of the Doob martingale stopping time theorem (see e.g., \cite[pp. 11-12]{Mao08}),
  we know that
\begin{align*}
\E_{k\wedge \beta}\Big[\big(\t B_{k\wedge \beta}^{(l)}\big)^{2r+1} I_{[[0, \beta]]}(k+1)\Big]=\E\Big[\big(\t B_{k\wedge \beta}^{(l)}\big)^{2r+1} I_{[[0, \beta]]}(k+1)\big|\mathcal{F}_{t_{k\wedge \beta}}\Big]=0,
\end{align*}
and
$$
\E_{k\wedge \beta}\Big[\big(\t B_{k\wedge \beta}^{(l)}\big)^{2r}  I_{[[0, \beta]]}(k+1)\Big]=(2r-1)!!\t^r  \E_{k\wedge \beta}\big[I_{[[0, \beta]]}(k+1)\big],
$$
where $\mathbb{E}_{k\wedge \beta}[\cdot]:=\mathbb{E}\big[\cdot|\mathcal{F}_{t_{k\wedge \beta}}\big]$ and $\t B_{k\wedge \beta}=\big(\t B_{k\wedge \beta}^{(1)}, \ldots, \t B_{k\wedge \beta}^{(m)}\big)^T$.
Using the techniques in the proof of Lemma \ref{le1.3},  we deduce that
 \begin{align*}
&\sum_{|\alpha|=1}^{3}\!\!\frac{D^{\alpha} V (\tilde{Y}_{k\wedge \beta})}{\alpha!}\big(\tilde{Y}_{(k+1)\wedge \beta}\!-\tilde{Y}_{k\wedge \beta}\big)^{\alpha}\nn\\
\leq &\Big[\mathcal{L}V(\tilde{Y}_{k\wedge \beta})\t\!+ \mathcal{R}^{\t}V(\tilde{Y}_{k\wedge \beta})+\!\sum_{i=1}^{3}\mathcal{S}_{i}^{\t}V(\tilde{Y}_{k\wedge \beta})\bigg] I_{[[0, \beta]]}(k\!+1).
\end{align*}
This together with \eqref{EQ1} and \eqref{Leq1} implies
\begin{align*}
 \big(1+V(\tilde{Y}_{(k+1)\wedge \beta})\big)^{\rho}
\leq&\Big[1+V(\tilde{Y}_{k\wedge \beta})+\Big(\mathcal{L}V(\tilde{Y}_{k\wedge \beta})\t+ \mathcal{R}^{\t}V(\tilde{Y}_{k\wedge \beta})\nn\\
&\qquad + \sum_{i=1}^{3}\mathcal{S}_{i}^{\t}V(\tilde{Y}_{k\wedge \beta})+J(\tilde{Y}_{(k+1)\wedge \beta}, \tilde{Y}_{k\wedge \beta})\Big)I_{[[0, \beta]]}(k+1) \Big]^{\rho}\nn\\
\leq&\big(1+V(\tilde{Y}_{k\wedge \beta})\big)^{\rho}\big(1+  \bar{\xi}_{k\wedge \beta}I_{[[0, \beta]]}(k+1)\big)^{\rho},
\end{align*}
where
\begin{align*}
\bar{\xi}_{k\wedge \beta}=\frac{ \mathcal{L}V(\tilde{Y}_{k\wedge \beta})\t\! + C\big(1+V(\tilde{Y}_{k\wedge \beta})\big)\t^{2(1-\theta)} \!+\sum_{i=1}^{3}\mathcal{S}_{i}^{\t}V(\tilde{Y}_{k\wedge \beta})\!+\mathcal{J}^{\t}V(\tilde{Y}_{k\wedge \beta}) }{1+V(\tilde{Y}_{k\wedge \beta})}.
\end{align*}
Using the techniques in the proof of Theorem \ref{co3.1}, we show that
\begin{align*}
&\mathbb{E}_{k\wedge \beta}\Big[ \big(1+V(\tilde{Y}_{(k+1)\wedge \beta})\big)^{\rho}\Big]\nn\\
\leq& \big(1+V(\tilde{Y}_{k\wedge \beta})\big)^{\rho}\Bigg\{1 +\mathbb{E}_{k\wedge \beta}\bigg[\frac{\rho\t}{2}\big(1+V(\tilde{Y}_{k\wedge \beta})\big)^{-2}\Big(2 \big(1+V(\tilde{Y}_{k\wedge \beta})\big)
\mathcal{L}V(\tilde{Y}_{k\wedge \beta})\nn\\
&\qquad\qquad\qquad\qquad-(1-\rho)| D^{(1)}V(\tilde{Y}_{k\wedge \beta}) g(\tilde{Y}_{k\wedge \beta})|^2\Big) I_{[[0, \beta]]}(k+1)\bigg]+C  \t^{2(1-\theta)}\Bigg\}\nn\\
\leq& \big(1+V(\tilde{Y}_{k\wedge \beta})\big)^{\rho} \big(1+ C  \t\big)+ \lambda\t.
\end{align*}
Repeating this procedure we obtain
\begin{align*}
\mathbb{E}\big[ \big(1+V(\tilde{Y}_{k\wedge \beta})\big)^{\rho}|\mathcal{F}_{0}\big]
 \leq \big(1+ V(x_0)\big)^{\rho}(1+C\t)^k+ \lambda\t\sum^{k-1}_{i=0}(1+C\t)^i
\end{align*}
for any integer $k$ satisfying $0\leq k\t\leq T$. Taking expectations on both sides yields
\begin{align*}
\mathbb{E}\big[ V^{\rho}(\tilde{Y}_{k\wedge \beta})\big]
 \leq& C (1+C\t)^{k} \leq   C\exp\left(C  k \t\right)
 \leq C\exp\left(CT\right).
\end{align*}
  Therefore, the desired assertion follows from
\begin{align*}
 \E V^{\rho}\big(\tilde{Y}(T\wedge \eta_{\t})\big)= \E V^{\rho}\big(\tilde{Y}_{\lfloor\frac{T}{\t}\rfloor\wedge \beta}\big)\leq C,
\end{align*}
where $\lfloor\frac{T}{\t}\rfloor$ represents the integer part of $T/\t$.  The proof is complete.
\end{proof}

Let
  $\mathcal{K}$ denote the family of all continuous increasing functions
 $\kappa: \bar{\mathbb{R}}_+\rightarrow\bar{\mathbb{R}}_+$ such that $\kappa(0)=0$ while $\kappa(u)>0$ for $u>0$. Denote by $\kappa^{-1}$ the inverse function of $\kappa\in \mathcal{K}$. Let  $\mathcal{K}_{\vee}$ denote the family of all convex functions $\kappa\in \mathcal{K}$, and $\mathcal{K}_{\wedge}$ denote the family of all concave functions $\kappa\in \mathcal{K}$.

\begin{theorem}\la{T:C_2}
Let the conditions of  Theorem \ref{co3.1} hold.
If there moreover exist a function $\kappa\in \mathcal{K}_{\vee}$ and a constant $\bar{p}>0$  such that
$$
\kappa\big(|x|^{\bar{p}}\big)\leq V^{\rho}(x)\quad \forall x\in \mathbb{R}^d,
$$
then for any $q\in (0, \bar{p})$,
\be \la{F_0}
\lim_{\t\rightarrow 0} \E |X(T)-Y(T)|^q=0,\quad \forall ~T\geq 0.
\ee
\end{theorem}
\begin{proof}
Let  $\eta_{\t}$ be defined as before.
Define
\begin{align}\la{tau_N}
\tau_{N}=:\inf\{t\geq 0:|X(t)|\geq N\},\
\theta_{N, \t}:=\tau_N \wedge \eta_{\t},\ e_{\t }(T):=X(T)- {Y}(T).
\end{align}
For any $l>0$, using the Young inequality  we obtain that
\begin{align}\la{F_1}
\!\!\!\!\!   \E|e_{\t }(T)|^q
 \!\!=& \E\lf(|e_{\t }(T)|^q I_{\{\theta_{N, \t} \geq T\}}\rt)
+ \E\lf(|e_{\t }(T)|^q I_{\{\theta_{N, \t} \leq T\}}\rt) \nonumber \\
 \!\!\leq &  \E\lf(|e_{\t }(T)|^q I_{\{\theta_{N, \t} \geq T\}}\rt)
\!+  \frac{q l}{\bar{p}}\E\lf(|e_{\t }(T)|^{\bar{p}} \rt)\!+\! \frac{\bar{p}-q}{\bar{p} l^{q/(\bar{p}-q)}} \PP {\{\theta_{N, \t} \leq T\}}.\!\!\!
\end{align}
 It follows from the virtues of   Theorem  \ref{th1} and  Theorem \ref{co3.1} that
\begin{align*}
 \E |e_{\t }(T)|^{\bar{p}}  \leq&  2^{{\bar{p}}}\big( \E |X(T)|^{\bar{p}}
+  \E |Y(T)|^{\bar{p}}\big)\nn\\
 \leq& 2^{{\bar{p}}}\Big(\kappa^{-1}\big(\E V^{\rho}(X(T))\big) +\kappa^{-1}\big(\E V^{\rho}(Y(T))\big)\Big)\leq    C,
\end{align*}
where $\kappa^{-1} $ is the inverse function of $\kappa$. Now let $\e>0$ be arbitrary.  Choose $l>0 $ small sufficiently such that $ {Cq l}/{\bar{p}} \leq \e/3$, then we have
\be\la{F_2}
\frac{q l}{\bar{p}}\E\lf(|e_{\t }(T)|^{\bar{p}}  \rt)\leq \frac{\e}{3}.
\ee
Then choose $N>1$ large sufficiently such that
$ \frac{C({\bar{p}}-q)}{\kappa(N^{\bar{p}}){\bar{p}} l^{q/({\bar{p}}-q)}}\leq   \frac{\e}{6}.$
Choose $\bar{\t}^{*}\in (0, \t^{*}]$ small sufficiently such that
$
 \f^{-1}\big(K(\t^*)^{-\theta}\big)\geq N.
$
Then for any $\t\in(0, \bar{\t}^{*}]$,
by the virtue of Theorem \ref{th1} and Lemma \ref{L:C_1}, we have
$$
 \PP {\{\tau_{N } \leq T\}} \leq  \frac{C}{\kappa\big(N^{\bar{p}}\big)},\qquad\PP {\{\eta_{\t} \leq T\}}\leq \frac{C }{\kappa\big( |\f^{-1}\big(K\t^{-\theta}\big)|^{\bar{p}}\big) }\leq   \frac{C }{\kappa\big(N^{\bar{p}}\big) }.
$$
This implies that
\be\la{F_3}
\frac{ {\bar{p}}-q }{   {\bar{p}} l^{q/({\bar{p}}-q)}} \PP {\{\theta_{N, \t} \leq T\}}\leq \frac{ {\bar{p}}-q }{   {\bar{p}} l^{q/({\bar{p}}-q)}} \Big(\PP {\{\tau_{N } \leq T\}}+\PP {\{\eta_{\t} \leq T\}}\Big)
\leq \frac{\e}{3}  .
\ee
Combining (\ref{F_1}), (\ref{F_2}) and (\ref{F_3}), we know that for the chosen $N$ and all $\t\in (0, \bar{\t}^{*}]$,
$$\E|e_{\t }(T)|^q \leq \E\lf(| {e}_{\t }(T)|^q I_{\{\theta_{N, \t} \geq T\}}\rt)+\frac{2\e}{3}. $$
If we can show that
\be\la{F_4}\lim_{\t\rightarrow 0}\E\lf(| {e}_{\t }(T)|^q I_{\{\theta_{N, \t} \geq T\}}\rt)=0,\ee
the required assertion  follows.
For this purpose  we define the truncated functions
$$
f_N(x)=f\Big(\big(|x|\wedge N \big) \frac{x}{|x|}\Big),~~\hbox{and}~~g_N( x)=g\Big(\big(|x|\wedge N \big) \frac{x}{|x|}\Big)
 $$
for any $x\in \RR^d$. Consider the truncated SDE
\be\la{F_5}
dz(t) =f_N(z(t))dt +g_N(z(t))dB(t)
\ee
with  the initial value $z(0)=x_0$. For the chosen $N$,  we know that $f_N(\cdot)$ and $g_N(\cdot)$ are globally Lipschitz continuous with the Lipschitz constant $C_N$. Therefore, SDE (\ref{F_5}) has a unique regular solution $z(t)$ satisfying
\be\la{F_6}
X(t\wedge \tau_N)=z(t\wedge \tau_N)\ \ \hbox{a.s.}\qquad\forall~ t\geq 0.
\ee
On the other hand, for each $\t\in (0, \bar{\t}^{*}]$, we apply the EM method to  (\ref{F_5}) and  denote by $u(t)$ the piecewise constant EM solution  (see \cite{Hi02,Kloeden}) which has the property
\be\la{F_7}
\E\Big(\sup_{0\leq t\leq T}|z(t)-u(t)|^q \Big)\leq C_N\t^{q/2}, \qquad\forall ~T\geq 0.
\ee
Due to $ \f^{-1}\big(K\t^{-\theta}\big)\geq N$ for any $\t\in (0, \bar{\t}^{*}]$,  we have
$
Y(t\wedge \theta_{N, \t})=u(t\wedge \theta_{N, \t})~~\hbox{a.s.}
 $  This together with  \eqref{F_6}  implies
\begin{align*}
\E\lf(| {e}_{\t }(T)|^q I_{\{\theta_{N, \t} \geq T\}}\rt)
 =&\E\lf(| {e}_{\t }(T\wedge \theta_{N, \t})|^q I_{\{\theta_{N, \t} \geq T\}}\rt)
 \\
 \leq &  \E\Big(\sup_{0\leq t\leq T\wedge \theta_{N, \t}}| z(t)- u(t )|^q   \Big).
\end{align*}
This together with \eqref{F_7} implies (\ref{F_4}) as desired. The proof is  complete.
\end{proof}

 \section{Convergence rate}\label{c-r}
In this section, our aim is to establish a rate of convergence result under the conditions of Theorem \ref{T:C_2} and additional
conditions on $f$ and $g$.   The convergence  rate to be established is optimal as it is similar to the standard results of the  explicit EM scheme for SDEs with globally Lipschitz  coefficients.
The work of Higham et al. \cite{Hi02} gave the optimal rate in $p$th moment
for the implicit EM scheme for $p\geq 2$ with the global Lipschitz $g$ and a one-sided Lipschitz $f$ together with the
polynomial growth conditoin. Using the similar conditions to those in \cite{Hi02}, the rate for the
tamed Euler was obtained by Hutzenthaler et al. \cite{Hutzenthaler12}. Using the special Lyapunov function $|x|^{p}$ with $p\geq 2$,  Sabanis \cite{Sabanis16} developed the tamed
EM scheme, and obtained the optimal convergence rate.

 For convenience, for a pair of integers  $p\in [2, +\infty)$ and $1/\delta_{p}\in [4, +\infty)$,
 define
\begin{align}\la{D**}
 \bar{\mathcal{V}}^{p}_{\delta_{p}}:= \Big\{V(\cdot) \in \mathcal{C}^{p}_\infty(\mathbb{R}^d;\bar{\mathbb{R}}_{+})\Big|& \mathrm{Ker}(V)=\{\mathbf{0}\}, \exists ~c>0~ ~\mathrm{s.t.}~ \nn\\
 &\quad  |D^{(n)}  V(\cdot)|\leq c V^{1-  n \delta_{p} }(\cdot) ,~ n=1, 2, \ldots, p\Big\}.
\end{align}
 To estimate the  (strong)  convergence rate, we need somewhat stronger conditions compared with the convergence alone, which are stated as follows.

\begin{assumption}  \label{a6}
There exist functions $\kappa^2\in \mathcal{K}_{\wedge}$, $V\in  \mathcal{V}^{4}_{\delta_4}$ and
  $U\in \bar{\mathcal{V}}^{2}_{\delta_2}$ for a pair of  integers $1/\delta_4\in [4, +\infty)$ and $1/\delta_2\in [2, +\infty)$, as well as   positive constants  $a$, $q$, $\tau$  and $c_1$ such that $\Delta^{\tau}\leq \kappa (\Delta^{\frac{q}{2}})$ for any $\Delta\in (0, \Delta^*],$
\begin{align}\label{xeq1}
U(x)\leq \kappa(|x|^q)\leq V^{a}(x),\qquad U(x+y)\leq c_1(U(x)+U(y)) \qquad\forall~x,y\in \mathbb{K},
   \end{align}
for any compact subset  $\mathbb{K}\subset\mathbb{R}^d$. Moreover,  for some  $\iota>0 $, there exist  positive constants $\bar{K}$ and $r$ such that $r> 2(\delta_4/a- \delta_2),$
   \begin{align*}
 2U_x(x-y)\big(f(x)-f(y)\big)+(1+\iota)|U_{xx}(x-y)||g(x)-g(y)|^2
\leq \bar{K} U(x-y),
   \end{align*}
   and
   \begin{align*}
           & | f(x)-f(y)|\leq \bar{K}\big(1+U^r(x)+U^r(y)\big)U^{\delta_2}(x-y),\nn\\
&|g(x)-g(y)|^2 \leq \bar{K}\big(1+U^r(x)+U^r(y)\big)U^{ 2 \delta_2} (x-y),
   \end{align*}
   for any $x, y \in \RR^d$.
   \end{assumption}

\begin{remark}
One observes that if Assumption \ref{a6} holds, by Young's inequality
\begin{align}\label{cond-4}
|g(x)| \leq&  |g(x)-g(\mathbf{0})|+|g(\mathbf{0})|\nn\\
\leq& \sqrt{\bar{K}}
\big(U^{ 2 \delta_2}(x)+U^{r+ 2 \delta_2}(x) \big)^{\frac{1}{2}} +|g(\mathbf{0})|
\leq C\big(1+U^{\frac{r}{2}+ \delta_2} (x)\big),
\end{align}
\end{remark}
\begin{remark}
Let $\ell:=r+2\delta_2-2\delta_4/a$.
Under Assumption \ref{a6}, $\ell>0$ and we may define $\varphi(u)=C\big(1+\kappa^{\ell}(u^{q})\big)$  for any $u>0$
 in \eqref{e21}, then $\varphi^{-1}(u)=\big[\kappa^{-1}\big((u/C-1)^{\frac{1}{\ell}}\big)\big]^{\frac{1}{q}}$ for all $u> C$.
%
\end{remark}
Making use of scheme \eqref{Y_0}, we define an auxiliary approximation process by
\begin{align}\label{cond-7}
\bar{Y}(t)
&=Y_k+ f(Y_k)( t-t_k )+   g(Y_k)( B(t)-B(t_k) ), \qquad\forall t\in [t_k,  t_{k+1}).
\end{align}
Note that
$  \bar{Y}(t_k) =Y(t_k)=Y_k$, that is $\bar{Y}(t)$
and $Y(t)$ coincide with the discrete solution at the grid points.

\begin{lemma}\la{lemma+1}
  Let Assumption   \ref{a6} and the conditions of Theorem \ref{co3.1} hold. Then
   for any $T>0$ and $q_0\in (0,   2 \rho/a(r+2\delta_2)] $,  the process defined by (\ref{cond-7}) satisfies
  \be\la{cond-8}
  \E  |\bar{Y}(t)-Y(t)|^{q_0}
   \leq  C \t^{\frac{q_0}{2}},\qquad  \forall t\in [0, T],
  \ee  where $C$ is a positive constant independent of $\t$.
  \end{lemma}
\begin{proof}
For any $t\in [0, T]$, there is  an integer $k\geq 0$ such that $t\in [t_k, t_{k+1})$. Then,
 \begin{align*}
 \E\big(|\bar{Y}(t)-Y(t)|^{q_0} \big)
= &\E\big(|\bar{Y}(t)-Y(t_k)|^{q_0} \big)\nn\\
 \leq&  2^{q_0} \E\big( |  f(Y_k) |^{q_0} \big) \t^{q_0} + 2^{q_0}\E\big( |  g(Y_k)|^{q_0}  |B(t)-B(t_k)|^{q_0}\big) \\
\leq & C\lf( \E |  f(Y_k) |^{q_0}\t^{q_0} +  \E |  g(Y_k)|^{q_0} \t^{\frac{q_0}{2}} \rt).
  \end{align*}
 Due to   \eqref{Y_01}, \eqref{xeq1} and \eqref{cond-4} as well as  Theorem  \ref{co3.1},
  \begin{align*}
 \E\big(|\bar{Y}(t)-Y(t)|^{q_0} \big)
 &\leq C \E \big( 1+ V(Y_k)\big)^{q_0 \delta_4 }  \t^{q_0(1-\theta)} + C \E \big( 1+U^{\frac{r}{2}+ \delta_2} (Y_k)\big)^{q_0} \t^{\frac{q_0}{2}}\nonumber\\
  &\leq C\t^{\frac{q_0}{2}}\bigg[\E \big( 1+ V(Y_k)\big)^{a(\frac{r}{2}+ \delta_2 )q_0} +  \E \big( 1+U^{\frac{r}{2}+ \delta_2} (Y_k)\big)^{q_0}  \bigg]\nonumber\\
 &\leq C\t^{\frac{q_0}{2}} \Big[1+  \big(\E    V^{\rho}(Y_k)\big)^{\frac{a q_0(r+2\delta_2)}{ 2\rho}} + \big(\E    V^{\rho}(Y_k)\big)^{\frac{a q_0(r+2\delta_2)}{2\rho}} \Big]\nn\\
&\leq  C   \t^{\frac{q_0}{2} }.
 \end{align*}
 The required assertion follows.
\end{proof}

Using techniques in the proofs  of Theorem \ref{co3.1} and Lemma \ref{L:C_1}, we obtain   the following lemmas.
\begin{lemma}\la{lemma+2}
Under the conditions of Theorem \ref{co3.1},  the    process defined by (\ref{cond-7}) has the property that
\begin{align}\la{cond-9}
\sup_{\t\in(0,\t^*]}\sup_{0\leq t\leq T} \E V^{\rho}(\bar{Y}(t))
 \leq  C,\qquad\forall~T>0,
\end{align}
 where $C$ is  a positive constant independent of   the iteration order  $k$ and the time stepsize $\t$.
  \end{lemma}

\begin{lemma}\la{lemma+3} Under the conditions of Theorem \ref{co3.1}, for any $\t\in (0,\t^*]$, define
 \be\la{cond-10}
\bar{\eta}_{\t} := \inf \big\{ t\geq 0: | {\bar{Y}}(t)|\geq  \varphi^{-1}\big(K\t^{-\theta}\big)\big\}.
 \ee
Then for any $T>0$,
  \be\la{cond-11}
 \mathbb{E}V^{\rho}({\bar{Y}}(T\wedge \bar{\eta}_{\t}))\leq C,
  \ee
 where $C$ is a positive constant independent of $\t$.
  \end{lemma}

\begin{theorem}\la{lemma+4}
Let Assumption   \ref{a6} hold. Assume that
the conditions of Theorem \ref{co3.1} hold
 for $\rho>a$ and
  $\rho\geq  (2ar/\delta_2)\vee a q(r/2+\delta_2)$. If
  $\tau\in (0, \theta(\rho-a)/a\ell]$,
then  the process defined by (\ref{cond-7})  has the property that
\begin{align*}
\E U(\bar{Y}(T)-X(T))\leq C \kappa(C\t^{\frac{q}{2}}),\qquad\forall ~T>0.
\end{align*}
\end{theorem}
\begin{proof}
Define $\bar{\theta}_{ \t}=\tau_{\varphi^{-1}(K\t^{-\theta})} \wedge \eta_{\t}\wedge \bar{\eta}_{\t}$, $  \Omega_1:= \{\omega:~\bar{\theta}_{  \t} > T\}$, $\bar{e}(t)= \bar{Y}(t)-X(t),$ for  any $t\in[0, T]$,
where  $\tau_N$, $\eta_{\t}$ and $\bar{\eta}_{\t}$ are defined by (\ref{tau_N}), (\ref{3.18}) and (\ref{cond-10}),  respectively.
Using the Young inequality,   we have
\begin{align}\la{cond-13}
\E U(\bar{e}(T)) =& \E\lf[U(\bar{e}(T)) I_{\Omega_1}\rt]
+ \E\lf[U(\bar{e}(T)) I_{\Omega_1^c}\rt] \nonumber \\
 \leq &\E\lf[U(\bar{e}(T)) I_{\Omega_1}\rt]
+ \frac{a\kappa(\t^{\frac{q}{2}})}{\rho}\E\lf[U^{\rho/a}(\bar{e} (T)) \rt]+ \frac{\rho-a}{\rho\kappa^{\frac{a}{  \rho-a }}(\t^{\frac{q}{2}})} \PP(\Omega_1^c).
\end{align}
 It follows from the results of   Theorem  \ref{th1} and Lemma \ref{lemma+2} that
\begin{align}\label{cond-19}
\frac{a\kappa (\t^{\frac{q}{2}})}{\rho} \E  U^{\rho/a}\big(\bar{e}_{\t }(T)\big)
\leq \!  C \kappa (\t^{\frac{q}{2}})\Big[ \E V^{\rho}(X(T))
+\!   \E V^{\rho}(\bar{Y}(T)) \Big]
\leq \!   C \kappa (\t^{\frac{q}{2}}).
\end{align}
 Moreover,
by the virtue of Theorem \ref{th1},  Lemmas \ref{L:C_1} and \ref{lemma+3} we yield
\begin{align}\la{cond-20}
&\frac{\rho-a}{ \rho\big[\kappa(\t^{\frac{q}{2}})\big]^{\frac{a}{\rho-a}}} \PP(\Omega_1^c)\nn\\
\leq &  \frac{ \rho-a}{\rho\big[\kappa(\t^{\frac{q}{2}})\big]^{\frac{a}{\rho-a}}}  \Big( \PP {\{\tau_{\f^{-1} ( K\t^{-\theta} ) } \leq T\}} +\PP {\{\eta_{  \t} \leq T\}} + \PP {\{\bar{\eta}_{  \t} \leq T\}} \Big)\nonumber\\
\leq &  \frac{ \rho-a}{\rho\big[\kappa(\t^{\frac{q}{2}})\big]^{\frac{a}{\rho-a}}}
\frac{3C}{\big[\kappa(|\f^{-1}(K\t^{-\theta})|^{q})\big]^{\frac{\rho}{a}}}
\leq   C  \big[\kappa(\t^{\frac{q}{2}})\big]^{\frac{a}{a-\rho}}
\Big(K\t^{-\theta}/C-1\Big)^{-\frac{\rho}{a\ell} } \nn\\
\leq &   C \big[\kappa(\t^{\frac{q}{2}})\big]^{\frac{a}{a-\rho}+\frac{\theta\rho}{a\ell\tau}}
 \leq  C\kappa(\t^{\frac{q}{2}}).
\end{align}
 On the other hand, note that for any $t\in (0, T \wedge \bar{\theta}_{ \t}] $,
 $$
  {\bar{e}} (t)  = \int_0^t \Big(f(X(s))-f(Y(s))\Big)\mathrm{d}s+\int_0^t \Big(g(X(s))-g(Y(s))\Big)\mathrm{d}B(s).
$$
Using
 It\^{o}'s formula we obtain
\begin{align*}
&U\big({\bar{e}} (T \wedge \bar{\theta}_{ \t})\big)\nn\\
= & \frac{1}{2}\int_0^{T \wedge \bar{\theta}_{ \t}} \Big[2 U_{x}\big({\bar{e}} (s)\big)\Big( f( X(s))-f(Y(s)) \Big)\nn\\
&\qquad +\mathrm{tr}\Big[\Big(g(X(s))-g(Y(s))\Big)^TU_{xx}\big({\bar{e}} (s)\big)\Big(g(X(s))-g(Y(s))\Big)\Big] \mathrm{d}s\! +\! M(T \wedge \bar{\theta}_{ \t})\nn\\
\leq &\frac{1}{2}\int_0^{T \wedge \bar{\theta}_{ \t}} \Big[2 U_{x}\big({\bar{e}} (s)\big)\Big( f(X(s))-f(Y(s)) \Big) \nn\\
&\qquad \qquad \qquad  +|U_{xx}\big({\bar{e}} (s)\big)||g(X(s))-g(Y(s))|^2\Big] \mathrm{d}s
 + M(T \wedge \bar{\theta}_{ \t}),
\end{align*}
where $
M(t)= \int_0^t  U_{x}\big({\bar{e}} (s)\big) \big(g(X(s))-g(Y(s))\big)\mathrm{d}B(s)
$
 is a local martingale with initial value 0. This
implies
\begin{align}\la{cond-4.14}
\mathbb{E}U\big({\bar{e}} (T \wedge \bar{\theta}_{ \t})\big)
\leq& \frac{1}{2}\mathbb{E}\int_0^{T \wedge \bar{\theta}_{ \t}} \Big[2 U_{x}\big({\bar{e}} (s)\big)\Big( f(X(s))-f(Y(s)) \Big) \nn\\
&\qquad \qquad \qquad +|U_{xx}\big({\bar{e}} (s)\big)||g(X(s))-g(Y(s))|^2\Big] \mathrm{d}s .
\end{align}
Then an application of Young's  inequality together with Assumption \ref{a6} leads to
\begin{align*}
&2 U_{x}\big({\bar{e}} (s)\big)\Big( f(X(s))-f(Y(s)) \Big)
 +|U_{xx}\big({\bar{e}} (s)\big)||g(X(s))-g(Y(s))|^2\nn\\
=& 2 U_{x}\big({\bar{e}} (s)\big)\Big( f(X(s))-f(\bar{Y}(s))
+f(\bar{Y}(s))-f(Y(s)) \Big)\nn\\
 &+|U_{xx}\big({\bar{e}} (s)\big)||g(X(s))-g(\bar{Y}(s))+g(\bar{Y}(s))-g(Y(s))|^2\nn\\
\leq &2 U_{x}\big({\bar{e}} (s)\big)\Big( f(X(s))-f(\bar{Y}(s)) \Big)+2 U_{x}\big({\bar{e}} (s)\big)\Big( f(\bar{Y}(s))-f(Y(s)) \Big)\nn\\
&  +(1\!+\iota)|U_{xx}\big({\bar{e}} (s)\big)||g(X(s))\!-g(\bar{Y}(s))|^2
+\big(1\!+\frac{1}{\iota}\big)|U_{xx}\big({\bar{e}} (s)\big)||g(\bar{Y}(s))\!-g(Y(s))|^2\nn\\
\leq &C\Big[U\big({\bar{e}} (s)\big)
 \!+ U^{1-\! \delta_2} \big({\bar{e}} (s)\big) \big| f( \bar{Y}(s))\!-f(Y(s)) \big|
+\! U^{1-\! 2 \delta_2} \big({\bar{e}} (s)\big)|g(\bar{Y}(s))\!-g(Y(s))|^2\Big]\nn\\
\leq &C\Big[U\big({\bar{e}} (s)\big)
 +   \big| f(\bar{Y}(s))-f(Y(s)) \big|^{\frac{1}{\delta_2}}
+  |g(\bar{Y}(s))-g(Y(s))|^{\frac{1}{\delta_2}}\Big]\nn\\
\leq &C\Big[U\big({\bar{e}} (s)\big)
 +   \big(1+U^{r}(Y(s))+U^{r}(\bar{Y}(s))\big)^{\frac{1}{\delta_2}}U\big(\bar{Y}(s)-Y(s)\big)\nn\\
&\qquad +\big(1+U^{r}(Y(s))+U^{r}(\bar{Y}(s))\big)^{\frac{1}{2\delta_2}}U\big(\bar{Y}(s)-Y(s)\big)\Big]\nn\\
\leq & C\Big[U\big({\bar{e}} (s)\big)
 +   \big(1+U^{\frac{r}{\delta_2}}(Y(s))+U^{\frac{r}{\delta_2}}(\bar{Y}(s))\big)U\big(\bar{Y}(s)-Y(s)\big)\Big].
    \end{align*}
Inserting the above inequality into \eqref{cond-4.14} and applying H\"{o}lder's equality and Jensen's equality, and then Lemmas \ref{lemma+1} and \ref{lemma+2}, we have
\begin{align}\la{cond-23}
 &\E\Big[U\big({\bar{e}} (T \wedge \bar{\theta}_{ \t})\big) \Big] \nn\\
 \leq&  C  \int_0^{T}
 \E\Big[U\big( \bar{e}(s\wedge \bar{\theta}_{ \t})\big)\Big]  \mathrm{d}s\nn\\
&\qquad    +C\int_0^{T}\bigg[\mathbb{E}\Big(1+U^{\frac{r}{ \delta_2}}(Y(s))+U^{\frac{r}{ \delta_2}}(\bar{Y}(s))\Big)^2\bigg]^{\frac{1}{2}}
 \kappa\big(\mathbb{E}|Y(s)
-\bar{Y}(s)|^{q}\big) \mathrm{d}s \nn\\
 \leq&  C  \int_0^{T}
 \E\Big[U\big( \bar{e}(s\wedge \bar{\theta}_{ \t})\big)\Big] \mathrm{d}s
 +C\int_0^{T}\bigg[\Big(1+\big(\mathbb{E}V^{\rho}(Y(s))\big)^{\frac{2ar}{\rho\delta_2} }
\nn\\
&\qquad\quad +\big(\mathbb{E}V^{\rho}(\bar{Y}(s))\big)^{\frac{2ar}{\rho\delta_2} }\Big)^{\frac{1}{2}}\times
 \kappa\Big(\big(\mathbb{E}|Y(s)
-\bar{Y}(s)|^{\frac{2\rho}{a(r+2\delta_2)}}\big)^{\frac{a(r+2\delta_2)q}{2\rho}}\Big)\bigg]
 \mathrm{d}s \nn\\
  \leq&  C  \int_0^{T}
 \E\Big[U\big( \bar{e}(s\wedge \bar{\theta}_{ \t})\big)\Big] \mathrm{d}s
+C \kappa(C\t^{\frac{q}{2}}).
    \end{align}
Applying the Gronwall inequality, we yield that
\begin{align} \label{cond-16}
\E\Big[U({\bar{e}} (T)) I_{\Omega_1}\Big]\leq   \E\Big[U\big({\bar{e}} (T \wedge \bar{\theta}_{ \t})\big) \Big]
  \leq C \kappa(C\t^{\frac{q}{2}}).
\end{align}
Inserting (\ref{cond-19}), (\ref{cond-20}) and (\ref{cond-16}) into (\ref{cond-13}) yields the desired assertion.
\end{proof}

Next, we provide two remarks to demonstrate   \eqref{xeq1} of  Assumption   \ref{a6}.

\begin{remark} For any $x\in \mathbb{R}$,  let  $\kappa(|x|^2)=\log(1+|x|^2)$, for any $\Delta\in (0, 0.6]$  we have
$\Delta^{\tau}\leq \kappa (\Delta^{\frac{q}{2}})$ with $\tau\geq1.5$, $q=2$. Let  $V(x)=U(x)=\log(1+|x|^2)$, we have
\begin{align*}
|U_x(x)|=\frac{2|x|}{1+x^2}\leq 3\sqrt{\log(1+x^2)},\qquad |U_{xx}(x)|=\Big|\frac{2}{1+x^2}-\frac{4x^2}{(1+x^2)^2}\Big|\leq 3.
\end{align*}
Thus the above inequalities imply $a=1$, $\delta_2=1/2$,
  $U\in \bar{\mathcal{V}}^{2}_{1/2}$ and $V\in  \mathcal{V}^{4}_{\delta_4}$ with $\delta_4\in(0, 1/4]$. For any $x, y\in \mathbb{R}$, we have
  \begin{align*}
U(x+y)\!=\log\big(1\!+|x+y|^2\big)\leq\! 4\Big(\log\big(1+|x|^2\big)\!+\log\big(1+|y|^2\big)\Big)\!\leq 4\big(U(x)\!+U(y)\big),
\end{align*}
 then  \eqref{xeq1} holds for $c_1=4$.  Clear, for any $r>0$, we have $r> 2(\delta_4/a-\delta_2)$.
\end{remark}

\begin{remark}  For any $x, y\in \mathbb{R}^d$, let  $\kappa(|x|^q)=|x|^{\bar{q}}$, for any $\Delta\in (0, 1]$  we have
$\Delta^{\tau}\leq \kappa (\Delta^{\frac{q}{2}})$ with $\tau\geq\bar{q}/2$,  $q=2\bar{q}$. Let $V(x)=|x|^{2}$ and $U(x)=|x|^{\bar{q}}$ with  $\bar{q}\geq2$, we have
\begin{align*}
U(x+y) \leq 2^{\bar{q}-1}\big(U(x)+U(y)\big),\ \   |U_x(x)|=\big|\bar{q}|x|^{\bar{q}-2}x^T\big|=\bar{q}|x|^{\bar{q}-1}\leq \bar{q}^2\big(|x|^{\bar{q}}\big)^{1-\frac{1}{\bar{q}}},
\end{align*}
and
\begin{align*} |U_{xx}(x)|=\big|\bar{q}(\bar{q}-2)|x|^{\bar{q}-4}xx^T\!+\bar{q}|x|^{\bar{q}-2}\mathbb{I}_{d}\big|
\leq\bar{q}(\bar{q}-1)|x|^{\bar{q}-2}\leq \bar{q}^2\big(|x|^{\bar{q}}\big)^{1-\frac{2}{\bar{q}}}.
\end{align*}
Thus the above inequalities imply  $a=\bar{q}/2$, $\delta_2=1/\bar{q}$,
  $U\in \bar{\mathcal{V}}^{2}_{1/\bar{q}}$ and $V\in  \mathcal{V}^{4}_{\delta_4}$ with $\delta_4\in(0, 1/\bar{q}]$,   then  \eqref{xeq1} holds for $c_1=2^{\bar{q}-1}$.
  Clear, for any $r>0$, we have $r> 2(\delta_4/a-\delta_2)$. If choose $\rho=p/2, \delta_4= 1/\bar{q} $ and $p>\bar{q}$,
  it should be emphasized that under such circumstances,
we can get the same rate of convergence as the literature \cite{Li18}.
\end{remark}

\section{Asymptotic stability}\label{5s-b}
Since the stability is one of the major concerns in many applications, the easily implementable scheme preserving the underlying stability is desired eagerly. In this  section we  cite the stability criterion  of the exact solution for SDEs (see, e.g. \cite{Mao2002, Mao08}) and go a further step to  construct the new explicit scheme while keeping this   long-time property well.   Moreover, for the stability purpose, we assume furthermore in this section that
 SDE (\ref{e1})
with  drift and diffusion satisfying $f(x^{*})\equiv \mathbf{0}\in \mathbb{R}^{d},~ g(x^{*}) \equiv \mathbf{0}\in \mathbb{R}^{d\times m}$  for some
$x^{*}\in \mathbb{R}^{d}$,
 which implies    $X(t)\equiv x^{*}$  is  a {\it trivial solution} of SDE (\ref{e1}) with the initial  value  $x_0= x^{*}$. Without loss of generality, we assume in this section
that  $x^{*}=\mathbf{0}$, namely, $f(\mathbf{0})\equiv\mathbf{0}$ and $g(\mathbf{0})\equiv \mathbf{0}.$

\subsection{Stability of the exact solution}\label{a-ss}
 To begin  this subsection, we cite a stochastic version of the LaSalle theorem on almost sure stability, please see details in  \cite{Mao2002}.
\begin{lemma}[\!\cite{Mao2002}]\la{th4.1}
 Assume that for some $\rho>0$,   there are   functions  $V\in \mathcal{C}_\infty^{2} (\RR^d;   \bar{\RR}_+)$   and  $w\in \mathcal{C} (\RR^d;   \bar{\RR}_+)$  such that
  \begin{align}\la{w1}
{\mathcal{L}}V^{\rho}(x)
\leq  -  w(x), \qquad\forall x\in  \mathbb{R}^d.
 \end{align}
 Then, for every $x_0\in \mathbb{R}^d$,
 $\mathrm{Ker}(w):=\{x\in \mathbb{R}^d| w(x)=0\}\neq \emptyset$, and the solution  $X(t)$ of (\ref{e1}) has the property
$$
 \limsup_{t\rightarrow\infty} V^{\rho}(X(t))<\infty,\qquad \lim_{t\rightarrow \infty}d(X(t),\mathrm{Ker}(w))=0\qquad\mathrm{a.s.}
$$
Moreover,  if $w(x)=0$ iff $x=\mathbf{0}\in \RR^d$, then  $\lim_{t\rightarrow \infty}X(t)=\mathbf{0} \ \ \mathrm{a.s.}$
\end{lemma}

Based on the above lemma, we further estimate the  exponential convergence rate.
\begin{corollary}\la{coco4.1}
Assume that the conditions of Lemma \ref{th4.1} hold and, moreover, $w(x)\geq \mu V^{\rho}(x)$
 for all $x\in \mathbb{R}^d$, where $\mu$ is a positive constant.
 Then
 the solution $X(t)$ of SDE (\ref{e1})
 has the  properties that
\begin{align}\la{yhfF_0}
   \limsup_{t\rightarrow \infty}\frac{\log\big(\E V^{\rho}(X(t))\big)}{t} \leq  - \mu, \qquad\quad
\end{align}
and
\begin{align}\la{yhf0F_0}
   \limsup_{t\rightarrow \infty}\frac{\log\big(V(X(t))\big)}{t} \leq  -\frac{\mu}{\rho} \qquad a.s.
\end{align}
  \end{corollary}
\begin{proof}
For  any given   $\rho >0$,  using  It\^{o}'s formula (see  e.g., \cite[Theorem 6.4, p.36]{Mao08})  and $w(x)\geq \mu V^{\rho}(x)$  we  derive that
   $
 \mathbb{E} \big[\mathrm{e}^{ \mu  t}V^{\rho}(X(t))\big]
\leq  V^{\rho}(x_0).
$
  This implies
the desired assertion (\ref{yhfF_0}).
Moreover, by  the nonnegative semimartingale convergence theorem (see, e.g., \cite[Theorem 3.9, p.14]{Mao08}), we obtain
$
\limsup\limits_{t\rightarrow \infty} \mathrm{e}^{ \mu  t} V^{\rho}(X(t)) <\infty~a.s.
$ We can easily carry out the proof of this corollary and hence is omitted to avoid repetition.
\end{proof}

\subsection{
  Stability of numerical solutions}\label{a-ss2}

 The conditions  of Lemma \ref{th4.1} provide  us an opportunity to construct a more precise scheme approximating the underlying stability.   Recall that for an integer $1/\delta_4\in [4, +\infty)$, $\bar{\mathcal{V}}^{4}_{\delta_4}$  has been defined by \eqref{D**}.
 If   a Lyapunov function $V(\cdot) \in   \bar{\mathcal{V}}^{4}_{\delta_{4}}$ satisfying \eqref{w1} is found, the almost surely asymptotic property of the solution process follows from Lemma \ref{th4.1}.
 In order to approximate this property in this subsection, assume
\begin{align}\la{CA1}
 \sup_{0<|x|\leq u} \frac{ |f(x)|^2\vee |g(x)|^2 }{\Lambda_{\rho}(x)V^{ {2}{\delta_4}}(x)}< +\infty
 \end{align}
for any $u>0$, where $\Lambda_{\rho}(x):=1\wedge\big({w(x)}/{V^{\rho}(x)}\big)$. Under the above assumption we estimate the growth rate of $f$ and $g$.
Choose
a strictly increasing continuous function
$\bar{\f}: \RR_+\rightarrow \RR_+$ with  $\bar{\f}(u)\rightarrow \infty$
as $u\rightarrow \infty$ such that
 \begin{align}\la{ee21}
 \sup_{0<|x|\leq u}    \bigg[ \frac{ |f( x)|}{\Lambda_{\rho}^{1/2}(x) V^{\delta_4}(x)}
 \vee  \frac{ |g(x)|^2}{\Lambda_{\rho}(x)V^{2\delta_{4}}(x)} \bigg]\leq \bar{\f}(u), \qquad\forall~ u\geq1.
 \end{align}
 Due to \eqref{CA1},   $\bar{\f}$  is well defined as well as its inverse function $\bar{\f}^{-1}: [\bar{\f}(1),\infty)\rightarrow \RR_+$. Then choose   a pair of positive
constants $\t^*\in (0,1)$ and  $K$ such that
$
 K(\t^{*})^{-\bar{\theta}}\geq \bar{\f}(|x_0|\vee 1)
$
  holds for some   $0<\bar{\theta}<1/2$, where  $K$ is a positive constant  independent of  $\t$.
 For  the given $\t\in (0, \t^*]$,   define a truncation mapping $\bar{\pi}_{\t}:\RR^d\ra \RR^d　$ by
$
\bar{\pi}_\t(x)= \big(|x|\wedge
\bar{\f}^{-1} (K\t^{-\bar{\theta}} )\big) \frac{x}{|x|},
$
where  let $\frac{x}{|x|}=\mathbf{0}$ as $x=\mathbf{0}\in \mathbb{R}^d$.
 Next,
for any given stepsize $\t\in (0, \t^*]$, define the $V$-truncated EM  scheme by
 \begin{align}\la{YY_0}
\left\{
\begin{array}{ll}
Z_0=x_0,&\\
\tilde{Z}_{k+1}=Z_k+f(Z_k)\t +g(Z_k)\t B_k,& \\
Z_{k+1}=\bar{\pi}_\t(\tilde{Z}_{k+1}),&
\end{array}
\right.
\end{align}
for any integer $k\geq 0$, where $t_k=k \t$ and $\t B_k=B(t_{k+1})-B(t_{k})$. To obtain the continuous-time approximation, we define $Z(t)$ by $Z(t):=Z_k$ for any $t\in[t_k,  t_{k+1})$. Then the drift and diffusion terms have the
    property
\begin{align}\la{2Y_01}
  |f (Z_k) |^2  \leq   K^2\t^{-2\bar{\theta}}
  \Lambda_{\rho}(Z_k)V^{ 2 \delta_4}(Z_k) ,\quad   |g(Z_k)|^2 \leq  K\t^{-\bar{\theta}} \Lambda_{\rho}(Z_k)
  V^{ 2 \delta_4} (Z_k),
\end{align}
which implies  (\ref{Y_01}) holds.   Thus,   the conclusion of Theorem \ref{T:C_2} still holds for the new   scheme \eqref{YY_0} under the conditions of Lemma \ref{th4.1}. 
 Next we prepare the   discrete version of the nonnegative semimartingale convergence  theorem (see, e.g. \cite[Theorem 7, p.139]{sh9} or \cite[Lemma 3.7]{Lukasz17}).
\begin{lemma}\la{LeY1}
Consider a non-negative stochastic process $\{V_k\}$ with representation
$$
V_k=V_0+A_k-U_k+\mathcal{M}_k,
$$
where $\{A_k\}$ and $\{U_k\}$ are almost surely non-decreasing, predictable processes with $A_0=U_0=0$, and $\mathcal{M}_k$ is a local martingale adapted to
$\mathcal{F}_{t_{k}}$  with $\mathcal{M}_0=0$. Then
$$
\left\{\lim_{k\rightarrow \infty}A_k<\infty\right\}\subset\left\{\lim_{k\rightarrow \infty}U_k <\infty \right\}\cap\left\{ \lim_{k\rightarrow \infty}V_k<\infty~exists\right\}\quad a.s.
$$
  \end{lemma}

 The following  lemmas will play an important role  in the  proof  of the asymptotic stability of the numerical solutions.
\begin{lemma}\la{Le5.3}
If  $V\in  \mathcal{\bar{V}}^{4}_{\delta_4}$ for some integer
  $1/\delta_4\in  [4, +\infty)$, we have
\begin{align*} 
 V(\tilde{Z}_{k+1})
\leq& V(Z_{k})+\mathcal{L}V(Z_{k})\t
+C\Lambda_{\rho}(Z_k)  V(Z_k)  \t^{2(1-\bar{\theta})}\nn\\
& +\sum_{i=1}^{3}\mathcal{S}_{i}^{\t}V(Z_k)  + \tilde{\mathcal{J}}^{\t}V(Z_k),
\end{align*}
and $\mathbb{E}_{k}\big[\mathcal{S}_i^{\t}V(Z_k)\big]=0,$
\begin{align}\la{EM_0}
|\tilde{\mathcal{J}}^{\t}V(Z_k)|
\leq C \Lambda_{\rho}(Z_k)V(Z_k) \t^{2(1-\bar{\theta})}\! +\mathcal{A}_1^{\t}V(Z_k),\qquad   \mathbb{E}_{k}\big[\mathcal{A}_1^{\t}V(Z_k)\big]=0,
\end{align}
where $\tilde{\mathcal{J}}^{\t}V(\cdot) $ and $\mathcal{A}_1^{\t}V(\cdot)$ are defined by \eqref{EM_1} and \eqref{EM_2}.
 Moreover, we also have
\begin{align*}
|\mathcal{S}_{1}^{\t}V(Z_k)|^2 =| D^{(1)}V(Z_k)  g(  Z_k)|^2\t+\mathcal{H}_{1,1}^{\t}V(Z_k),
\qquad   \mathbb{E}_k\big[\mathcal{H}_{1,1}^{\t}V(Z_k)\big]=0,~~~
\end{align*}
and for   $i=1, 2, 3$,
\begin{align*}
|\mathcal{S}_{i}^{\t}V(Z_k)|^2 \leq C \Lambda_{\rho}(Z_k)V^2(Z_k) \t^{1-\bar{\theta}}+\mathcal{H}_{i,i}^{\t}V(Z_k),\qquad   \mathbb{E}_k\big[\mathcal{H}_{i,i}^{\t}V(Z_k)\big]=0,
\end{align*}
where $\mathcal{H}_{i,i}^{\t}V(\cdot)$ is defined by \eqref{H_11}, \eqref{H_22} and \eqref{H_33}, respectively, and we also have
\begin{align*}
  \mathcal{S}_{1}^{\t}V(Z_k)\mathcal{S}_{j}^{\t}V(Z_k) \geq- C  \Lambda_{\rho}(Z_k)V^2(Z_k) \t^{2(1-\bar{\theta})}+ \mathcal{H}_{1,j}^{\t}V(Z_k),
\end{align*}
and $\mathbb{E}_k\big[\mathcal{H}_{1,j}^{\t}V(Z_k)\big]=0$ for $j=2, 3$, where  $\mathcal{H}_{1,j}^{\t}V(\cdot)$ are defined by \eqref{H_12} and \eqref{H_13}, respectively.
  \end{lemma}

\begin{lemma}\la{Le5.4}
  If  $V\in  \mathcal{\bar{V}}^{4}_{\delta_4}$ for some integer $1/\delta_4\in [4, +\infty)$, we have
\begin{align*}
|\tilde{\mathcal{J}}^{\t}V(Z_k)|^2
 \leq C \Lambda_{\rho}(Z_k)  V^2(Y_k)   \t^{4(1-\bar{\theta})}+\mathcal{A}_2^{\t}V(Z_k),\qquad  \mathbb{E}_{k}\big[\mathcal{A}_2^{\t}V(Z_k)\big]=0,
  \end{align*}
and
\begin{align*}
|\tilde{\mathcal{J}}^{\t}V(Z_k)|^3
 \leq C \Lambda_{\rho}(Z_k) V^3(Z_k)  \t^{6(1-\bar{\theta})}+\mathcal{A}_3^{\t}V(Z_k),\qquad  \mathbb{E}_{k}\big[\mathcal{A}_3^{\t}V(Z_k)\big]=0,
\end{align*}
where $\mathcal{A}_2^{\t}V(\cdot)$ and $\mathcal{A}_3^{\t}V(\cdot)$ are defined by \eqref{M2} and \eqref{M3}.
 Moreover, we also have
\begin{align*}
\big|\mathcal{S}_{1}^{\t}V(Z_k)\big|^2|\tilde{\mathcal{J}}^{\t}V(Z_k)|
\leq  C \Lambda_{\rho}(Z_k)V^3(Z_k)  \t^{3(1-\bar{\theta})} +\mathcal{A}_4^{\t}V(Z_k)
,
\end{align*}
\begin{align*}
\big|\mathcal{S}_{2}^{\t}V(Z_k)\big|^2|\tilde{\mathcal{J}}^{\t}V(Z_k)|
\leq C \Lambda_{\rho}(Z_k)V^3(Z_k) \t^{4(1-\bar{\theta})}+\mathcal{A}_5^{\t}V(Z_k)
,
\end{align*}
and $\mathbb{E}_k\big[\mathcal{A}_4^{\t}V(Z_k)\big]=0,$ $ \mathbb{E}_k\big[\mathcal{A}_5^{\t}V(Z_k)\big]=0,$
\begin{align*}
\big|\mathcal{S}_{3}^{\t}V(Z_k)\big|^2|\tilde{\mathcal{J}}^{\t}V(Z_k)|
\leq C \Lambda_{\rho}(Z_k)V^3(Z_k) \t^{5(1-\bar{\theta})}+\mathcal{A}_6^{\t}V(Z_k)
,
\end{align*}
and $\mathbb{E}_k\big[\mathcal{A}_6^{\t}V(Z_k)\big]=0,$ where $\mathcal{A}_4^{\t}V(\cdot)$, $\mathcal{A}_5^{\t}V(\cdot)$  and $\mathcal{A}_6^{\t}V(\cdot)$ are defined by \eqref{M4}, \eqref{M5} and \eqref{M6}, respectively.
  \end{lemma}

The proofs of both above lemmas  can be found in  Appendix A.
  Now we analyze the asymptotic properties of the numerical solutions of the  $V$-truncated EM scheme \eqref{YY_0}. It makes use of Lemma \ref{Le5.3} and Lemma \ref{Le5.4} above.

\begin{theorem}\la{th4.2}    Let the conditions of Lemma \ref{th4.1} hold. Assume also that the function  $V\in   \bar{\mathcal{V}}^{4}_{\delta_4}$ for some integer
 $1/\delta_4\in [4, +\infty)$ with  the property
$$
 V(\epsilon x)\leq V(x)   \qquad\forall\, x \in \mathbb{R}^d,
 ~ 0<\epsilon\leq 1.$$
 Then
   there is a constant $\t^{**}\in (0, \t^*]$ such that for any $\t\in(0,\t^{**}]$ the corresponding $V$-truncated EM scheme   \eqref{YY_0}  has the property that
\begin{align}\la{2YHF_01}
\limsup_{k\rightarrow\infty} V^{\rho}(Z_k)<\infty,\qquad  \lim_{k\rightarrow \infty}d(Z_k,\mathrm{Ker}(w))=0 \qquad\mathrm{a.s.}
\end{align}
Moreover, if $w(x)=0$   iff $x=\mathbf{0}\in \mathbb{R}^d$, then for any $\t\in(0,\t^{**}]$,
\begin{align*} 
\lim_{k\rightarrow\infty} Z_k=\mathbf{0} \ \ \mathrm{a.s.}
\end{align*}
\end{theorem}
\begin{proof}
  For any $\rho>0$,  we deduce  by  virtue of
Lemma  \ref{Le5.3}   that
\begin{align}\la{4x1}
 V^{\rho}(\tilde{Z}_{k+1})
\leq  V^{\rho}(Z_{k}) \big(1+ \zeta_k\big)^{\rho},
\end{align}
where
$$
\zeta_k=V^{-1}(Z_{k})\bigg[\mathcal{L}V(Z_{k})\t
+C \Lambda_{\rho}(Z_k)
  V(Z_k) \t^{2(1-\bar{\theta})} +\sum_{i=1}^{3}\mathcal{S}_{i}^{\t}V(Z_k)+|\tilde{\mathcal{J}}^{\t}V(Z_k)|\bigg],
$$
  if $V(Z_{k})\neq0$, otherwise it is equal to $-1$.
  Clear,  $\zeta_k\geq-1$. Now we only prove the case that $0<\rho\leq 1$ and the proofs for other cases are similar. By the virtue of \cite[Inequality (3.12)]{Li18}, for any $0<\rho\leq 1$, we derive from  \eqref{4x1} that
\begin{align}\la{ys1}
 \!\!V^{\rho}(\tilde{Z}_{k+1})
\! \leq \! V^{\rho}(Z_{k})I_{\{V(Z_{k})\neq 0\}}(Z_{k}) \!\bigg(\! 1\!+\! \rho \zeta_k\! +\!\frac{\rho(\rho\!-1)}{2}
 \zeta_k^2
+\!\frac{\rho(\rho\!-1)(\rho\!-2)}{6} \zeta_k^3 \!\bigg).\!\!\!\!
\end{align}
It is easy to see that
\begin{align}\la{ys2}
& I_{\{V(Z_{k})\neq 0\}}(Z_k) \zeta_k\nn\\
 =& I_{\{V(Z_{k})\neq 0\}}(Z_k) V^{-1}(Z_{k})\Big[ \mathcal{L}V(Z_{k})\t+ C\Lambda_{\rho}(Z_k)
  V(Z_k) \t^{2(1-\bar{\theta})}\nn\\
  &\qquad\qquad\qquad\qquad\qquad\qquad+  \sum_{i=1}^{3}\mathcal{S}_{i}^{\t}V(Z_k)+|\tilde{\mathcal{J}}^{\t}V(Z_k)|\Big]\nn\\
\leq& I_{\{V(Z_{k})\neq 0\}}(Z_k) V^{-1}(Z_{k})  \mathcal{L}V(Z_{k})\t+ C\Lambda_{\rho}(Z_k)\t^{2(1-\bar{\theta})} +\mathcal{M}_{1}^{\t}V(Z_k),
\end{align}
where
\begin{align*}
\mathcal{M}_{1}^{\t}V(Z_k)
 = I_{\{V(Z_{k})\neq 0\}}(Z_k) V^{-1}(Z_{k})\Big[  \sum_{i=1}^{3}\mathcal{S}_{i}^{\t}V(Z_k)+\mathcal{A}_{1}^{\t}V(Z_k)\Big],
\end{align*}
and by   Lemma \ref{Le5.3} we know
$
\mathbb{E}_k\big[\mathcal{M}_{1}^{\t}V(Z_k)\big]=0.
$
 We can now analyze the last two terms of the summation in \eqref{ys1}. First,   combining \eqref{fA.1} and  \eqref{AA_13}   we obtain
\begin{align*}
&\mathcal{S}_{1}^{\t}V(Z_k) |\tilde{\mathcal{J}}^{\t}V(Z_k)|\nn\\
  =&C \mathcal{S}_{1}^{\t}V(Z_k)\Lambda_{\rho}(Z_k)V(Z_k) |\t B_k|^{4}\t^{-2\bar{\theta}}\nn\\
  &+ C \Lambda_{\rho}(Z_k)V(Z_k) \Big(\langle D^{(1)}V(Z_k), g(  Z_k)\t B_k\rangle
|\t B_k|^{\frac{1}{\delta_4}}  \t^{\frac{-\bar{\theta}}{2\delta_4}}\Big)\nn\\
\geq&C \mathcal{S}_{1}^{\t}V(Z_k)\Lambda_{\rho}(Z_k)V(Z_k)|\t B_k|^{4}\t^{-2\bar{\theta}}-C \Lambda_{\rho}(Z_k)V^2(Z_k) \t^{\frac{(\delta_4+1)(1-\bar{\theta})}{2\delta_4}} \nn\\
&- C\t^{\frac{-\bar{\theta}}{2\delta_4}} \Lambda_{\rho}(Z_k)V(Z_k)  | D^{(1)}V(Z_k)| |g(  Z_k)|\Big(    |\t B_k|^{1+\frac{1}{\delta_4}}-K_{1+\frac{1}{\delta_4}}\t^{\frac{\delta_4+1}{2\delta_4}} \Big)
\end{align*}
for  $1/\delta_4\in [4, +\infty)$.  This together with
  Lemma \ref{Le5.3} implies  
\begin{align}\la{ys3}
&I_{\{V(Z_{k})\neq 0\}}(Z_k)\zeta_k^2 \\
\geq& I_{\{V(Z_{k})\neq 0\}}(Z_k) V^{-2}(Z_k)
 \bigg[2\mathcal{S}_{1}^{\t}V(Z_k)\Big(\mathcal{L}V(Z_{k})\t+ C\Lambda_{\rho}(Z_k)
  V(Z_k) \t^{2(1-\bar{\theta})}\Big) \nn\\
 &\qquad \qquad + 2\mathcal{S}_{1}^{\t}V(Z_k)\Big(\mathcal{S}_{2}^{\t}V(Z_k)
 +\mathcal{S}_{3}^{\t}V(Z_k)+ |\tilde{\mathcal{J}}^{\t}V(Z_k)|
\Big)+|\mathcal{S}_{1}^{\t}V(Z_k)|^2\bigg] \nn\\
\geq& I_{\{V(Z_{k})\neq 0\}}(Z_k) V^{-2}(Z_k)
 \bigg[ 2\mathcal{S}_{1}^{\t}V(Z_k)\Big(\mathcal{L}V(Z_{k})\t+ C\Lambda_{\rho}(Z_k)
  V(Z_k) \t^{2(1-\bar{\theta})} \Big)\nn\\
 &\qquad\qquad\qquad\qquad\qquad  +2\mathcal{S}_{1}^{\t}V(Z_k) |\tilde{\mathcal{J}}^{\t}V(Z_k)| -C \Lambda_{\rho}(Z_k)V^{2}(Z_k) \t^{2(1-\bar{\theta})}\nn\\
 &\qquad\qquad\qquad\qquad\qquad  +2\sum_{i=2}^{3}\mathcal{H}_{1,i}^{\t}V(Z_k)+| D^{(1)}V(Z_{k})  g(  Z_{k})|^2\t+\mathcal{H}_{1,1}^{\t}V(Z_k)
 \bigg]\nn\\
 \geq& I_{\{V(Z_{k})\neq 0\}}(Z_k) V^{-2}(Z_k) | D^{(1)}V(Z_{k})  g(  Z_{k})|^2\t
  - C  \Lambda_{\rho}(Z_k)\t^{2(1-\bar{\theta})}+\mathcal{M}_2^{\t}V(Z_k), \nn
\end{align}
where
\begin{align*}
&\mathcal{M}_2^{\t}V(Z_k)=I_{\{V(Z_{k})\neq 0\}}(Z_k) V^{-2}(Z_k)
 \bigg[\mathcal{H}_{1,1}^{\t}V(Z_k) +2\sum_{i=2}^{3}\mathcal{H}_{1i}^{\t}V(Z_k)\nn\\
 &\qquad\quad\quad+2\mathcal{S}_{1}^{\t}V(Z_k)\Big(\mathcal{L}V(Z_{k})\t+ C\Lambda_{\rho}(Z_k)
  V(Z_k) \big(\t^{2(1-\bar{\theta})}+\t^{-2\bar{\theta}} |\t B_k|^{4}\big)\Big) \nn\\
 &\qquad\quad\quad- C\t^{\frac{-\bar{\theta}}{2\delta_4}} \Lambda_{\rho}(Z_k)V(Z_k)  | D^{(1)}V(Z_k)| |g(  Z_k)|\Big(    |\t B_k|^{1+\frac{1}{\delta_4}}-K_{1+\frac{1}{\delta_4}}\t^{\frac{\delta_4+1}{2\delta_4}} \Big)  \bigg],
\end{align*}
and from Lemma \ref{Le5.3} and \eqref{E_k} we have
 $\mathbb{E}_k\big[\mathcal{M}^{\t}_2 V(Z_{k})\big]=0$.
  By \eqref{D**} and \eqref{2Y_01}, we deduce that
\begin{align*}
  \big|\mathcal{L}V(Z_{k})\big|\t=&\Big|\langle  D^{(1)}V(Z_{k}), f(  Z_{k})\rangle
 +\frac{1}{2} \hbox{tr}\Big[g^T( Z_{k})  D^{(2)}V(Z_{k}) g(  Z_{k})\Big]\Big|\t \nn\\
\leq&|D^{(1)}V(Z_{k})| |f(  Z_{k})|\t
 + | D^{(2)}V(Z_{k})| |g(  Z_{k})|^2\t \nn\\
 \leq&c K\t^{1-\bar{\theta}}\Big(\Lambda_{\rho}^{\frac{1}{2}}(Z_k) V(Z_{k})
 + \Lambda_{\rho}(Z_k)
  V(Z_k)\Big)
 \leq C\Lambda_{\rho}^{\frac{1}{2}}(Z_k) V (Z_k)
  \t^{1-\bar{\theta}}.
\end{align*}
This together with
  Lemmas \ref{Le5.3} and \ref{Le5.4} implies
\begin{align}\la{ys4}
&I_{\{V(Z_{k})\neq 0\}}(Z_k) \zeta_k^3 \\
\leq&I_{\{V(Z_{k})\neq 0\}}(Z_k) V^{-3}(Z_{k})
\bigg[C \Lambda_{\rho}^{\frac{1}{2}}(Z_k) V(Z_k)\t^{1-\bar{\theta}}
 +  \sum_{i=1}^{3}\mathcal{S}_{i}^{\t}V(Z_k)
 +|\tilde{\mathcal{J}}^{\t}V(Z_k)|\bigg]^3\nn\\
\leq& I_{\{V(Z_{k})\neq 0\}}(Z_k) V^{-3}(Z_{k})
\bigg\{ \! C\Lambda_{\rho}^{\frac{3}{2}}(Z_k)
  V^3(Z_k)  \t^{3(1\!-\bar{\theta})} \!+\!\Big[\sum_{i=1}^{3}\mathcal{S}_{i}^{\t}V(Z_k)\!+\!|\tilde{\mathcal{J}}^{\t}V(Z_k)|\Big]^3\nn\\
&\qquad\qquad\qquad\qquad\qquad+  C\Lambda_{\rho}(Z_k) V^2(Z_k)  \t^{2(1-\bar{\theta})} \Big[\sum_{i=1}^{3}\mathcal{S}_{i}^{\t}V(Z_k) +|\tilde{\mathcal{J}}^{\t}V(Z_k)|\Big]\nn\\
&\qquad\qquad\qquad\qquad\qquad+  C \Lambda_{\rho}^{\frac{1}{2}}(Z_k)V(Z_k)   \t^{1-\bar{\theta}}  \Big[\sum_{i=1}^{3}\mathcal{S}_{i}^{\t}V(Z_k)+|\tilde{\mathcal{J}}^{\t}V(Z_k)|\Big]^2\bigg\}\nn\\
\leq&I_{\{V(Z_{k})\neq 0\}}(Z_k) V^{-3}(Z_{k})
\bigg\{  C\Lambda_{\rho}(Z_k)
  V^3(Z_k)  \t^{3(1-\bar{\theta})}+\bigg(\sum_{i=1}^{3}\mathcal{S}_{i}^{\t}V(Z_k)\bigg)^3 \nn\\
&\qquad\qquad\qquad\qquad\qquad+C\big|\tilde{\mathcal{J}}^{\t}V(Z_k)\big|^3
+C\sum_{i=1}^{3}\big|\mathcal{S}_{i}^{\t}V(Z_k)\big|^2|\tilde{\mathcal{J}}^{\t}V(Z_k)|\nn\\
&\qquad\qquad\qquad\qquad +  C \Lambda_{\rho}(Z_k)V^2(Z_k)  \t^{2(1-\bar{\theta})} \bigg(\sum_{i=1}^{3}\mathcal{S}_{i}^{\t}V(Z_k)+|\tilde{\mathcal{J}}^{\t}V(Z_k)|\bigg)\nn\\
&\qquad\qquad\qquad\qquad +  C \Lambda^{\frac{1}{2}}_{\rho}(Z_k)V(Z_k)   \t^{1-\bar{\theta}}  \bigg(\sum_{i=1}^{3}\big|\mathcal{S}_{i}^{\t}V(Z_k)\big|^2
+\big|\tilde{\mathcal{J}}^{\t}V(Z_k)\big|^2\bigg)
\bigg\}\nn\\
\leq&C I_{\{V(Z_{k})\neq 0\}}(Z_k) V^{-3}(Z_{k})\Lambda_{\rho}(Z_k)
\bigg[
  V^3(Z_k)  \t^{3(1-\bar{\theta})} +   V^3(Z_k)  \t^{6(1-\bar{\theta})}
\nn\\
&\qquad\qquad+   V^3(Z_k)   \t^{3(1-\bar{\theta})}+  V^3(Z_k)  \t^{4(1-\bar{\theta})}+  V^3(Z_k)  \t^{5(1-\bar{\theta})}+  V^3(Z_k)  \t^{4(1-\bar{\theta})}\nn\\
&\qquad\qquad\qquad\qquad+  V^3(Z_k)  \t^{2(1-\bar{\theta})}
+  V^3(Z_k)   \t^{5(1-\bar{\theta})}\bigg]+\mathcal{M}^{\t}_3V(Z_{k})\nn\\
  \leq&C\Lambda_{\rho}(Z_k)
 \t^{2(1-\bar{\theta})} +\mathcal{M}^{\t}_3V(Z_{k}),\nn
\end{align}
where
\begin{align*}
&\mathcal{M}^{\t}_3V(Z_{k})\nn\\
=&I_{\{V(Z_{k})\neq 0\}}(Z_k) V^{-3}(Z_{k})
 \bigg\{
 C \Lambda_{\rho}(Z_k)V^2(Z_k)  \t^{2(1-\bar{\theta})} \bigg(\sum_{i=1}^{3}\mathcal{S}_{i}^{\t}V(Z_k)+\mathcal{A}_1^{\t}V(Z_k)\bigg)\nn\\
 &+  C \Lambda_{\rho}^{\frac{1}{2}}(Z_k)V(Z_k)   \t^{1-\bar{\theta}}  \bigg(\sum_{i=1}^{3} \mathcal{H}_{i,i}^{\t}V(Z_k) + \mathcal{A}_2^{\t}V(Z_k) \bigg)
+\bigg(\sum_{i=1}^{3}\mathcal{S}_{i}^{\t}V(Z_k)\bigg)^3 \nn\\
&\qquad\qquad\qquad\qquad\qquad +C\sum_{i=4}^{6} \mathcal{A}_{i}^{\t}V(Z_k)+  C\mathcal{A}_3^{\t}V(Z_k)\bigg\},
\end{align*}
and   from Lemmas \ref{Le5.3} and \ref{Le5.4} it is easy to see that $\mathbb{E}_k\big[\mathcal{M}^{\t}_3V(Z_{k})\big]=0$.
Thus,  for any integer $k\geq 0$,
 substituting (\ref{ys2})-(\ref{ys4}) into (\ref{ys1}), we deduce from (\ref{w1}) that
%
%
\begin{align}\la{**}
V^{\rho}(\tilde{Z}_{k+1})
\leq& I_{\{V(Z_{k})\neq 0\}}(Z_k)V^{\rho}(Z_{k})\bigg\{1+ C  \Lambda_{\rho}(Z_k) \t^{2(1-\bar{\theta})}\nn\\
 &       +\frac{\rho\t}{2}V^{-2}(Z_{k})\Big[ 2V(Z_{k})
\mathcal{L}V(Z_{k})-(1-\rho)| D^{(1)}V(Z_{k}) g(  Z_{k})|^2 \Big] \bigg\}+\mathcal{N}^{\t}_k\nn\\
=& I_{\{V(Z_{k})\neq 0\}}(Z_k)\bigg\{  V^{\rho}(Z_{k})+ C   \Lambda_{\rho}(Z_k) V^{\rho}(Z_{k})   \t^{2(1-\bar{\theta})}\nn\\
 &    +\frac{\rho\t}{2}V^{\rho-2}(Z_{k})\Big[ 2V(Z_{k})
\mathcal{L}V(Z_{k})-(1-\rho)| D^{(1)}V(Z_{k}) g(  Z_{k})|^2 \Big] \bigg\}+\mathcal{N}^{\t}_k\nn\\
\leq& I_{\{V(Z_{k})\neq 0\}}(Z_k)\Big(  V^{\rho}(Z_{k})+ C  w  (Z_k)  \t^{2(1-\bar{\theta})}-   w  (Z_k)   \t \Big)+\mathcal{N}^{\t}_k\nn\\
\leq&  V^{\rho}(Z_{k})
 -\big(1-C \t^{1-2\bar{\theta}}\big)w(Z_{k})\t +\mathcal{N}^{\t}_k,
\end{align}
where
\begin{align*}
\mathcal{N}^{\t}_k:=V^{\rho}(Z_{k})\big[\mathcal{M}^{\t}_1V(Z_{k})+\mathcal{M}^{\t}_2V(Z_{k})+\mathcal{M}^{\t}_3V(Z_{k})\big],
\end{align*}
and  we can see that $\mathbb{E}_k\big[\mathcal{N}^{\t}_k\big]=0$.
 Thus the required inequality \eqref{**} for the case $\rho>1$  can be proved similarly.
%
Furthermore,
$$V(\bar{\pi}_{\t}(x))\leq  V(x) \qquad\forall~x\in \mathbb{R}^d,$$
which implies that
 \begin{align*}
V^{\rho}(Z_{k+1})\leq V^{\rho}(\tilde{Z}_{k+1})
 \leq V^{\rho}(Z_{k})-\big(1-C \t^{1-2\bar{\theta}}\big)w(Z_{k})\t+ \mathcal{N}^{\t}_k.
\end{align*}
For any given $\gamma\in(0,1)$, we choose    $ {\t} ^{**}\in(0,\t^*] $ small  sufficiently
such that
\begin{align*}
 {\t} ^{**}\leq\big[(1-\gamma)/C\big]^{1/(1-2\bar{\theta})},  
\end{align*}
which implies that for all integer $k\geq1$ and any $ {\t}\in(0,\t^{**}] $,
\begin{align*}
V^{\rho}(Z_{k+1})
 \leq V^{\rho}(Z_{k})-\gamma w(Z_{k}) \t+ \mathcal{N}^{\t}_k,
\end{align*}
Then taking sum on $k$, we have
\begin{align*}
  V^{\rho}(Z_{k})
\leq&  V^{\rho}(Z_{0})
 -\gamma\sum_{i=0}^{k-1}w(Z_{i})\t +\mathcal{M}_k,
\end{align*}
where $\mathcal{M}_k=\sum_{i=0}^{k-1}\mathcal{N}^{\t}_i$. It is  easy to show that
\begin{eqnarray*}
 \mathbb{E}\left.\left[\mathcal{M}_{k}\right| \mathcal{F}_{{k-1}}\right] =
\mathcal{M}_{k-1}+\mathbb{E}\left[\mathcal{N}^{\t}_{k-1}| \mathcal{F}_{{k-1}}\right]=\mathcal{M}_{k-1}.
\end{eqnarray*}
This implies immediately that $\mathcal{M}_{k}$ is a  martingale with $\mathcal{M}_0=0$.
Set
$$
V_k:=V_0
 -\gamma\sum_{i=0}^{k-1}w(Z_{i})\t +\mathcal{M}_k
$$
with $V_0=V^{\rho}(Z_{0})$, and
$
 \sum_{i=0}^{k-1}w(Z_i)\t
$
is increasing in $k$.  It is easy to see that $V^{\rho}(Z_{k})\leq V_k$ a.s. Thus, by   Lemma \ref{LeY1}, we infer that
$$
\lim_{k\rightarrow\infty} V^{\rho}(Z_k)<\infty\quad\mathrm{a.s.}\qquad\mathrm{and}\qquad \sum_{i=0}^{\infty}w(Z_i)\t <\infty \ \ \mathrm{a.s.}
$$
which implies that
\begin{align}\la{0YHF_3}
\lim_{k\rightarrow\infty}w(Z_k)=0 \quad \mathrm{a.s.}\qquad
\mathrm{and}\qquad
\sup_{0\leq k< \infty} V^{\rho}(Z_k) <+\infty \ \ \ \mathrm{a.s.}
\end{align}
Define $\phi:\bar{\mathbb{R}}_+\rightarrow \bar{\mathbb{R}}_+$ by
$
\phi(r)=\inf\limits_{|x|\geq r}V^{\rho}(x)$ for $r\geq 0$.
 Clear, $\phi(|Z_k|)\leq V^{\rho}(Z_k)$. So
$$\sup_{0\leq k<\infty}\phi(|Z_k|)\leq \sup_{0\leq k<\infty}V^{\rho}(Z_k)<+\infty \ \ \ \mathrm{a.s.}$$
On the other hand, due to $ V(\cdot)\in \mathcal{C}_{\infty}^{4}(\mathbb{R}^{d}; \bar{\mathbb{R}}_+)
$  we know
$
\lim\limits_{r\rightarrow \infty}\phi(r)=\infty,
$
 which implies
\begin{align}\la{new69}
\sup_{0\leq k< \infty} |Z_k| <+\infty \ \ \ \mathrm{a.s.}
\end{align}
Moreover, when $w(x)=0$ iff $x=\mathbf{0}$ one concludes that
$$
\lim_{k\rightarrow\infty} Z_k =\mathbf{0}  \ \ \ \mathrm{a.s.}
$$
Observe, from \eqref{0YHF_3}  and \eqref{new69} that there is an $\Omega_0\subset \Omega$ with $\mathbb{P}(\Omega_0)=1$, such that
\begin{align}\la{0YHF_2}
\lim_{k\rightarrow\infty}w(Z_k(\omega))=0,\quad\sup_{0\leq k< \infty}|Z_k(\omega)|<+\infty \qquad\forall \omega \in \Omega_0.
\end{align}
We shall now show that
\begin{align}\la{0YHF}
\lim_{k\rightarrow \infty}d\left( Z_k( \omega) , \mathrm{Ker}(w)\right)=0 \qquad \forall \omega\in \Omega_0.
\end{align}
  If this is false, then there is some $\bar{\omega}\in \Omega_0$   such that
$$
\limsup_{k\rightarrow \infty}d\left(Z_k(\bar{\omega}) ,\mathrm{Ker}(w)\right)>0,
$$
whence there is a subsequence $\{Z_{k_{n}}(\bar{\omega})\}_{n\geq 1}$ of $\{ Z_k(\bar{\omega}) \}_{k\geq 0}$ such that
$$
d\left(Z_{k_{n}}(\bar{\omega}),\mathrm{Ker}(w)\right)\geq \epsilon, \quad n\geq 1,
$$
for some $\epsilon>0$. Since $\{ Z_{k_{n}}(\bar{\omega})\}_{n\geq 0}$ is bounded, we can find a subsequence $\{ Z_{k_{n_{j}}}\!(\bar{\omega}) \}_{j\geq 0}$ of $\{Z_{k_{n}}(\bar{\omega})\}_{n\geq 0}$ and
$$
\lim_{j\rightarrow\infty}Z_{k_{n_{j}}}(\bar{\omega}) =\hat{z},\qquad\hat{z}\in \mathbb{R}^d.
$$
Clearly, $\hat{z}\notin \mathrm{Ker}(w)$ so $w(\hat{z})>0$. However, by (\ref{0YHF_2}),
$$
w(\hat{z})=\lim_{j\rightarrow\infty}w\big(Z_{k_{n_{j}}}(\bar{\omega})\big)=0,
$$
which contradicts $w(\hat{z})>0$. Hence (\ref{0YHF}) must hold and the required assertion (\ref{2YHF_01}) follows since $\mathbb{P}(\Omega_0)=1$. Therefore, the proof is complete.
\end{proof}

\begin{remark}
If $V\in  \bar{\mathcal{V}}^{4}_{\delta_4}$ in Theorme \ref{th4.2} is replaced by $V\in   \bar{\mathcal{V}}^{p}_{\delta_p}\cap \{D^{(p+1)}V(\cdot)\equiv0\}$ for $p=2$ or $3$ and some integer  $1/\delta_p\in [p, +\infty)$, then the conclusion of Theorme \ref{th4.2} still holds. This implies $1/\delta_4$ may be less than 4, since Taylor's formula with  integral remainder  term $J(\tilde{Z}_{k+1}, Z_k)\equiv 0$  in the estimation of \eqref{4x1}.
\end{remark}

\begin{remark}
If $\tilde{V}\in  \bar{\mathcal{V}}^{4}_{\delta_4}$ in Theorme \ref{th4.2} is replaced by $\tilde{V}\in  \bar{\mathcal{V}}^{3}_{\delta_3}$ for some integer  $1/\delta_3\in [3, +\infty)$,  we choose
a strictly increasing continuous function
$\bar{\f}: \RR_+\rightarrow \RR_+$ with  $\bar{\f}(u)\rightarrow \infty$
as $u\rightarrow \infty$ such that
$$
 \sup_{0<|x|\leq u}   \bigg[ \frac{ |f( x)|}{\Lambda_{\rho}^{1/2}(x) \tilde{V}^{\delta_3}(x)}
 \vee  \frac{ |g(x)|^2}{\Lambda_{\rho}(x)\tilde{V}^{2 \delta_{3}}(x)} \bigg]\leq \bar{\f}(u),\qquad\forall~ u>0.
$$
Then for any given $0<\bar{\theta}< 1/3$ the corresponding  $V$-truncated EM scheme \eqref{YY_0} may reproduce the   LaSalle-type theorem.
  It turns out that the smoothness of $V(x)$ affects the construction of the   scheme (\ref{YY_0}).
\end{remark}

Next, we shall discuss the exponential stability of $V$-truncated EM scheme   \eqref{YY_0}. Here, the key idea is to show not only that the  quantity $V(Z_k)$ decays with time, but also that the decay is exponential.
\begin{corollary}\la{coco4.2}
Under the conditions of Theorem \ref{th4.2}, and $w(x)\geq \mu V^{\rho}(x)$
 for all $x\in \mathbb{R}^d$, where $\mu$ is a positive constant.
Then for any    $\varrho\in(0,  \mu )$ there is a constant $ {\t}^{**}\in(0,\t^*]$  such that  for any $\t\in(0,  {\t}^{**}]$ the $V$-truncated EM scheme   \eqref{YY_0} has the   property that
  \be\la{FyhfF_4}
 \limsup_{k\rightarrow \infty}\frac{\log\big(\E V^{\rho}(Z_{k})\big)}{k\t} \leq  -  \big(\mu-\varrho\big) <0,
   \ee
and
  \be\la{FyhfF_5}
  \limsup_{k\rightarrow \infty}\frac{\log\big(V(Z_{k})\big) }{k\t}
   \leq-\frac{\mu-\varrho}{\rho}\qquad a.s.
   \ee
  \end{corollary}
\begin{proof}
If $w(x)\geq \mu V^{\rho}(x)$ for some $\mu>0$ and $\rho>0$, then instead of  \eqref{**} one runs the same calculation to get
\begin{align*}
 \mathbb{E}_k\big[ V^{\rho}(\tilde{Z}_{k+1})\big]
\leq&  V^{\rho}(Z_{k})+CV^{\rho}(Z_{k})\t^{2(1-\bar{\theta})}-\mu V^{\rho}(Z_{k})\t.
\end{align*}
  For any  $\varrho\in(0, \mu)$,   choose    $ {\t} ^{**}\in(0,\t^*] $ sufficiently small
such that
$C \big({\t}^{**}\big)^{1-2\bar{\theta}} \leq\varrho {\t} ^{**}$,
${\t} ^{**}<  1/(\mu-\varrho)$.
 Taking expectations on both sides, then we have
\begin{align*}
 \mathbb{E}\big[ V^{\rho}(Z_{k+1})\big]
\leq \mathbb{E}\big[ V^{\rho}(\tilde{Z}_{k+1})\big]
\leq \Big(1-  (\mu-\varrho) \t   \Big) \mathbb{E}\big[ V^{\rho}(Z_{k})\big]
\end{align*}
for any $\t\in(0, {\t} ^{**}]$.   
Repeating this procedure we obtain
\begin{align*}
 \mathbb{E}\big[ V^{\rho}(Z_{k})\big]
\leq   \Big(1- (\mu-\varrho) \t   \Big)^k  V^{\rho}(x_{0})\leq   \exp\Big(- (\mu-\varrho)  k\t \Big) V^{\rho}(x_{0}),\ \ \forall k\geq 1,
\end{align*}
which implies  the required assertion (\ref{FyhfF_4}) holds.
Moreover, the other required assertion (\ref{FyhfF_5}) follows from the Chebyshev inequality and the Borel-Cantelli lemma (see, e.g., \cite[p.7]{Mao08}) directly,
please refer to \cite[p.600]{Higham1}.
\end{proof}

\begin{remark}   In the special case where $V^{\rho}(x)=|x|^{2\rho}$ with $\rho>0$ and $w(x)\geq \mu V^{\rho}(x)$
 for all $x\in \mathbb{R}^d$. We choose $\delta_4=1/2$, then $V\in  \bar{\mathcal{V}}^{2}_{\delta_4}\cap \{D^{(3)}V(\cdot)\equiv0\}$ and equation \eqref{ee21} become much simpler. That is to say,  we only need to choose a strictly increasing continuous function
$\bar{\f}: \bar{\RR}_+\rightarrow \bar{\RR}_+$ such that $\bar{\f}(u)\rightarrow \infty$
as $u\rightarrow \infty$ and
\begin{align*}
 \sup_{0<|x|\leq u}  \bigg\{  \frac{ |f(x)|}{|x|}
 \vee  \frac{ |g(x)|^2}{|x|^2} \bigg\}\leq \bar{\f}(u) \qquad\forall~u\geq 1.
 \end{align*}
\end{remark}


\subsection{Further results of numerical solutions}\label{a-ss3}
 In practice we always wish that the numerical solutions  will preserve the stability of the exact solution  perfectly for some given Lyapunov functions.
 Although most practically relevant Lyapunov functions can be found in the subset $\bar{\mathcal{V}}^{p}_{\delta_p}$ defined in \eqref{D**},  we may treat them as a special case.
 In this subsection, we will consider the following class of Lyapunov  functions $\hat{\mathcal{V}}^{p}_{\delta_p}$.
 It is sufficiently large so is rich enough for one to choose suitable Lyapunov functions for many of important SDEs
(see \cite{16,Lukasz17} for more details).
 To be precise,
 define
\begin{align*}
 \hat{\mathcal{V}}^{p}_{\delta_{p}}:=
 \mathcal{V}^{p}_{\delta_{p}}\cap
 \big\{\mathrm{Ker}(V)=\{\mathbf{0}\}\big\}.
\end{align*}
It is easy to see that
 $\hat{\mathcal{V}}^{p}_{\delta_{p}}$  includes almost all Lyapunov functions  presented in \cite{16,Lukasz17}.  In this  subsection,  we
assume there is a function
  $V(\cdot) \in   \hat{\mathcal{V}}^{4}_{\delta_4}$  such that
 \begin{align}\la{CA2}
\sup_{0<|x|\leq u} \frac{ (1+V(x))^{1- 2 \delta_4} \big(|f(x)|^2\vee |g(x)|^2\big) }{\Lambda_{\rho}(x)V(x)}< +\infty
 \end{align}
for any $u>0$. To define the truncation mapping for a super-linear diffusion and drift terms, we choose function $V  \in   \hat{\mathcal{V}}^{4}_{\delta_{4}}$ and a strictly increasing continuous function
$\hat{\f}: \bar{\RR}_+\rightarrow \bar{\RR}_+$ such that $\hat{\f}(u)\rightarrow \infty$
as $u\rightarrow \infty$ and
 \begin{align*}
 \sup_{0<|x|\leq u}  \Bigg\{ \frac{ (1+V(x))^{\frac{1}{2}- \delta_4}
 |f(x)|}{\big(\Lambda_{\rho}(x)V(x)\big)^{ {1}/{2}}}
 \vee  \frac{ (1+V(x))^{1- 2 \delta_4} |g(x)|^2}{\Lambda_{\rho}(x) V(x)} \Bigg\}\leq \hat{\f}(u)
 \end{align*}
for any $u\geq 1$ .
 Owing to   $f$ and $g$ satisfy \eqref{CA2},   the function $\hat{\f}$ can be well defined as well as its inverse function $\hat{\f}^{-1}: [\hat{\f}(1),\infty)\rightarrow \bar{\RR}_+$.
Then choose   a pair of positive
constants $\t^*\in (0,1)$ and  $K$ such that
 $ K(\t^*)^{-\hat{\theta}}\geq\hat{\f}(|x_0|\vee 1)
$
holds for some   $0<\hat{\theta}<1/2$, where  $ {K}$ is a positive constant  independent of the iteration order  $k$ and the time stepsize $\t$.
For  the given $\t\in (0, \t^*]$,   define a truncation mapping $\hat{\pi}_{\t}:\RR^d\ra \RR^d　$ by
  $
\hat{\pi}_\t(x)= \Big(|x|\wedge \hat{\f}^{-1}\big(K\t^{-\hat{\theta}}\big)\Big) \frac{x}{|x|},
$
where we let  $\frac{x}{|x|}=\mathbf{0}$ when $x=\mathbf{0}\in \mathbb{R}^d$ is a zero vector.
 Next,
for any given stepsize $\t\in (0, \t^*]$, define the $V$-truncated EM  scheme by
 \begin{align}\la{YYY_0}
\left\{
\begin{array}{ll}
Z_0=x_0,&\\
\tilde{Z}_{k+1}=Z_k+f(Z_k)\t +g(Z_k)\t B_k,& \\
Z_{k+1}=\hat{\pi}_\t(\tilde{Z}_{k+1}),&
\end{array}
\right.
\end{align}
for any integer $k\geq 0$, where $t_k=k \t$ and $\t B_k=B(t_{k+1})-B(t_{k})$. To obtain the continuous-time approximation, we define $Z(t)$ by $Z(t):=Z_k$ for any $t\in[t_k,  t_{k+1})$.
 Then the drift and diffusion terms have the
    property
\begin{align}\la{eqy5.31}
\begin{split}
  |f (Z_k) |^2  \leq  \frac{ \big(K\t^{-\hat{\theta}}\big)^2
  \Lambda_{\rho}(Z_k)V(Z_k)}{(1+V(Z_k))^{1-2\delta_4}} ,\qquad
  |g(Z_k)|^2 \leq  \frac{K\t^{-\hat{\theta}}\Lambda_{\rho}(Z_k)
  V(Z_k)}{(1+V(Z_k))^{1-2\delta_4}} .
\end{split}
\end{align}
which implies   (\ref{Y_01})  holds.  As the special  case
 $V  \in   \bar{\mathcal{V}}^{4}_{\delta_4}$,    all of the conclusions  in Section \ref{a-ss2} holds  for the   scheme \eqref{YYY_0} under the conditions of Lemma \ref{th4.1} and we don't mention them to avoid duplication.


\begin{theorem}\la{yTh4.2}    Assume the conditions of Theorem \ref{th4.2}  hold  for $\rho=1$ and $V\in   \hat{\mathcal{V}}^{4}_{\delta_4}$.
 Then
   there is a constant $\t^{**}\in (0, \t^*]$ such that for any $\t\in(0,\t^{**}]$ the $V$-truncated EM scheme   \eqref{YY_0}  has the property that
\begin{align*}
\limsup_{k\rightarrow\infty} V(Z_k)<\infty,\qquad \lim_{k\rightarrow \infty}d(Z_k,\mathrm{Ker}(w))=0\quad\mathrm{a.s.}
\end{align*}
Moreover, if $w(x)=0$   iff $x=\mathbf{0}\in \mathbb{R}^d$, then for any $\t\in(0,\t^{**}]$,
\begin{align*} 
\lim_{k\rightarrow\infty} Z_k=\mathbf{0},\ \ \mathrm{a.s.}
\end{align*}
\end{theorem}
\begin{proof}
For $\rho=1$, making use of the techniques in the proof of Lemma \ref{Le5.3} as well as \eqref{eqy5.31}, \eqref{w1}  yields
\begin{align}
 V(\tilde{Z}_{k+1})
\!\leq& V(Z_{k})\!+\!\mathcal{L}V(Z_{k})\t
\!+C\Lambda_{1}(Z_k)  V(Z_k)  \t^{2(1\!-\hat{\theta})} \!+\!\sum_{i=1}^{3}\mathcal{S}_{i}^{\t}V(Z_k) \! + \mathcal{A}_1^{\t}V(Z_k)\nn\\
\leq& V(Z_{k})-\big(1-C\t^{1-2\hat{\theta}}\big)w(Z_{k})\t
 +\sum_{i=1}^{3}\mathcal{S}_{i}^{\t}V(Z_k)  + \mathcal{A}_1^{\t}V(Z_k),
\end{align}
where $\mathbb{E}_k\big[\mathcal{S}_{i}^{\t}V(Z_k)\big]=0$ and $\mathbb{E}_k\big[\mathcal{A}_1^{\t}V(Z_k)\big]=0$. Using the same method as employed in the proof of Theorem \ref{th4.2}, we can easily carry out the proof of this theorem and hence is omitted to avoid repetition.
\end{proof}
Based on the above theorem, we further obtain the  following corollary.
\begin{corollary}\la{coco4.3}
Assume that the conditions of Theorem \ref{yTh4.2} hold and, moreover,  $w(x)\geq \mu V(x)$
 for all $x\in \mathbb{R}^d$, where $\mu$ is a positive constant.
Then for any    $\varrho\in(0,  \mu )$ there is a constant $ {\t}^{**}\in(0,\t^*]$  such that  for any $\t\in(0,  {\t}^{**}]$ the $V$-truncated EM scheme   \eqref{YY_0} has the   property that
\begin{align*}
 \limsup_{k\rightarrow \infty}\frac{\log\big(\E V(Z_{k})\big)}{k\t} \leq  -  \big(\mu-\varrho\big) <0,
\end{align*}
and
\begin{align*}
  \limsup_{k\rightarrow \infty}\frac{\log\big(V(Z_{k})\big) }{k\t}
   \leq-\big(\mu-\varrho\big)\quad a.s.
\end{align*}
  \end{corollary}

\section{Examples and simulations}\la{exa}

\begin{example}\la{expx4.0}
Consider the two-dimensional nonlinear SDE
\begin{align}\la{exp5.0}
\left\{
\begin{array}{ll}
\mathrm{d}X_1(t)= \big(-2X^3_1(t)-2X^2_2(t)X_1(t)\big) \mathrm{d}t+ 2\sqrt{2}\big(X^2_1(t)+X^2_2(t)\big) \mathrm{d}B^{(1)}(t), \\
\mathrm{d}X_2(t)=\big(-2X^2_1(t)X_2(t)-2X^3_2(t)\big) \mathrm{d}t+ 2\sqrt{2}\big(X^2_1(t)+X^2_2(t)\big) \mathrm{d}B^{(2)}(t),
\end{array}
\right.
\end{align}
with the initial value $x_0=(1,\sqrt{3})^T$, where $B(t)=(B^{(1)}(t), B^{(2)}(t))^T$ is a two-dimensional Brownian motion. Obviously, its drift and diffusion coefficients
\begin{align*}
f(x)=\left[
  \begin{array}{ccc}
    -2x^3_1 -2x^2_2 x_1 \\
   -2x^2_1 x_2 -2x^3_2 \\
  \end{array}
\right],\qquad g(x)=\left[
  \begin{array}{ccc}
    2\sqrt{2}\big(x^2_1 +x^2_2 \big) &0  \\
  0&   2\sqrt{2}\big(x^2_1 +x^2_2 \big) \\
  \end{array}
\right].
\end{align*}
are locally Lipschitz continuous for any $x\in \mathbb{R}^2$.
 In this case one can set the Lyapunov function to be    $V(x)=|x|^2,$ which is from a broader class $V\in  \bar{\mathcal{V}}^{2}_{1/2}$ with $0<\rho<1/4$. Then one observes that
\begin{align*}
\mathcal{L}V^{\rho}(x)
=\frac{\rho}{2}V^{\rho-2}(x)\big(24|x|^6+32(\rho-1)|x|^6\big)
=4\rho  \big(4\rho-1\big)|x|^{2\rho+2}=:-w(x).
\end{align*}
We choose $\rho=1/8$, then $w(x)=1/4|x|^{2\rho+2}$ and $w(x)=0$ iff $x=\mathbf{0}\in \mathbb{R}^2$.
 Note that in this case the exact solution  $X(t)$ of SDE \eqref{exp5.0} is not necessarily moment exponentially stable, but   the exact solution  $X(t)$ tends to  $\mathbf{0}$ almost surely, see Lemma \ref{th4.1}.
 To the best of our knowledge the  numerical methods in the literatures such as \cite{Guo2017,Li18,S7,Lukasz17} cannot treat this case. However, the performance of the $V$-truncated EM scheme (\ref{YY_0}) is very nice for this case, see Fig.  \ref{exp5_0fig1}.
%

\begin{figure}[!htp]
\vspace{-1mm}
\includegraphics[angle=0, height=4cm, width=8cm]{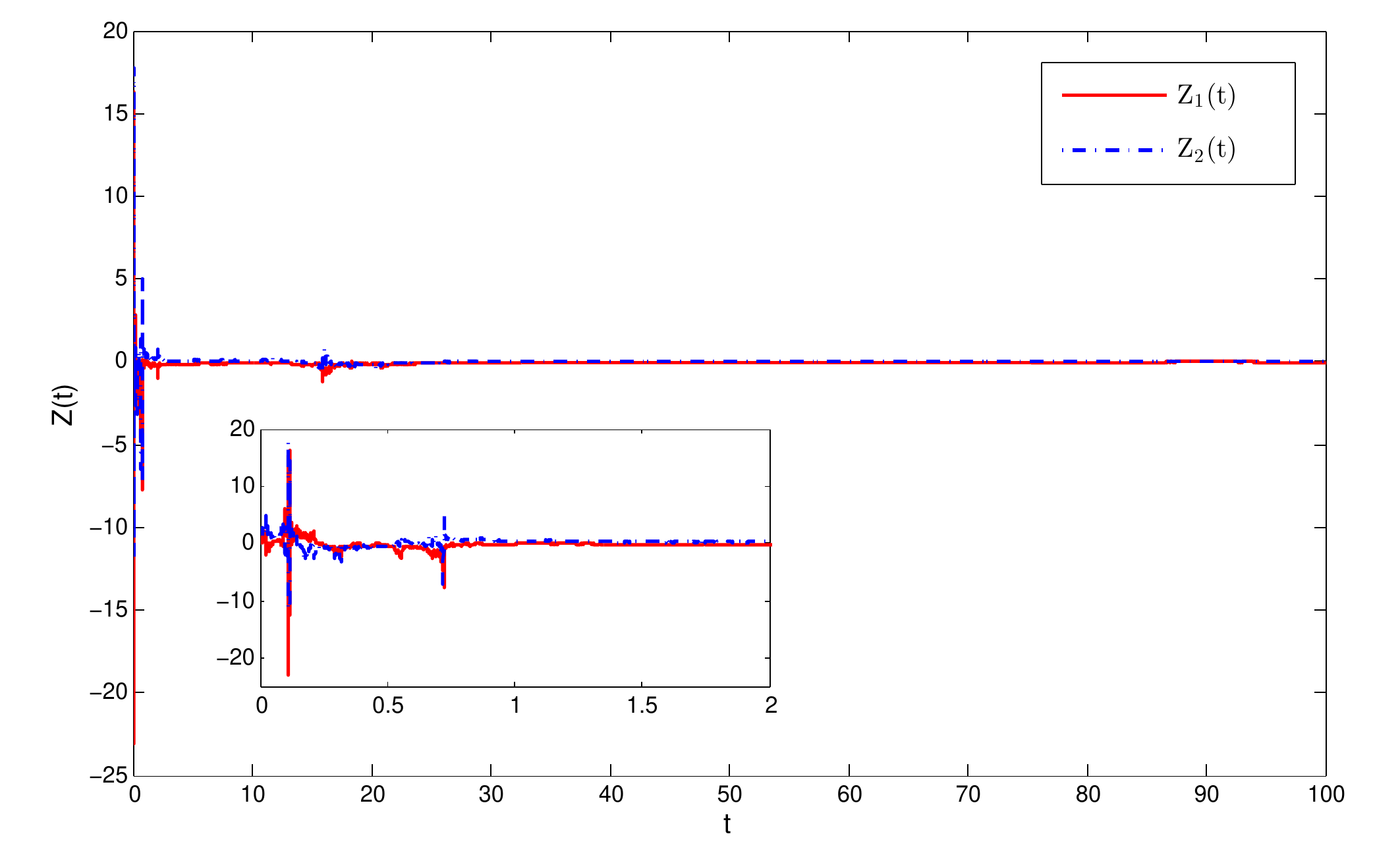}  
\caption{Sample path of  the $V$-truncated EM  numerical solution $Z(t)$ with the initial value $x_0=(1, \sqrt{3})^T$ for  stepsize  $\t=10^{-4}$ and $t\in [0, 100]$.}  \label{exp5_0fig1}
\vspace{-1mm}
\end{figure}

On the other hand,  in order to represent the simulations of the scheme (\ref{YY_0}), we divide it into three steps.
\\
{\bf Step 1.} Examine the hypothesis. Since that
\begin{align*}
 \sup_{0<|x|\leq u} \frac{ |f(x)|^2\vee |g(x)|^2 }{\Lambda_{\rho}(x)V^{2 \delta_4}(x)}= \sup_{0<|x|\leq u}\frac{ 4|x|^4\vee 16|x|^2 }{\big(1/4|x|^2\big)\wedge 1 }< +\infty,
 \end{align*}
which implies condition  \eqref{CA1} is satisfied. \\
{\bf Step 2.} Choose $\bar{\varphi}(\cdot)$ and $\bar{\theta}$. For any $x\in \mathbb{R}^2$, compute
\begin{align*}
  \sup_{0<|x|\leq u}    \bigg[ \frac{ |f( x)|}{\Lambda_{\rho}^{1/2}(x) V^{ \delta_4}(x)}
 \vee  \frac{ |g(x)|^2}{\Lambda_{\rho}(x)V^{2 \delta_{4}}(x)} \bigg]
=&\sup_{0<|x|\leq u}  \bigg[ \frac{4|x|^2}{ |x|\wedge 2  }
 \vee  \frac{ 64|x|^2}{  |x|^2\wedge 4 } \bigg]\nn\\
 \leq& 16(u+2)^2=:\bar{\varphi}(u) \qquad\forall~ u>0.
 \end{align*}
Then
$\bar{\varphi}^{-1}(u)=0.25\sqrt{u}-2$, $\forall u>64$. Fix a number $\t^*=10^{-4}$ and define  $K\t^{-\bar{\theta}}:=\bar{\varphi}(|x_0|\vee 1)\t^{-0.4}$  for any $\t\in(0,\t^*]$.
\\
 {\bf Step 3.} Construct an explicit scheme. For a fixed $\t\in(0,\t^{*}]$,   the $V$-truncated EM scheme for \eqref{exp5.0} is
\begin{align}
\left\{
\begin{array}{ll}
Z_0=x_0,&\\
\tilde{Z}_{k+1}=Z_k+ f(Z_k) \t + g(Z_k) \t B_k,& \\
Z_{k+1}=\Big[|\tilde{Z}_{k+1}|\wedge\big(4 \t^{-0.2} -2\big)\Big]\frac{\tilde{Z}_{k+1}}{|\tilde{Z}_{k+1}|}.
\end{array}
\right.
\end{align}
Therefore, by virtue of Theorem \ref{th4.2}, using this scheme we can preserves the underlying stability perfectly, see Fig.  \ref{exp5_0fig1}.
 \end{example}

\begin{example}\la{expx4.1}{\rm
Consider the scalar SDE
\begin{align} \la{xeq7.1}
 \mathrm{d}X(t)=\big(-0.5X(t)- X^3(t) \big)\mathrm{d}t+ X(t) \mathrm{d}B(t),
\end{align}
with the initial value $x_0=19$, where $B(t)$ is a scalar Brownian motion. Obviously, the drift and diffusion coefficients are locally Lipschitz continuous. In this case one can set the Lyapunov function to be   $V(x)=|x|^2,$  which is from a broader class $V\in  \bar{\mathcal{V}}^{2}_{1/2}$ with $0<\rho<1$. Then one observes that
\begin{align*}
\mathcal{L}V^{\rho}(x)=&\frac{\rho}{2}V^{\rho-2}(x)
\Big[2V(x)\mathcal{L}V(x)+(\rho-1) | D^{(1)}V(x)  g(x)|^2\Big]\nn\\
=&\frac{\rho}{2}V^{\rho-2}(x)\big(-4|x|^6+4(\rho-1)|x|^4\big)\leq -2\rho(1-\rho)V^{\rho}(x).
\end{align*}
  By virtue of Theorem \ref{th1} and Corollary \ref{coco4.1},  equation \eqref{xeq7.1} with any initial value $x_0>0$ has a unique regular solution $X(t)$, which is exponentially  stable for $0<\rho<1$. It can be verified that  condition  \eqref{CA1} holds. Note that for all $u > 0$,
\begin{align*}
\sup_{0<|x|\leq u}  \bigg\{ \frac{|f (x)|}{|x|}
 \vee  \frac{|g(x)|^2}{|x|^2} \bigg\}\leq\sup_{0<|x|\leq u}  \bigg\{ \frac{|x|+|x|^3}{|x|}
 \vee  \frac{|x|^2}{|x|^2} \bigg\}\leq u^2+1.
 \end{align*}
Take $\bar{\varphi}(u)= u^2+1$, $\forall u>0$, then
$\bar{\varphi}^{-1}(u)=\sqrt{ u-1}$, $\forall u>1$. Fix a number $\t^*=0.008$ and  define $K\t^{-\bar{\theta}}:=110\t^{-1/4}$ for any $\t\in(0,\t^*]$. For a fixed $\t\in(0,\t^{*}]$, the $V$-truncated EM scheme for \eqref{xeq7.1} is
\begin{align*}
\left\{
\begin{array}{ll}
Z_0=x_0,&\\
\tilde{Z}_{k+1}=Z_k-\big(0.5  Z_k +Z_k^3\big)\t + Z_k \t B_k,& \\
Z_{k+1}=\Big(|\tilde{Z}_{k+1}|\wedge\sqrt{  110\t^{-1/4}-1}\Big)\frac{\tilde{Z}_{k+1}}{|\tilde{Z}_{k+1}|}.
\end{array}
\right.
\end{align*}
Therefore, by virtue of Theorem \ref{T:C_2}, using this scheme we can approximate the exact solution
in the $p$th moment with $p<2\rho$.

\begin{figure}[!htp]
\includegraphics[angle=0, height=3.2cm, width=10cm]{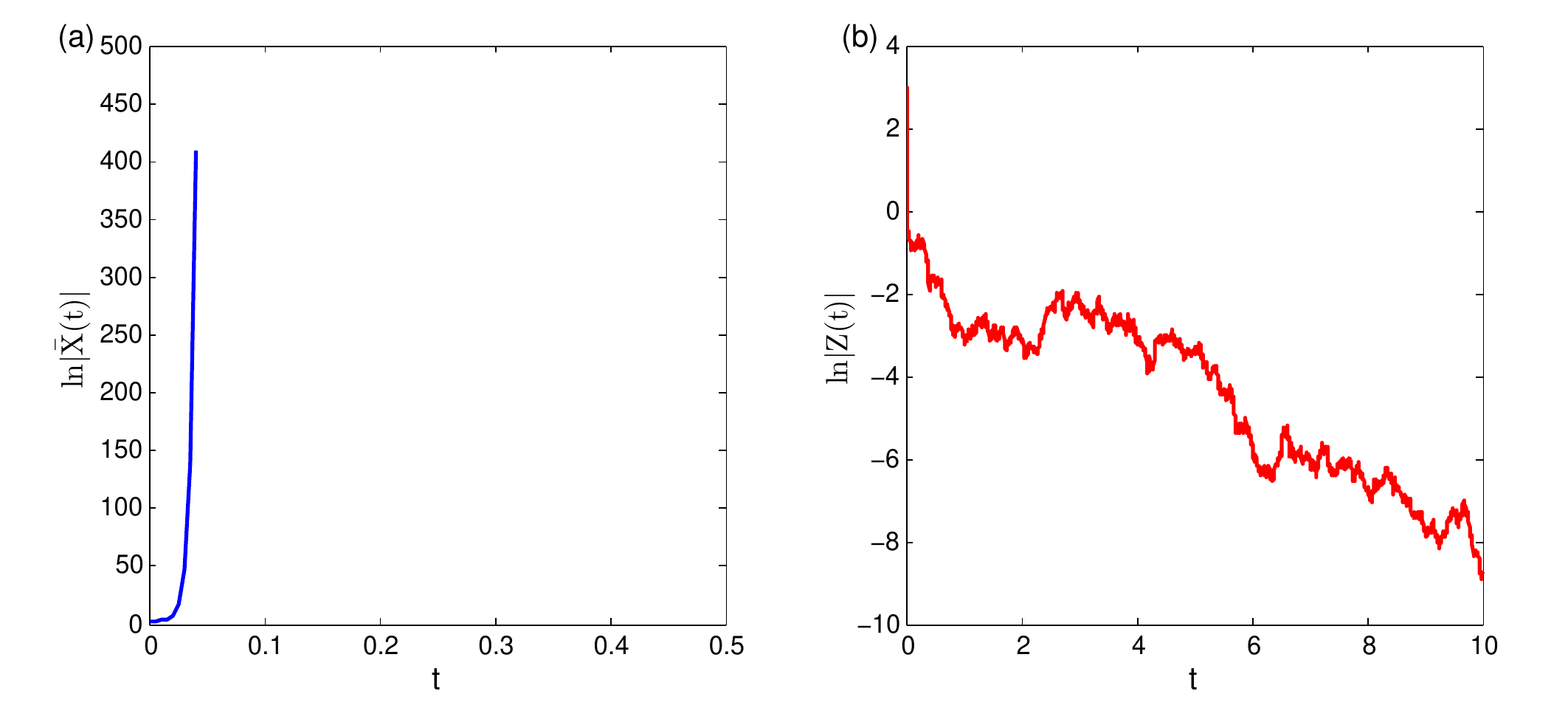}  
\caption{ (a) Sample path   of  the classical EM numerical solution $\ln|\bar{X}(t)|$, (b) Sample path of  the $V$-truncated EM numerical solution $\ln|Z(t)|$ with the same initial value $x_0=19$ for  stepsize  $\t=0.005$ and $t\in [0, 10]$.}  \label{exp5_1fig1}\vspace{-2mm}                                 
\end{figure}

As a fairly well-known result (see, e.g., \cite{Higham1}), the class EM numerical solution for a nonlinear
stable SDE  is unstable in a positive probability. However, Corollary \ref{coco4.2} reveals that the
$V$-truncated EM solution $Z(t)$  preserves the underlying stability perfectly, see Fig. \ref{exp5_1fig1}.

%
%

Fig. \ref{exp5_1fig1} gives sample path  of the classical EM solution $\ln|\bar{X}(t)|$ and of the $V$-truncated EM solution $\ln|Z(t)|$ with the same initial value
$x_0=19$ for  stepsize $\t=0.005$ and $t\in[0, 10]$. Fig. \ref{exp5_1fig1}(a) displays that the classical EM solution blows up quickly,
so it cannot capture the stability behavior of SDE \eqref{xeq7.1}. Fig. \ref{exp5_1fig1}(b) displays clearly that the truncated EM solution reproduces
the almost sure stability of SDE \eqref{xeq7.1}.
}\end{example}

\begin{example}\la{exp8.1}{\rm
Consider a stochastic Duffing-van der Pol oscillator (see \cite[p.81]{Hutzenthaler15})
\begin{align}\la{ex8.1}
\! \ddot{z}(t) \!+3{z}(t)+2\dot{z}(t)\!+2\dot{z}(t) z^2(t)\! + z^3(t) \!=\sqrt{2}z(t)\mathrm{d}B^{(1)}(t) \!+\sqrt{2.5}\dot{z}(t)\mathrm{d}B^{(2)}(t)
\end{align}
for $t\in \mathbb{R}_+$, where $B(t)=[B^{(1)}(t), B^{(2)}(t)]^T$. Introducing a new variable $(x_1, x_2)^T=(z, \dot{z})^T$,  we can write this Duffing-van der Pol oscillator as a two-dimensional SDE with drift and diffusion coefficients
\begin{align} \la{xeq8.1}
f(x)=\left[
  \begin{array}{ccc}
    x_2\\
   -3x_1-2 x_2-2 x_2x_1^2-x_1^3\\
  \end{array}
\right],\qquad g(x)=\left[
  \begin{array}{ccc}
    0 &0  \\
  \sqrt{2} x_1&   \sqrt{2.5}x_2 \\
  \end{array}
\right].
\end{align}
Obviously, its coefficients are locally Lipschitz continuous for any $x\in \mathbb{R}^2$. Therefore, Lemma \ref{th4.1} applies here with the Lyapunov-type function $V:\mathbb{R}^2\rightarrow \mathbb{R}_+$ given by
 $V(x)= x_1^4+ x_2^2+  x_1x_2+4 x_1^2, $
which is from a broader class $\hat{\mathcal{V}}^{4}_{1/4}$ and
\begin{align*}
\mathcal{L}V(x)
= &-4x_1^2x_2^2 - x_1^2-0.5x_2^2- x_1^4 \leq -0.5|x|^2=:-w(x).
\end{align*}
Note that in this case the solution is not necessarily $p$th moment exponentially stable for $p\geq 2$, but Lemma \ref{th4.1} still holds. We can then conclude that the SDE \eqref{ex8.1} has the property that
 $$
 \lim_{t\rightarrow \infty} [|z(t)|+|\dot{z}(t)|] =0\qquad \mathrm{a.s.}
$$
On the other hand,  in order to represent the simulations of the scheme (\ref{YY_0}), we divide it into three steps.
\\
{\bf Step 1.} Examine the hypothesis.
Since that $V(x)\geq w(x)$ for any $x\in \mathbb{R}^2$ and
\begin{align*}
\sup_{|x|\leq N}\big(|f(x)|^2\vee |g(x)|^2\big)
\leq   C_N w(x), \qquad\forall N>0,
\end{align*}
which implies condition  \eqref{CA2} is satisfied. \\
{\bf Step 2.} Choose $\hat{\varphi}(\cdot)$  and $\hat{\theta}$.  For any $x\in \mathbb{R}^2$, compute
\begin{align*}
 &\sup_{0<|x|\leq u}  \left[\frac{ \big(1+V(x)\big)^{1/4}|f(x)| }{  \sqrt{0.5}|x| }\vee
 \frac{\big(1+V(x)\big)^{1/2}|g(x)|^2}{ 0.5|x|^2  } \right]\nn\\
\leq& 5\big(8u^4+21\big)^{1/4}\big(36+16u^4)^{\frac{1}{2}}\leq   \big(36+16u^4)^{\frac{3}{4}}=:\hat{\varphi}(u), \qquad\forall~u\geq 1,
\end{align*}
which implies $\hat{\varphi}^{-1}(u)=0.5\big(u^{\frac{4}{3}}-36\big)^{\frac{1}{4}}~\forall~u\geq 20$. Let $K\t^{-\hat{\theta}}=\hat{\varphi}(|x_0|\vee 1)\t^{-0.4}$.\\
{\bf Step 3.} Construct an explicit scheme. For a fixed $\t\in (0,1)$, the $V$-truncated EM scheme for \eqref{ex8.1} is
\begin{align*}
\left\{
\begin{array}{ll}
Z_0=x_0,&\\
\tilde{Z}_{k+1}=Z_k+f(Z_k)\t +g(Z_k)\t B_k,& \\
Z_{k+1}=\Big(|\tilde{Z}_{k+1}|\wedge \big(|x_0|\vee 1\big)\t^{-0.4}\Big)\frac{\tilde{Z}_{k+1}}{|\tilde{Z}_{k+1}|}.
\end{array}
\right.
\end{align*}
By virtue of Theorem \ref{T:C_2}, using this scheme we can approximate the exact solution
in the mean square sense.
Moreover, by Theorem \ref{yTh4.2}, we can  conclude that the $V$-truncated EM scheme (\ref{YYY_0}) has the property that
 $$
 \lim_{k\rightarrow \infty} \big[|Z_{1,k}|+|Z_{2,k}|\big] =0\qquad\mathrm{a.s.}
$$
}\end{example}

\section{Concluding remarks}\la{concluding}
 This paper deals  with numerical solutions of nonlinear SDEs that are flexible enough to the stochastic Lyapunov method. We have constructed two explicit numerical schemes of SDEs in which their drift and diffusion coefficients are not globally Lipschitz but  grow faster than linearly.  A novelty of this paper, is to construct two explicit schemes in approximating the dynamical properties of SDEs with respect to a larger class of Lyapunov functions. By using one scheme in the  finite time interval, we obtained convergence and $V$-integrability of the numerical solutions under local Lipschitz condition and the structure conditions required by the exact solutions.  Moreover, the convergence rate (see Theorem \ref{lemma+4}) is also obtained under certain conditions which extends the results in the related literature. On the other hand, in the infinite time interval  we used the other scheme to produce the well-known LaSalle-type theorem of SDEs, from which we deduced the asymptotic stability of numerical solutions. Some simulation and examples are provided to support the theoretical results and demonstrate the validity of our approaches.

\section*{Appendix A.}\la{appendix}
In this appendix, we  will provide  the proofs of some results in Sections \ref{3s-c} and \ref{5s-b}.

{\bf  Proof of Lemma \ref{le1.3}.}~~First of all, note that $Y_k$ is $\mathcal{F}_{t_k}$-measurable, we have
\begin{align*}
&\sum_{|\alpha|=1}\frac{D^{\alpha} V (Y_k)}{\alpha!}\big(\tilde{Y}_{k+1}-Y_k\big)^{\alpha}
=\sum_{|\alpha|=1}\frac{D^{\alpha} V (Y_k)}{\alpha!}\big(f(Y_k)\t +g( Y_k)\t B_k\big)^{\alpha}\nn\\
=&\sum_{|\alpha|=1}\frac{D^{\alpha} V (Y_k)}{\alpha!}\big(
f(Y_k)\t\big)^{\alpha} +\sum_{|\beta|=1}\frac{D^{\beta} V (Y_k)}{\beta!}\big(g(Y_k)\t B_k\big)^{\beta}\nn\\
=& \langle   D^{(1)} V(Y_k), f(Y_k)\rangle \t+\mathcal{S}_{1}^{\t}V(Y_k),
\end{align*}
where \begin{align}\la{fA.1}
\dis\mathcal{S}_{1}^{\t}V(Y_k):=\langle D^{(1)}V(Y_k), g(Y_k)\t B_k\rangle.\tag{A.1}\end{align}  Moreover,
\begin{align*}
 &\sum_{|\alpha|=2}\frac{D^{\alpha} V (Y_k)}{\alpha!}\big(\tilde{Y}_{k+1}-Y_k\big)^{\alpha}\nn\\
=& \sum_{|\alpha|=2}\frac{D^{\alpha} V (Y_k)}{\alpha!}\big(
f( Y_k)\t\big)^{\alpha}+\sum_{|\alpha|=2,|\beta|=1 \atop \alpha\geq \beta}\frac{D^{\alpha} V (Y_k)}{\beta!(\alpha-\beta)!}\big(f( Y_k)\t\big)^{\alpha-\beta}\big(g( Y_k)\t B_k\big)^{\beta}\nn\\
  &+\sum_{|\beta|=2}\frac{D^{\beta} V (Y_k)}{\beta!} \big(g(Y_k)\t B_k\big)^{\beta} \nn\\
  =& \frac{1}{2}  \langle f(Y_k), D^{(2)}V(Y_k) f( Y_k)\rangle\t^2+\!\!\!\sum_{|\alpha|=2,|\beta|=1 \atop \alpha\geq \beta}\frac{D^{\alpha} V (Y_k)}{\beta!(\alpha-\beta)!}\big(f( Y_k)\t\big)^{\alpha-\beta}\big(g(Y_k)\t B_k\big)^{\beta}\nn\\
  &+\frac{1}{2} \hbox{tr}\Big[\big(g(Y_k)\t B_k\big)^TD^{(2)}V(Y_k)\big(g(Y_k)\t B_k\big)\Big] \nn\\
 =:& \frac{1}{2}  \langle f(Y_k), D^{(2)}V(Y_k) f( Y_k)\rangle\t^2+\frac{1}{2} \hbox{tr}\Big[g^T(Y_k)D^{(2)}V(Y_k)g( Y_k)\Big]\t +\mathcal{S}_{2}^{\t}V(Y_k)\nn\\
\leq& \left| D^{(2)}V(Y_k)\right| |f(Y_k)|^2\t^2+\frac{1}{2} \hbox{tr}\Big[g^T(Y_k)D^{(2)}V(Y_k)g(Y_k)\Big]\t +\mathcal{S}_{2}^{\t}V(Y_k) ,
\end{align*}
where
\begin{align}\la{fA.3}
\mathcal{S}_{2}^{\t}V(Y_k) :=&\sum_{|\alpha|=2,|\beta|=1\atop \alpha\geq \beta}\frac{D^{\alpha} V (Y_k)}{\beta!(\alpha-\beta)!}\big(
f(Y_k)\t\big)^{\alpha-\beta}\big(g(Y_k)\t B_k\big)^{\beta}\nn\\
  &+\frac{1}{2} \hbox{tr}\Big[D^{(2)}V(Y_k) g(Y_k)\big(\t B_k\t B_k^T-\mathbb{I}_m\t\big) g^T(Y_k) \Big],\tag{A.2}
\end{align}
and $\mathbb{I}_m$ denotes the $m\times m$ identity matrix. The fact that
$\t B_k$ is independent of $\mathcal{F}_{t_k}$ implies that
\begin{align}\la{L_2+}  \mathbb{E}_k\big[\triangle B_k\big]=\mathbb{E}\big[\triangle B_k|\mathcal{F}_{t_k}\big]=\mathbb{E}(\triangle B_k)=\mathbf{0}\in \mathbb{R}^m,\quad \mathbb{E}_k\big[\t B_k\t B_k^T\big]=\mathbb{I}_m \t.\tag{A.3}
\end{align}
Hence
$
\mathbb{E}_k\big[\mathcal{S}_{i}^{\t}V(Y_k)\big]=0
$ for $i=1, 2$.
We can now analyse the rest of the expansion for $|\alpha|=3$, we have
\begin{align}\la{A.4}
&\sum_{|\alpha|=3}\frac{D^{\alpha} V (Y_k)}{\alpha!}\big(\tilde{Y}_{k+1}-Y_k\big)^{\alpha} \tag{A.4}\\
=&\sum_{|\alpha|=3}\frac{D^{\alpha} V (Y_k)}{\alpha!} \big(f(Y_k) \big)^{\alpha} \t^3
   +\!\!\!\sum_{|\alpha|=3,|\beta|=1 \atop \alpha\geq \beta}\frac{D^{\alpha} V (Y_k)}{\beta!(\alpha-\beta)!}\big(f(Y_k) \big)^{\alpha-\beta} \big(g(Y_k)\t B_{k} \big)^{\beta}\t^2\nn\\
&\!\!+\!\!\!\!\sum_{|\alpha|=3,|\beta|=2 \atop \alpha\geq \beta}\frac{D^{\alpha} V (Y_k)}{\beta!(\alpha-\beta)!}\big(f(Y_k) \big)^{\alpha-\beta}  \big(g( Y_k)\t B_{k} \big)^{\beta} \t
 +\!\sum_{|\beta|=3}\frac{D^{\beta} V (Y_k)}{\beta!} \big(g(Y_k)\t B_{k} \big)^{\beta}\nn\\
 \leq
 &\!\!\sum_{|\alpha|=3,|\beta|=1 \atop \alpha\geq \beta}\frac{D^{\alpha} V (Y_k)}{\beta!(\alpha-\beta)!}\big(f(Y_k) \big)^{\alpha-\beta} \big(g( Y_k)\t B_{k} \big)^{\beta}\t^2
 +\!\!\sum_{|\beta|=3}\!\frac{D^{\beta} V (Y_k)}{\beta!} \big(g(Y_k)\t B_{k} \big)^{\beta}\nn\\
&+ C|D^{(3)} V (Y_k)|\Big( |f(Y_k)|^3\t^3 +|g(Y_k)|^2|f(Y_k)||\t B_{k}|^2 \t \Big)
.\nn
\end{align}
 Using  the properties
\begin{align}\la{E_k}
\mathbb{E}_k\big[|\t B_k|^{i}\big]=K_{i}\t^{\frac{i}{2}},\qquad i=1, 2, 3, \ldots\tag{A.5}
\end{align}
where $K_i$ is a positive constant  dependent on $i$, as well as \eqref{A.4} yields
\begin{align*}
&\sum_{|\alpha|=3}\frac{D^{\alpha} V (Y_k)}{\alpha!}\big(\tilde{Y}_{k+1}-Y_k\big)^{\alpha}\nn\\
 \leq&C|D^{(3)} V (Y_k)|\t^2\Big( |f(Y_k)|^3\t + |g(Y_k)|^2|f(Y_k)|\Big)+\mathcal{S}_{3}^{\t}V(Y_k),
\end{align*}
and $
\mathbb{E}_k\big[\mathcal{S}_{3}^{\t}V(Y_k)\big]=0$, where
\begin{align}\la{fA.5}
  \mathcal{S}_{3}^{\t}V(Y_k)
:= & \sum_{|\alpha|=3,|\beta|=1 \atop \alpha\geq \beta}\frac{D^{\alpha} V (Y_k)}{\beta!(\alpha-\beta)!}\big(f(Y_k) \big)^{\alpha-\beta} \big(g(  Y_k)\t B_{k} \big)^{\beta}\t^2\nn\\
& +C|D^{(3)} V (Y_k)||g(Y_k)|^2|f(Y_k)|\big(|\t B_{k}|^2 -K_2\t\big)\t\nn\\
& + \sum_{|\beta|=3} \frac{D^{\beta} V (Y_k)}{\beta!} \big(g(Y_k)\t B_{k} \big)^{\beta}. \tag{A.6}
\end{align}
The proof is therefore complete.\qed

{\bf  Proof of Lemma \ref{Z:1}.}~~First of all, note  that    $Y_{k}$  is $\F_{t_k}$-measurable.
 Then, using \eqref{1e1.2},   \eqref{Y_01} as well as  the estimates of $|D^{(n)}V(\cdot)|$,
  we derive that
\begin{align*}
\big| \mathcal{L}V(Y_{k})\big|
\leq&|D^{(1)}V(Y_{k})| |f( Y_{k})|
 + | D^{(2)}V(Y_{k})| |g( Y_{k})|^2 \nn\\
 \leq&c K\t^{-\theta}\big[1+V(Y_{k})\big]^{1- \delta_4}
 \big[1+V(Y_{k})\big]^{\delta_4} \nn\\
 &+cK\t^{-\theta}\big[1+V(Y_{k})\big]^{1- 2 \delta_4}
 \big[1+V(Y_{k})\big]^{ 2 \delta_4}
\leq 2cK  \big(1+V(Y_{k})\big)  \t^{-\theta},
\end{align*}
and
\begin{align*}
 \mathcal{R}^{\t}V(Y_k)=&C\sum_{i=2}^3\sum_{j=0}^{i-2}|f(  Y_k)|^{i-2j}|g( Y_k)|^{2j}|D^{(i)} V (Y_k)|\t^{i-j}\nn\\
\leq &C\big(1+V(Y_{k})\big) \Big[  \t^{2(1-\theta)} + \sum_{r=0}^{1} \t^{(3-r)(1-\theta)} \Big]
\leq   C\big(1+V(Y_{k})\big) \t^{2(1-\theta)}.
\end{align*}
Thus, the required assertion \eqref{lyhf1} is obtained.  Moreover,
\begin{align*} |\mathcal{S}_{1}^{\t}V(Y_k)|^2 =& \langle D^{(1)}V(Y_{k}), g(  Y_{k})\t B_k\rangle\langle D^{(1)}V(Y_{k}), g(  Y_{k})\t B_k\rangle \nn\\
=& \hbox{tr}\Big[  \big(D^{(1)}V(Y_{k})\big)^T \big(g( Y_{k})\t B_k\big) \big(g( Y_{k})\t B_k\big)^T D^{(1)}V(Y_{k})  \Big] \nn\\
=&| D^{(1)}V(Y_{k})  g( Y_{k})|^2\t+\mathcal{H}_{1,1}^{\t}V(Y_k),
\end{align*}
where
\begin{align}\la{H_11}
 \!\!\! \mathcal{H}_{1,1}^{\t}V(Y_k)\!=\!\hbox{tr}\Big[ \big(D^{(1)}V(Y_{k})\big)^T  g( Y_{k})\big(\t B_k\t B_k^T\!-\mathbb{I}_m \t\big) g^T( Y_{k}) D^{(1)}V(Y_{k}) \Big],\!\!\tag{A.7}
\end{align}
 Thus, we obtain that
\begin{align*}
&\mathbb{E}_k\big[\mathcal{H}_{1,1}^{\t}V(Y_k)\big]\nn\\
=&\mathbb{E}_k\bigg\{\hbox{tr}\Big[ \big(D^{(1)}V(Y_{k})\big)^T  g(  Y_{k})\big(\t B_k\t B_k^T-\mathbb{I}_m \t\big) g^T( Y_{k}) D^{(1)}V(Y_{k}) \Big]\bigg\}\nn\\
=&\hbox{tr}\bigg\{ \big(D^{(1)}V(Y_{k})\big)^T  g(  Y_{k})\big[\mathbb{E}_k\big( \t B_k\t B_k^T\big) -\mathbb{I}_m \t\big]g^T( Y_{k}) D^{(1)}V(Y_{k}) \bigg\}
= 0.
\end{align*}
Using    \eqref{Y_01} as well as  the estimates of $|D^{(n)}V(\cdot)|$, we have
$$\mathbb{E}_k\big[|\mathcal{S}_{1}^{\t}V(Y_k)|^2\big]\!\leq |D^{(1)}V(Y_{k})|^2| g(  Y_{k})|^2\t\leq C \big(1+V(Y_{k})\big)^2\t^{1-\theta}.$$
One further observes that
\begin{align*}
|\mathcal{S}_{2}^{\t}V(Y_k)|^2=&\bigg(\sum_{|\alpha|=2,|\beta|=1\atop \alpha\geq \beta}\frac{D^{\alpha} V (Y_k)}{\beta!(\alpha-\beta)!}\big(f(  Y_k)\t\big)^{\alpha-\beta}\big(g(  Y_k)\t B_k\big)^{\beta}\nn\\
  &\qquad   +\frac{1}{2} \hbox{tr}\Big[D^{(2)}V(Y_k) g( Y_k)\big(\t B_k\t B_k^T-\mathbb{I}_m\t\big) g^T( Y_k) \Big]\bigg)^2\nn\\
  \leq& C|D^{(2)}V(Y_{k})|^2 \Big( |f( Y_{k})|^2| g( Y_{k})|^2|\t B_k|^2\t^2\nn\\
  &\qquad \qquad \qquad \qquad +| g( Y_{k})|^4 |\t B_k|^4+| g( Y_{k})|^4 \t^2\Big)\nn\\
\leq& C|D^{(2)}V(Y_{k})|^2 \Big(|f( Y_{k})|^2| g( Y_{k})|^2 \t^3 +| g( Y_{k})|^4 \t^2\Big)+\mathcal{H}_{2,2}^{\t}V(Y_k) ,
\end{align*}
where
\begin{align}\la{H_22}
 \mathcal{H}_{2,2}^{\t}V(Y_k)
 = C|D^{(2)}V(Y_{k})|^2 \Big[&|f(  Y_{k})|^2| g(  Y_{k})|^2\t^2\big(|\t B_k|^2-K_2 \t\big)\nn\\
  &~+ | g( Y_{k})|^4\big(|\t B_k|^4-K_4\t^2\big) \Big] . \tag{A.8}
\end{align}
Using   \eqref{Y_01}, \eqref{E_k} as well as  the estimates of $|D^{(n)}V(\cdot)|$, we have
\begin{align*}
\mathbb{E}_k\big[|\mathcal{S}_{2}^{\t}V(Y_k)|^2\big]
    \leq&C\big(1+V(Y_{k})\big)^2\Big( \t^{2(1-\theta)}+ \t^{3(1-\theta)} \Big)+\mathbb{E}_k\big[\mathcal{H}_{2,2}^{\t}V(Y_k)\big]\nn\\
\leq&C  \big(1+V(Y_{k})\big)^2 \t^{2(1-\theta)}.
\end{align*}
Similarly, we can also prove that
\begin{align*}
&|\mathcal{S}_{3}^{\t}V(Y_k)|^2\nn\\
=&\bigg[ \sum_{|\alpha|=3,|\beta|=1 \atop \alpha\geq \beta}\frac{D^{\alpha} V (Y_k)}{\beta!(\alpha-\beta)!}\big(f( Y_k) \big)^{\alpha-\beta} \big(g(  Y_k)\t B_{k} \big)^{\beta}\t^2+C|D^{(3)}V(Y_{k})||g( Y_k)|^2\nn\\
&\qquad \qquad \times|f( Y_k)|\big(|\t B_{k}|^2\!-K_2\t\big)\t+ \sum_{|\beta|=3} \frac{D^{\beta} V (Y_k)}{\beta!} \big(g(Y_k)\t B_{k} \big)^{\beta}\bigg]^2\nn\\
&\leq C|D^{(3)}V(Y_{k})|^2\Big( |f( Y_{k})|^4| g( Y_{k})|^2  \t^5+  | g(  Y_{k})|^6\t^3\nn\\
&\qquad \qquad \qquad \qquad \qquad + |f( Y_{k})|^2| g(  Y_{k})|^4\t^4\Big)+\mathcal{H}_{3,3}^{\t}V(Y_k),
\end{align*}
where
\begin{align} \la{H_33}
\mathcal{H}_{3,3}^{\t}V(Y_k)
=& C|D^{(3)}V(Y_{k})|^2\bigg[| g( Y_{k})|^2|f(  Y_{k})|^4\big(|\t B_k|^2 -K_2\t\big)\t^4 \nn\\
 & + | g( Y_{k})|^6\big(|\t B_k|^6-K_6\t^3\big) +|f(  Y_{k})|^2| g(  Y_{k})|^4\big(|\t B_k|^4-K_4\t^2\big)\t^2   \bigg] .\tag{A.9}
\end{align}
Thus, we obtain that
$
\mathbb{E}_k\big[|\mathcal{S}_{3}^{\t}V(Y_k)|^2\big]
\leq C  \big(1+V(Y_{k})\big)^2 \t^{3(1- \theta)}.
$
Returning to \eqref{fA.1}, \eqref{fA.3} and using  \eqref{Y_01} as well as  the estimates of $|D^{(n)}V(\cdot)|$, we obtain
\begin{align*}
  &\mathcal{S}_{1}^{\t}V(Y_k)\mathcal{S}_{2}^{\t}V(Y_k)\nn\\
 =& \langle D^{(1)}V(Y_{k}), g( Y_{k})\t B_k\rangle \bigg\{\sum_{|\alpha|=2,|\beta|=1\atop \alpha\geq \beta}\frac{D^{\alpha} V (Y_k)}{\beta!(\alpha-\beta)!}\big(f(  Y_k)\t\big)^{\alpha-\beta}\big(g(  Y_k)\t B_k\big)^{\beta}\nn\\
  &\qquad \qquad \qquad \qquad \qquad   +\frac{1}{2} \hbox{tr}\Big[D^{(2)}V(Y_k) g( Y_k)\big(\t B_k\t B_k^T-\mathbb{I}_m\t\big) g^T( Y_k) \Big]\bigg\}
\nn\\
\geq&   - C  |D^{(1)}V(Y_{k})||D^{(2)}V(Y_{k})|  |f( Y_{k})| |g(  Y_{k})|^{2}\t^{2} +\mathcal{H}_{1,2}^{\t}V(Y_k)\nn\\
\geq&   -C \big(1+V(Y_{k})\big)^2\t^{2(1-\theta)} +\mathcal{H}_{1,2}^{\t}V(Y_k),
\end{align*}
where
\begin{align}\la{H_12}
  \mathcal{H}_{1,2}^{\t}V(Y_k)
&=- C  |D^{(1)}V(Y_{k})||D^{(2)}V(Y_{k})|  |f(  Y_{k})| |g( Y_{k})|^{2}\big(|\t B_k|^2-K_2\t\big)\t\nn\\
  &\qquad  +\frac{1}{2} \hbox{tr}\Big[D^{(2)}V(Y_k) g( Y_k)\big(\t B_k\t B_k^T\!-\mathbb{I}_m\t\big) g^T( Y_k) \Big]\mathcal{S}_{1}^{\t}V(Y_k).\! \tag{A.10}
\end{align}
Using \eqref{L_2+} and \eqref{E_k},  it is easy to see that
\begin{align*}
&\mathbb{E}_k\big[\mathcal{H}_{1,2}^{\t}V(Y_k)\big]\nn\\
   =&- C  |D^{(1)}V(Y_{k})||D^{(2)}V(Y_{k})|  |f( Y_{k})| |g(  Y_{k})|^{2}\Big[\mathbb{E}_k\big(|\t B_k|^2\big)-K_2\t\Big]\t
 =  0.
\end{align*}
Thus, we obtain that
\begin{align*}
  \mathbb{E}_k\big[\mathcal{S}_{1}^{\t}V(Y_k)\mathcal{S}_{2}^{\t}V(Y_k)\big]
\geq -C \big(1+V(Y_{k})\big)^2\t^{2(1-\theta)}.
\end{align*}
One further observes that
\begin{align*}
  &\mathcal{S}_{1}^{\t}V(Y_k)\mathcal{S}_{3}^{\t}V(Y_k)\nn\\
 =& \langle D^{(1)}V(Y_{k}), g( Y_{k})\t B_k\rangle \bigg[ \sum_{|\alpha|=3,|\beta|=1\atop \alpha\geq \beta}\frac{D^{\alpha} V (Y_k)}{\beta!(\alpha-\beta)!}\big(f( Y_k) \big)^{\alpha-\beta} \big(g(  Y_k)\t B_{k} \big)^{\beta}\t^2\nn\\
&\qquad \qquad \qquad \qquad \qquad \quad +C|D^{(3)}V(Y_{k})||g(  Y_k)|^2|f( Y_k)|\t\big(|\t B_{k}|^2 -K_2\t\big)\nn\\
&\qquad \qquad \qquad \qquad \qquad \quad + \sum_{|\beta|=3}\frac{D^{\beta} V (Y_k)}{\beta!} \big(g(  Y_k)\t B_{k} \big)^{\beta}\bigg]
\nn\\
\geq& - C  |D^{(1)}V(Y_{k})||D^{(3)}V(Y_{k})| \Big( |f( Y_{k})|^2 |g(  Y_{k})|^{2} \t^{3} + |g( Y_{k})|^{4}\t^2\Big)+\mathcal{H}_{1,3}^{\t}V(Y_k),
\end{align*}
where
\begin{align}\la{H_13}
 \mathcal{H}_{1,3}^{\t}V(Y_k)
=&-C  |D^{(1)}V(Y_{k})||D^{(3)}V(Y_{k})| \Big[ |f(  Y_{k})|^2 |g(  Y_{k})|^{2}\big(|\t B_{k}|^2-K_2\t\big)\t^{2}\nn\\
  &\qquad \qquad \qquad \qquad\qquad \qquad   +|g(  Y_{k})|^{4}\big(|\t B_{k}|^4-K_4\t^2\big)\Big]\nn\\
 &   +C|D^{(3)}V(Y_{k})||g( Y_k)|^2|f(  Y_k)|\t\big(|\t B_{k}|^2 -K_2\t\big)
\mathcal{S}_{1}^{\t}V(Y_k).\tag{A.11}
\end{align}
Using   \eqref{Y_01}, \eqref{L_2+} and \eqref{E_k} as well as  the estimates of $|D^{(n)}V(\cdot)|$, it is easy to see that
\begin{align*}
 \mathbb{E}_k\big[\mathcal{S}_{1}^{\t}V(Y_k)\mathcal{S}_{3}^{\t}V(Y_k)\big]
\geq& -C \big(1+V(Y_{k})\big)^2\t^{2(1- \theta)}.
\end{align*}
 Therefore the desired result follows. \qed

{\bf  Proof of Lemma \ref{Le5.3}.}~~Due to $V\in  \mathcal{\bar{V}}^{4}_{\delta_4}$, using the    Taylor formula with integral remainder term we get
\begin{align}\la{yhf*}
V(\tilde{Z}_{k+1})
=&V(Z_{k})+\sum_{|\alpha|=1}^{3}\frac{ D^{\alpha} V (Z_k) }{\alpha!} \big(\tilde{Z}_{k+1}-Z_k\big)^{\alpha}+J(\tilde{Z}_{k+1}, Z_k).\tag{A.12}
\end{align}
One observes that
\begin{align*}
 \big| J(\tilde{Z}_{k+1}, Z_k)\big|
\leq& 4\sum_{|\alpha|=4}  \frac{\big|\big(\tilde{Z}_{k+1}- Z_k\big)^{\alpha}\big|}{\alpha!} \int_{0}^{1}(1-t)^{3} \big|
 D^{(4)} V\big(Z_k+t\big(\tilde{Z}_{k+1}- Z_k\big)\big)\big|\mathrm{d}t\nn\\
 \leq& \frac{c}{3!}\bigg(\sum_{i=1}^{d} \Big|f_{i}( Z_k)\t  +\sum_{j=1}^{m}g_{ij}(Z_k) \t B_{k}^{(j)}\Big|\bigg)^4 \nn\\
 &\qquad \qquad \qquad \times\int_{0}^{1}(1 -t)^{3} V^{1-  4 \delta_4} \big(Z_k+t\big(\tilde{Z}_{k+1} - Z_k\big)\big) \mathrm{d}t.
\end{align*}
Note that for any  $U\in \mathcal{\bar{V}}^{4}_{\delta_4}$ we know
$|D^{(1)} U(x)  |\leq c  U^{1-\delta_4}(x)$. By the result of \cite[Lemma 2.12, p.22]{Hutzenthaler15} we have
\begin{align*}
 U(x+y)\leq c^{\frac{1}{\delta_4}}2^{\frac{1}{\delta_4}-1}\Big(\big|U(x)\big| +|y|^{\frac{1}{\delta_4}}\Big),\qquad\forall x,y\in \mathbb{R}^d,
\end{align*}
which leads to
\begin{align*}
 \big[V\big(Z_k+t\big(\tilde{Z}_{k+1}- Z_k\big)\big)\big]^{1- 4 \delta_4}  \leq &\Big[ c^{\frac{1}{\delta_4}}2^{\frac{1}{\delta_4}-1}\Big( V(Z_k)+  t^{\frac{1}{\delta_4}} |\tilde{Z}_{k+1}- Z_k|^{\frac{1}{\delta_4}}\Big)\Big]^{1- 4 \delta_4}\nn\\
 \leq & C\Big[ V^{1- 4 \delta_4}(Z_k)+ |\tilde{Z}_{k+1}- Z_k|^{\frac{1}{\delta_4} -4}\Big]
\end{align*}
for $1/\delta_4\in [4, +\infty)$. Therefore, we derive from \eqref{2Y_01} that for any integer $k\geq0$,
\begin{align}\la{EM_1}
\big|J(\tilde{Z}_{k+1}, Z_k)\big|
\leq&C \Big[\Big(|f( Z_k)|^{4}\t^{4}+|g( Z_k)|^{4}|\t B_k|^{4}\Big)V^{1-  4 \delta_4} (Z_k)\nn\\
&\qquad \qquad \quad + \Big(|f( Z_k)|^{\frac{1}{\delta_4}}\t^{\frac{1}{\delta_4}}+|g( Z_k)|^{\frac{1}{\delta_4}}|\t B_k|^{\frac{1}{\delta_4}}\Big)  \Big]\nn\\
\leq&C \Lambda_{\rho}(Z_k)\Big\{\Big[\t^{4(1-\bar{\theta})}+\t^{-2\bar{\theta}}|\t B_k|^{4}\Big]
V(Z_k)\nn\\
&\qquad \qquad \quad  + V(Z_k)\Big[\t^{\frac{1}{\delta_4}(1-\bar{\theta})}
+\t^{-\frac{\bar{\theta}}{2\delta_4}} |\t B_k|^{\frac{1}{\delta_4}}\Big]  \Big\}\nn\\
\leq& C \Lambda_{\rho}(Z_k)V(Z_k)\t^{4(1-\bar{\theta})} +\tilde{\mathcal{J}}^{\t}V(Z_k),\tag{A.13}
\end{align}
where
\begin{align}\la{AA_13}
\tilde{\mathcal{J}}^{\t}V(Z_k)=C\Lambda_{\rho}(Z_k)V(Z_k) \Big[\t^{-2\bar{\theta}}|\t B_k|^{4}
 + \t^{-\frac{\bar{\theta}}{2\delta_4}}|\t B_k|^{\frac{1}{\delta_4}}  \Big].\tag{A.14}
\end{align}
Using the similar techniques in the proofs of Lemmas \ref{le1.3} and  \ref{Z:1}, we can also prove that
\begin{align*}
\!\sum_{|\alpha|=1}^{3}\!\! \frac{ D^{\alpha} V (Z_k) }{\alpha!} \big(\tilde{Z}_{k+1}\!-Z_k\big)^{\alpha}
 \leq&\mathcal{L}V(Z_{k})\t\!+C  \Lambda_{\rho}(Z_k)V(Z_k)  \t^{2(1\!-\bar{\theta})}+\!\sum_{i=1}^{3}\mathcal{S}_{i}^{\t}V(Z_k). ~~
\end{align*}
This together with \eqref{yhf*} and \eqref{EM_1} implies
\begin{align*}
V(\tilde{Z}_{k+1})\!\leq&\! V(Z_{k})+\!\mathcal{L}V(Z_{k})\t
\!+C  \Lambda_{\rho}(Z_k)V(Z_k)  \t^{2(1\!-\bar{\theta})} \!+\!\! \sum_{i=1}^{3}\!\mathcal{S}_{i}^{\t}V(Z_k) \! +\! |\tilde{\mathcal{J}}^{\t}V(Z_k)|,
\end{align*}
and by \eqref{AA_13}, one observes
\begin{align*}
|\tilde{\mathcal{J}}^{\t}V(Z_k)|
  =&C \Lambda_{\rho}(Z_k)V(Z_k) \Big[K_4\t^{2(1-\bar{\theta})}+K_{\frac{1}{\delta_4}} \t^{\frac{1-\bar{\theta}}{2\delta_4}} \Big]+\mathcal{A}_1^{\t}V(Z_k)\nn\\
\leq&C\Lambda_{\rho}(Z_k) V(Z_k) \t^{2(1-\bar{\theta})} +\mathcal{A}_1^{\t}V(Z_k),
\end{align*}
where
\begin{align}\la{EM_2}
  \mathcal{A}_1^{\t}V(Z_k)=C\Lambda_{\rho}(Z_k)V(Z_k) \Big[&\t^{ -2\bar{\theta} } \big(|\t B_k|^{4} -K_4\t^{2}\big)\nn\\
 &+ \t^{-\frac{\bar{\theta}}{2\delta_4}}\big(|\t B_k|^{\delta_4} -K_{\frac{1}{\delta_4}}\t^{\frac{1}{2\delta_4}}\big)  \Big]. \tag{A.15}
\end{align}
Using   the property \eqref{E_k} we have
$
\mathbb{E}_k\big[\mathcal{A}_1^{\t}V(Z_k)\big]=0.
$
   Thus, the required assertion \eqref{EM_0} is obtained.
Using the similar techniques in the proof of Lemma \ref{Z:1}, we can also prove that
\begin{align*}
|\mathcal{S}_{1}^{\t}V(Z_k)|^2 =| D^{(1)}V(Z_k)  g(  Z_k)|^2\t+\mathcal{H}_{1,1}^{\t}V(Z_k),
\end{align*}
and for   $i=1, 2, 3$,
\begin{align*}
|\mathcal{S}_{i}^{\t}V(Z_k)|^2 \leq C \Lambda_{\rho}(Z_k)V^2(Z_k) \t^{1-\bar{\theta}}+\mathcal{H}_{i,i}^{\t}V(Z_k),\qquad     \mathbb{E}_k\big[\mathcal{H}_{i,i}^{\t}V(Z_k)\big]=0,
\end{align*}
where $\mathcal{H}_{i,i}^{\t}V(\cdot)$ is defined by \eqref{H_11}, \eqref{H_22} and \eqref{H_33}, respectively,
\begin{align*}
  \mathcal{S}_{1}^{\t}V(Z_k)\mathcal{S}_{j}^{\t}V(Z_k) \geq- C  \Lambda_{\rho}(Z_k)V^2(Z_k) \t^{2(1-\bar{\theta})}+ \mathcal{H}_{1,j}^{\t}V(Z_k),
\end{align*}
and $\mathbb{E}_k\big[\mathcal{H}_{1,j}^{\t}V(Z_k)\big]=0$ for $j=2, 3$, where  $\mathcal{H}_{1,j}^{\t}V(\cdot)$ are defined by \eqref{H_12} and \eqref{H_13}, respectively.
 \qed

{\bf  Proof of Lemma \ref{Le5.4}.}~~ First, for an integer $1/\delta_4\in [4, +\infty)$,
 by \eqref{AA_13} and \eqref{E_k}, we deduce that
\begin{align}
|\tilde{\mathcal{J}}^{\t}V(Z_k)|^2
\leq&C \Lambda_{\rho}^2(Z_k)V^2(Z_k)  \big(\t^{-4\bar{\theta}}|\t B_k|^{8}
 + \t^{-\frac{\bar{\theta}}{\delta_4}}|\t B_k|^{\frac{2}{\delta_4}} \big) \nn\\
\leq&C \Lambda_{\rho}(Z_k)V^2(Y_k)   \t^{4(1-\bar{\theta})}+\mathcal{A}_2^{\t}V(Z_k),\tag{A.16}
  \end{align}
and $
\mathbb{E}_k\big[\mathcal{A}_2^{\t}V(Z_k)\big]=0
$, where
 \begin{align}\la{M2}
\mathcal{A}_2^{\t}V(Z_k)
  =&C\Lambda_{\rho}^2(Z_k)V^2(Z_k)  \Big[\t^{-4\bar{\theta}}\Big(|\t B_k|^{8}-K_8\t^{4}\Big)
\nn\\
 &\qquad \qquad \qquad \qquad
 + \t^{-\frac{\bar{\theta}}{\delta_4}}  \Big( |\t B_k|^{\frac{2}{\delta_4}}-K_{\frac{2}{\delta_4}}\t^{\frac{1}{\delta_4}}\Big) \Big].\tag{A.17}
  \end{align}
One further observes that
\begin{align}
 |\tilde{\mathcal{J}}^{\t}V(Z_k)|^3
 \leq&C \Lambda_{\rho}^3(Z_k)V^3(Z_k)  \big(\t^{-6\bar{\theta}}|\t B_k|^{12}
 + \t^{-\frac{3\bar{\theta}}{2\delta_4}}|\t B_k|^{\frac{3}{\delta_4}} \big) \nn\\
\leq&C\Lambda_{\rho}(Z_k)  V^3(Z_k)  \t^{6(1-\bar{\theta})}+\mathcal{A}_3^{\t}V(Z_k),\tag{A.18}
  \end{align}
and $\mathbb{E}_k\big[\mathcal{A}_3^{\t}V(Z_k)\big]=0$, where
 \begin{align}\la{M3}
\mathcal{A}_3^{\t}V(Z_k)
  =&C\Lambda_{\rho}^3(Z_k)V^3(Z_k)  \Big[\t^{-6\bar{\theta}}\Big(|\t B_k|^{12}-K_{12}\t^{6}\Big) \nn\\
&\qquad \qquad \qquad \qquad + \t^{-\frac{3\bar{\theta}}{2\delta_4}}\Big(|\t B_k|^{\frac{3}{\delta_4}}-K_{\frac{3}{\delta_4}}\t^{\frac{3}{2\delta_4}} \Big)\Big].\tag{A.19}
  \end{align}
On the other hand, by \eqref{H_11}, \eqref{2Y_01} and \eqref{AA_13} as well as  the estimates of $D^{(n)}V(\cdot)$, one observes
\begin{align*}
 \big|\mathcal{S}_{1}^{\t}V(Z_k)\big|^2|\tilde{\mathcal{J}}^{\t}V(Z_k)|
\leq&C\Lambda_{\rho}(Z_k) V(Z_k)   \Big( C\Lambda_{\rho}(Z_k)V^2(Z_{k}) \t^{1-\bar{\theta}}
+\mathcal{H}_{1,1}^{\t}V(Z_k)\Big)\nn\\
&\qquad \qquad \qquad   \times \Big(\t^{-2\bar{\theta}}|\t B_k|^{4}
 + \t^{-\frac{\bar{\theta}}{2\delta_4}}|\t B_k|^{\frac{1}{\delta_4}}  \Big)\nn\\
\leq&C\Lambda_{\rho}(Z_k) V^3(Z_k)  \Big(\t^{1-\bar{\theta}}+\t^{-\bar{\theta}}|\t B_k|^{2}\Big)\nn\\
&\qquad \qquad \qquad   \times\Big(\t^{-2\bar{\theta}}|\t B_k|^{4}
 + \t^{-\frac{\bar{\theta}}{2\delta_4}}|\t B_k|^{\frac{1}{\delta_4}}  \Big)\nn\\
 \leq&C\Lambda_{\rho}(Z_k) V^3(Z_k)
\Big(\t^{1-3\bar{\theta}}|\t B_k|^{4}
 + \t^{1-\frac{\bar{\theta}}{2\delta_4}-\bar{\theta}}|\t B_k|^{\frac{1}{\delta_4}}\nn\\
&\qquad \qquad \qquad \quad +\t^{-3\bar{\theta}}|\t B_k|^{6}
 + \t^{-\frac{\bar{\theta}}{2\delta_4}-\bar{\theta}}|\t B_k|^{2+\frac{1}{\delta_4}}  \Big)\nn\\
\leq&C\Lambda_{\rho}(Z_k) V^3(Z_k)  \t^{3(1-\bar{\theta})}+\mathcal{A}_4^{\t}V(Z_k)
\end{align*}
 for $1/\delta_4\in [4, +\infty)$, where
\begin{align}\la{M4}
 \mathcal{A}_4^{\t}V(Z_k)
=&C \Lambda_{\rho}(Z_k)V^3(Z_k)   \Big[ \t^{-3\bar{\theta}}\Big(|\t B_k|^{6}-K_6 \t^3+\t|\t B_k|^{4}-K_4 \t^3\Big) \nn\\
&\qquad \qquad \qquad \qquad + \t^{-\frac{\bar{\theta}}{2\delta_4}-\bar{\theta}}
\Big(|\t B_k|^{2+\frac{1}{\delta_4}} -K_{2+\frac{1}{\delta_4}}\t^{1+\frac{1}{2\delta_4}}\Big)\nn\\
&\qquad \qquad \qquad \qquad +\t^{1-\frac{\bar{\theta}}{2\delta_4}-\bar{\theta}}
\Big(|\t B_k|^{\frac{1}{\delta_4}} -K_{\frac{1}{\delta_4}}\t^{\frac{1}{2\delta_4}} \Big) \Big].\tag{A.20}
\end{align}
and  it is easy to see that $\mathbb{E}_k\big[\mathcal{A}^{\t}_4  V(Z_{k})\big]=0$.
 Similarly,   we can also prove that
\begin{align*}
 & \big|\mathcal{S}_{2}^{\t}V(Z_k)\big|^2 |\tilde{\mathcal{J}}^{\t}V(Z_k)|\nn\\
 \leq& C \Lambda_{\rho}(Z_k)V^3(Z_k)  \Big( \t^{-2\bar{\theta}}  |\t B_k|^4+ \t^{2(1-\bar{\theta})}\nn\\
 &\qquad \qquad \qquad \qquad \qquad + \t^{2-3\bar{\theta}}|\t B_k|^2 \Big) \Big(\t^{-2\bar{\theta}}|\t B_k|^{4}
 + \t^{-\frac{\bar{\theta}}{2\delta_4}}|\t B_k|^{\frac{1}{\delta_4}}  \Big) \nn\\
  \leq& C \Lambda_{\rho}(Z_k)V^3(Z_k) \t^{4(1-\bar{\theta})}+\mathcal{A}_5^{\t}V(Z_k),
\end{align*}
where
\begin{align}\la{M5}
  \mathcal{A}_5^{\t}V(Z_k)
=&C \Lambda_{\rho}(Z_k)V^3(Z_k) \! \bigg\{\t^{-4\bar{\theta}} \Big(|\t B_k|^{8}\!-K_8 \t^4\Big)\!+ \t^{2-5\bar{\theta}}\Big(|\t B_k|^{6}\!-K_6 \t^3\Big) \nn\\
& + \t^{-(\frac{1}{2\delta_4}+2)\bar{\theta}}
\Big(|\t B_k|^{\frac{1}{\delta_4}+4} -K_{\frac{1}{\delta_4}+4}\t^{\frac{1}{2\delta_4}+2} \Big)\nn\\
&+\t^{2(1-\bar{\theta})} \Big[ \t^{-2\bar{\theta}}\Big(|\t B_k|^{4}-K_4\t^2\Big)
 + \t^{\frac{-\bar{\theta}}{2\delta_4}} \Big(|\t B_k|^{\frac{1}{\delta_4}}-K_{\frac{1}{\delta_4}}\t^{\frac{1}{2\delta_4}}\Big)  \Big]\nn\\
& + \t^{2-\bar{\theta}(\frac{1}{2\delta_4}+3)}
\Big(|\t B_k|^{\frac{1}{\delta_4}+2} -K_{\frac{1}{\delta_4}+2}\t^{\frac{1}{2\delta_4}+1} \Big) \bigg\},\tag{A.21}
\end{align}
and  it is easy to see that $\mathbb{E}_k\big[\mathcal{A}^{\t}_5 V(Z_{k})\big]=0$. Moreover,
\begin{align*}
 & \big|\mathcal{S}_{3}^{\t}V(Z_k)\big|^2 |\tilde{\mathcal{J}}^{\t}V(Z_k)| \nn\\
 \leq& C\Lambda_{\rho}(Z_k) V^3(Z_k)   \Big(  |\t B_k|^2\t^{4-5\bar{\theta}}+\t^{2-4\bar{\theta}} |\t B_k|^4 \nn\\
  &\qquad \qquad \qquad \qquad +\t^{4(1-\bar{\theta})}+ \t^{-3\bar{\theta}}|\t B_k|^6\Big)\Big(\t^{-2\bar{\theta}}|\t B_k|^{4}
 + \t^{\frac{-\bar{\theta}}{2\delta_4}} |\t B_k|^{\frac{1}{\delta_4}}  \Big) \nn\\
   \leq& C \Lambda_{\rho}(Z_k)V^3(Z_k) \t^{5(1-\bar{\theta})}+\mathcal{A}_6^{\t}V(Z_k),
\end{align*}
where
\begin{align}\la{M6}
  \mathcal{A}_6^{\t}V(Z_k)
=&C\Lambda_{\rho}(Z_k) V^3(Z_k)  \bigg\{\t^{-5\bar{\theta}}\Big(|\t B_k|^{10}-K_{10} \t^5\Big)\nn\\
&+ \t^{2-6\bar{\theta}}\Big(|\t B_k|^{8}-K_8 \t^4\Big)+\t^{4-7\bar{\theta}}\Big(|\t B_k|^{6}-K_6 \t^3\Big)\nn\\
&+ \t^{-\bar{\theta}(\frac{1}{2\delta_4}+3)}
\Big(|\t B_k|^{\frac{1}{\delta_4}+6} -K_{\frac{1}{\delta_4}+6}\t^{\frac{1}{2\delta_4}+3} \Big)\nn\\
&+\t^{4(1-\bar{\theta})} \Big[\t^{ -2\bar{\theta}} \Big(|\t B_k|^{4}-K_4\t^2\Big)
 + \t^{\frac{-\bar{\theta}}{2\delta_4}} \Big(|\t B_k|^{\frac{1}{\delta_4}}-K_{\frac{1}{\delta_4}}\t^{\frac{1}{2\delta_4}}\Big)  \Big]\nn\\
&+ \t^{2-\bar{\theta}(\frac{1}{2\delta_4}+4)}
\Big(|\t B_k|^{\delta_4+4} -K_{\frac{1}{\delta_4}+4}\t^{\frac{1}{2\delta_4}+2} \Big)\nn\\
&+ \t^{4-\bar{\theta}(\frac{1}{2\delta_4}+5)}
\Big(|\t B_k|^{\frac{1}{\delta_4}+2} -K_{\frac{1}{\delta_4}+2}\t^{\frac{1}{2\delta_4}+1} \Big) \bigg\},\tag{A.22}
\end{align}
and  it is easy to see that $\mathbb{E}_k\big[\mathcal{A}^{\t}_6 V(Z_{k})\big]=0$.
 \qed

%



\bibliographystyle{amsplain}

\end{document}